\documentclass[12pt]{amsart}

\usepackage{amsthm}
\newtheorem{theorem}{Theorem}[section]
\newtheorem{lemma}[theorem]{Lemma}
\newtheorem{proposition}[theorem]{Proposition}

\theoremstyle{definition}
\newtheorem{definition}[theorem]{Definition}

\newtheorem{example}[theorem]{Example}
\theoremstyle{remark}
\newtheorem{remark}[theorem]{Remark}

\counterwithin*{section}{part}

\usepackage{amssymb}
\usepackage{empheq}
\usepackage{stmaryrd}
\usepackage{enumerate}
\usepackage[shortlabels]{enumitem}
\usepackage{calc}
\usepackage{url}
\usepackage{comment}
\usepackage{mathabx}

\usepackage{xcolor}

\newcommand{\norm}[1]{\Vert#1\Vert}
\usepackage{mathtools}
\DeclarePairedDelimiter{\ceil}{\lceil}{\rceil}
\DeclarePairedDelimiter{\floor}{\lfloor}{\rfloor}
\DeclareMathOperator*{\argmin}{arg\,min}
\DeclareMathOperator*{\dist}{dist}

\usepackage[left=2.0cm,%
                right=2.0cm,%
                top=2.5cm,%
                bottom=3.5cm,%
                headheight=12pt,%
                a4paper]{geometry}%

\begin{document}

\title[Effective rates for quasi-Fej\'er monotone dynamical systems]{Effective rates for continuous-time quasi-Fej\'er monotone dynamical systems}

\author[A. Freund and N. Pischke]{Anton Freund${}^{\MakeLowercase a}$ and Nicholas Pischke${}^{\MakeLowercase b}$}
\date{\today}
\maketitle
\vspace*{-5mm}
\begin{center}
{\scriptsize 
${}^a$ Institute of Mathematics, University of W\"urzburg,\\
Emil-Fischer-Str.~40, 97074 W\"urzburg, Germany,\\
${}^b$ Department of Computer Science, University of Bath,\\
Claverton Down, Bath, BA2 7AY, United Kingdom.\\ 
E-mails: anton.freund@uni-wuerzburg.de, nnp39@bath.ac.uk}
\end{center}

\maketitle 
\begin{abstract}
We provide quantitative convergence results for continuous-time dynamical systems in metric spaces that satisfy a continuous-time analog of quasi-Fej\'er monotonicity. More precisely, we provide a (strong) convergence result for such dynamical systems over compact metric spaces which is quantitatively outfitted with a continuous-time rate of metastability, which moreover can be explicitly and effectively constructed in a very uniform way, only depending on a few moduli representing quantitative witnesses to key properties of the dynamical system and a measure for the compactness of the space. We further show how this convergence result can be extended to non-compact spaces under a regularity assumption of the associated problem, where moreover rates of convergence can then be explicitly constructed which are similarly uniform. In both cases, already the associated ``infinitary'' convergence result is qualitatively novel in its present generality. Beyond this abstract quantitative theory for such dynamical systems, we motivate how the presently studied continuous-time variant of quasi-Fej\'er monotonicity naturally occurs as a unifying property of many dynamical systems and differential equations and inclusions, and in that way can be used to provide a comprehensive quantitative theory for many such dynamical systems. We illustrate this with three case studies for both classical first- and second-order dynamical systems in Hilbert spaces as well as (generalized) gradient flows and associated semigroups in nonlinear Hadamard spaces.  
\end{abstract}

\section{Introduction}

Fej\'er monotonicity is one of the most fundamental  concepts in modern (convex) programming. Traditionally formulated by 
\[
\forall n\in\mathbb{N}\ \forall z\in F\left(d(x_{n+1},z)\leq d(x_n,z)\right)
\] 
for a sequence $(x_n)$, say in some metric space $(X,d)$, it expresses that the sequence as a whole descends towards a solution set $F\subseteq X$. This rather convenient property, or at least one of its many extensions, is enjoyed by many, if not most, iterative methods in optimization, and as such has considerably streamlined and unified the convergence analysis of such dynamical systems, where next to establishing Fej\'er monotonicity, all that is required to prove a method convergent is generally some mild approximation property of the sequence towards the solution set. As such, Fej\'er monotonicity over time became one of the most central organisational principles in the analysis of algorithms.

The notion itself has a rather long history already. The name Fej\'er monotonicity was coined in \cite{MotzkinSchoenberg1954}, referring to a use of a similar concept made by Fej\'er in \cite{Fejer1922}. Subsequently, it has been heavily studied by the ``Russian school'' of convex programming, in particular Eremin \cite{Eremin1968a,Eremin1968b} as well as Ermol'ev \cite{Ermolev1969,Ermolev1971,ErmolevTuniev1973}. The modern appreciation of Fej\'er monotonicity as a central unifying concept for many strands of convex programming is however centrally due to Combettes and Bauschke, in particular due to the surveys and expositions of Combettes \cite{Combettes2001,Combettes2009} and the central monograph \cite{BauschkeCombettes2017}. 

This paper is devoted to a study of a continuous-time analog of Fej\'er monotonicity, as formulated by
\[
\forall t\geq s\geq 0\ \forall z\in F\left(d(x(t),z)\leq d(x(s),z)\right)
\]
and an associated quasi-variant incorporating an error function $e$ with $\lim_{t\geq s\to\infty} e(s,t)=0$, that is
\[
\forall t\geq s\geq 0\ \forall z\in F\left(d(x(t),z)\leq d(x(s),z)+e(s,t)\right),
\]
now for a continuous-time dynamical system $x:[0,\infty)\to X$, still over some metric space $(X,d)$. This property, and its extensions by perturbations of the metric or inclusions of error terms, is intimately linked to differential inequalities for associated energy functionals as arising in the context of Lyapunov analyses of differential equations and inclusions, but in particular allows one to capture many further situations like flows and semigroups, also over nonlinear spaces (as discussed in greater detail in Section \ref{sec:motivation}). Nevertheless, this property is often left implicit, only explicitly coming to the forefront in the proofs, if at all. A notable exception is the well-known survey \cite{PeypouquetSorin2010} by Peypouquet and Sorin, where the property is prominently isolated as a link between continuous- and discrete-time settings (see in particular property (a3) therein). But even in \cite{PeypouquetSorin2010}, the investigation is not centered around the above property but rather constructed ad-hoc around the particular problem at hand, that is the first-order Cauchy problem for a maximal monotone operator. 

As such, the present paper is intended to illustrate how the above, more ``sequential'', notion naturally allows for a uniform study of such dynamical systems arising through differential equations, and beyond. In particular, this more sequential character seems particularly suitable for quantitative considerations, which are the main object of this paper.

Concretely, the first part of the present paper presents highly general and abstract convergence theorems for dynamical systems subscribing to (a generalized version of) the above continuous-time Fej\'er monotonicity property. These results fall into two categories:
\begin{enumerate}
\item We provide a strong convergence result under a (relative) compactness assumption on the space, utilizing only mild approximation assumptions on the dynamical systems.
\item We provide a strong convergence result without any explicit compactness assumption, but while imposing a certain rather general type of regularity for the associated problem. 
\end{enumerate}

In the former case, we in particular provide rates of metastability for the dynamical system~$x$, that is a function $\Delta(\varepsilon,g)$ such that
\[
\forall \varepsilon>0\ \forall g:\mathbb{N}\to\mathbb{N}\ \exists n\leq\Delta(\varepsilon,g)\ \forall t,s\in [n,n+g(n)]\left( d(x(t),x(s))\leq \varepsilon\right).
\]
These rates thereby represent a continuous-time analog of the usual rates of metastability as coined by Tao \cite{Tao2008b,Tao2008a} for discrete-time sequences, and similar to that context present an alternative phrasing of the Cauchy property of a dynamical system which allows for effective quantitative information, contrary to usual rates of convergence which, due to results from computability theory (see \cite{Specker1949} as well as \cite{Neumann2015}) are generally non-computable.

Nevertheless, in the latter case, these rates of metastability can generally be upgraded to full rates of convergence, that is we can even provide a function $\rho(\varepsilon)$ such that
\[
\forall \varepsilon>0\ \forall t,s\geq \rho(\varepsilon)\left( d(x(t),x(s))\leq \varepsilon\right).
\]
In either case, these rates can be explicitly constructed from some simple data representing quantitative formulations of the main assumptions, and we provide these constructions in full detail. In particular, the resulting rates are highly uniform, depending only on these data and not on further auxiliary properties of the space or dynamical system. 

The quantitative results pertaining to the first item in particular generalize related such results for discrete-time Fej\'er monotone sequences presented in \cite{KohlenbachLeusteanNicolae2018}, while the latter generalize similar such results given in \cite{KohlenbachLopezAcedoNicolae2019} (see also \cite{Pischke2023}). In particular, our methods are quite similar, and result in rates of a similar style and complexity, thereby highlighting how one can give a unified treatment of continuous- and discrete-time systems, in particular for quantitative questions, by following the more discrete reasoning to a large degree in the continuous context. We regard this as another methodological contribution of the present paper, and regard our paper in a similar spirit to the survey \cite{PeypouquetSorin2010}, which already heavily highlights such similarities, also in various other parts of the analysis of continuous- and discrete-time systems.

Like the works mentioned above, also the present results are derived using methods from proof mining, a subfield of mathematical logic which aims at utilizing tools from the foundations of mathematics to provide additional computational information for ineffectively proven results in mathematics. As such, it has found wide success in particular in the context of (nonlinear) analysis and optimization. We refer to the seminal monograph \cite{Kohlenbach2008} as well as the surveys \cite{Kohlenbach2019,KohlenbachOliva2003} for further background on proof mining. However, regardless of this logical background, the entire paper is formulated without explicit reference to logic, as typical for results arising in the proof mining program, and in that way can be comprehended without any knowledge of mathematical logic.

All these methodological aspects notwithstanding, we regard the actual applications of the abstract results presented here as the key exponents of our paper. To illustrate this, we provide three concrete case studies in this paper. The first presents a quantitative analysis of the work \cite{BotCsetnek2017} by Bo\c{t} and Csetnek for a first-order dynamical system over a nonexpansive operator and the associated first-order continuous-time variant of the forward-backward scheme. The second presents a quantitative analysis of the work \cite{BotCsetnek2016} by Bo\c{t} and Csetnek for a second-order dynamical system over a cocoercive operator, and its associated second-order continuous-time variant of the forward-backward scheme. Both case studies are chosen due to their rather abstract presentation and in particular their generality, allowing for a uniform treatment of various other preceding systems (as discussed in more detail in the respective introductions to the case studies later on, but we already want to refer to the survey \cite{Csetnek2020}). The second case study in particular presents an example where our approach, even in the context of metric regularity assumptions, cannot provide rates of convergence in general, as we can only provide a rate of metastability for the respective error function in the quasi-Fej\'er monotonicity property, which presents a novel structural reason for why no rates of convergence have appeared for that particular scheme, at that level of generality. This should in particular be compared to more restrictive choices of parameter processes and operators, in particular covering the continuous-time forward-backward method in the context of strongly monotone operators, where fast effective rates can nevertheless be constructed as illustrated in \cite{BotCsetnek2018}, which we partially extend in the present paper to more general settings, again however following a quite different methodology motivated by the discrete-time approach. The last case study is concerned with nonlinear semigroups on nonpositively curved metric spaces, making use of the general metric setting of our results. In particular, we consider the perhaps canonical example of the gradient flow semigroup as studied in this context by Mayer \cite{Mayer1998} and later Ba\v{c}\'ak \cite{Bacak2013}, and then further consider a related nonlinear semigroup generated by a nonexpansive mapping as studied by Stojkovic \cite{Stojkovic2012} and later Ba\v{c}\'ak and Reich \cite{BacakReich2014}, which in particular allows for an abstract approach towards heat flows in singular spaces (as further discussed later on).

At last, we think of our paper as laying the groundwork for quantitative studies of other facets of continuous-time Fej\'er monotonicity, and can conceive of various natural successors to the present work. Indeed, discrete-time Fej\'er monotonicity encompasses many variants, and we think that also these variants can be quantitatively lifted to a continuous-time setting by following the present methods. Initial examples of such variants include adaptations to varying metrics \cite{CombettesVu2013} or to different distance functions like Bregman distances \cite{BauschkeBorweinBorwein2003} or even broader general classes of distances on metric spaces as considered in \cite{Pischke2025} (see also \cite{NeriPischkePowell2025}). However, these extensions would be rather routine. More interestingly, one could also generalize the present results to continuous-time variants of the recently introduced Fej\'er$^*$ monotonicity \cite{ArakcheevBauschke2025b,BehlingBelloCruzIusemLiuSantos2024,BehlingBelloCruzIusemAlvesRibeiroSantos2024} as well as the rather general and extensive localized and relativized generalization of Fej\'er monotonicity \cite{KohlenbachPinto2023} (which was also introduced through a proof mining perspective), where then the very interesting question arises whether natural dynamical systems fall under that regime, with a natural contender being continuous-time variants of the method presented in \cite{BehlingBelloCruzIusemLiuSantos2024}. Last but not least, the present results naturally lend themselves to extensions into the stochastic setting. In particular, the present results on rates of convergence under regularity assumptions could be generalized in that way by combining the present approach with that of the recent works \cite{NeriPischkePowell2025,PischkePowell2026} on rates of convergence for discrete-time stochastically quasi-Fej\'er monotone sequences under stochastic regularity assumptions, while the general constructions of rates of metastability under a compactness assumption could be extended by instead following the approach of the upcoming work \cite{NeriPischkePowell2026}. The previously cited works rely on a quantitative theory for discrete-time supermartingales, and extensions of these results would then require a similar quantitative approach to continuous-time martingale theory, which would be interesting in its own right. In particular, such results could then unlock quantitative treatment of associated stochastic differential equations and inclusions, as to render them applicable to recent works in that direction such as \cite{BotSchindler2024,BotSchindler2025,LiuLongLiHuang2025,LukeSchnebelStaudiglPeypouquetQu2026,MaulenSotoFadiliAttouch2025,MaulenSotoFadiliAttouchOchs2026}.

\subsection*{Notation}

Throughout, if not stated otherwise, we let $(X,d)$ be a given metric space. We denote closed balls in such spaces by
\[
\overline{B}_r(x):=\{y\in X\mid d(x,y)\leq r\}.
\]
Natural numbers $\mathbb{N}$ are throughout assumed to contain $0$. The function $\ceil*{\cdot}$ is defined as usual, that is by $\ceil*{x}:=\min\{z\in\mathbb{Z}\mid z\geq x\}$. Further, we define bounded subtraction on $\mathbb{N}$ via $n\dotdiv m:=\max\{n-m,0\}$. If we are given a function $f:\mathbb{N}\to\mathbb{N}$, we denote iteration of that function by $f^{(i)}$, defined recursively by $f^{(0)}:=\mathrm{Id}$ and $f^{(i+1)}:=f\circ f^{(i)}$.

\section{Continuous-time Fej\'er and quasi-Fej\'er monotonicity}\label{sec:motivation}

\subsection{Definition and motivation}

We begin with our notion of continuous-time quasi-Fej\'er monotonicity, the central notion studied in the present paper. This notion, like all other variants of Fej\'er monotonicity, refers to an underlying solution set which we in this paper take to be given as the set of zeros
\[
\mathrm{zer}F:=\{x\in X\mid F(x)=0\}
\]
for a given solution function $F:X\to[0,+\infty]$.

\begin{definition}[Quasi-Fej\'er monotone dynamical system]\label{def:quasiFejer}
Let $F:X\to[0,+\infty]$ be a given solution function and let $G,H:[0,\infty)\to [0,\infty)$ be functions with $a\to 0$ whenever $H(a)\to 0$ as well as $G(a)\to 0$ whenever $a\to 0$.

A dynamical system $x:[0,+\infty)\to X$ is called $(G,H)$-quasi-Fej\'er monotone w.r.t.\ $F$ if
\[
H(d(x(t),y))\leq G(d(x(s),y))+e(s,t)
\]
for all $t\geq s$ and all $y\in\mathrm{zer}F$, where $e$ is some error function with $\lim_{t\geq s\to\infty}e(s,t)=0$. If $e\equiv 0$, then $x$ is called $(G,H)$-Fej\'er monotone w.r.t.\ $F$.
\end{definition}

Motivated by their use in \cite{KohlenbachLeusteanNicolae2018}, the general perturbation functions $G$ and $H$ above allow for quite a great variety of different notions of Fej\'er monotonicity familiar from the discrete-time world to be encompassed by the above continuous notion. 

As such, beyond the more ``ordinary'' quasi-Fej\'er monotonicity property
\[
d(x(t),y)\leq d(x(s),y)+e(s,t)\text{ for all }t\geq s
\]
derived by $G=H=\mathrm{Id}$, which presents a continuous-time analog of quasi-Fej\'er monotonicity of type I as considered by Combettes \cite{Combettes2001}, they in particular encompass functions such as  $G=H=(\cdot)^p$ for $p\in (0,+\infty)$, where the associated monotonicity property
\[
d^p(x(t),y)\leq d^p(x(s),y)+e(s,t)\text{ for all }t\geq s
\]
in particular presents a continuous-time analog of quasi-Fej\'er monotonicity of order $p$ as discussed in \cite{ErmolevTuniev1973}. In the case of $p=2$, this in particular presents a continuous-time analog of quasi-Fej\'er monotonicity of type II as considered in \cite{Combettes2001}. 

A more or less immediate lift of the discrete-time notion of quasi-Fej\'er monotonicity, the above notion nevertheless abstracts and unifies various scattered formulations and conditions often associated with ``continuous-time quasi-Fej\'er monotonicity''. These, while often distinct in their use and analysis, commonly take the form of differential inequalities
\[
\frac{d}{dt}\mathcal{E}_y(t)\leq \overline{e}(t)
\]
for a given (family of) function(s) $\mathcal{E}_y(t)=\rho(d(x(t),y))$ over $y\in\mathrm{zer}F$, where $x(t)$ is a given dynamical system, $\overline{e}\in L^1$ is an error function and $\rho$ again serves as a perturbation (commonly taking the forms $\rho=\mathrm{Id}$ or $\rho=(\cdot)^2$). Indeed, it is often already the stricter variant 
\[
\frac{d}{dt}\mathcal{E}_y(t)\leq 0,
\]
formulated without an error function, which occurs. The perhaps canonical case for this that we have in mind is the first-order Cauchy problem over a maximally monotone operator in Hilbert spaces (already due to \cite{Brezis1973,CrandallPazy1969,Komura1967}, see also \cite{BrezisPazy1970,CrandallLiggett1971} and e.g.\ the survey \cite{PeypouquetSorin2010}), that is
\[
\begin{cases}
-\dot{x}(t)\in A(x(t)),\\
x(0)=x_0,
\end{cases}
\]
where $A$ is a maximally monotone operator on a Hilbert space $X$ and $x_0\in\mathrm{dom}(A)$, in particular generalizing the gradient flow equation $\dot{x}=\nabla f(x)$ to general inclusions. Now, any solution $x:[0,\infty)\to X$ to the above problem (that is $x$ is an absolutely continuous function satisfying $x(0)=x_0$ and $-\dot{x}(t)\in A(x(t))$ almost everywhere; in fact, there is a unique such solution, see e.g.\ \cite{PeypouquetSorin2010}) satisfies the differential inequality
\[
\frac{d}{dt}\left(\frac{1}{2}\norm{x(t)-y}^2\right)=\langle \dot{x}(t),x(t)-y\rangle\leq 0
\]
almost everywhere, where $y\in\mathrm{zer}(A)$ is any zero of $A$, i.e. $0\in A(y)$.  Beyond this rather immediate example, we further refer to \cite{AttouchCzarnecki2010,BanertBot2018,BotCsetnek2017,BotCsetnekLaszlo2020,BotCsetnekVuong2020,CsetnekMalitskyTam2019} for other instances of differential inequalities such as the above, just to name a few.

In any case, in the context of the above differential inequality, integration of that property yields
\[
\rho(d(x(t),y))\leq \rho(d(x(s),y))+\int_s^t \bar{e}(\tau)\,d\tau\text{ for all }t\geq s
\]
for all $y\in \mathrm{zer}F$, so that the above family of differential inequalities results in the dynamical system $x$ being $(\rho,\rho)$-quasi-Fej\'er monotone in the sense of Definition \ref{def:quasiFejer}, with errors $e(s,t):=\int_s^t \overline{e}(\tau)\,d\tau$. In the special case of $\overline{e}\equiv 0$ already mentioned above, we in particular get that $x$ is $(\rho,\rho)$-Fej\'er monotone.

Such a presentation of the errors via an integral, that is
\[
e(s,t)=\int_s^t \overline{e}(\tau)\,d\tau
\]
for $\overline{e}\in L^1$, is in a way the standard form of an error that we want to accommodate in the above Definition \ref{def:quasiFejer}. In this error form, the $(G,H)$-quasi-Fej\'er monotonicity of a dynamical system $x$ can be even more directly recognised as a direct lift of the usual discrete-time $(G,H)$-quasi-Fej\'er monotonicity of a sequence $(x_n)$ (see e.g.\ Definition 6.2 in \cite{KohlenbachLeusteanNicolae2018}), that is
\[
H(d(x_{n+m},y))\leq G(d(x_n,y))+\sum_{i=n}^{n+m-1}\varepsilon_i\text{ for all }n\geq m\text{ and }y\in\mathrm{zer}F,
\]
where $(\varepsilon_i)$ is a nonnegative sequence of summable errors, to the continuous setting, with errors taking the form of  integrals instead of series.

However, similar as with the perturbation functions, the generality of the generic error function $e$ compared to a specific representation such as the above allows to encompass even more general situations. Namely, beyond such rather simple situations as detailed above, many analyses of continuous-time dynamical systems do not rely on such simple functions $\mathcal{E}_y(t)=\rho(d(x(t),y))$ as above but instead consider much more extensive (families of) so-called energy functionals $\mathcal{E}_y(t)$, suitably chosen to allow for the analysis of the underlying dynamical system $x(t)$ at hand, often in a rather ad-hoc way (see e.g.\ the survey \cite{MoucerTaylorBach2023}). Regardless, such energy functions often decompose in special ways.

One such decomposition takes the form of
\[
\mathcal{E}_y(t)=\rho(d(x(t),y))+\mathcal{E}'(t).
\]
Indeed, such energy functions and associated differential inequalities are rather common, and we refer to e.g.\ \cite{AbbasAttouch2015,AbbasAttouchSvaiter2014,Antipin1994,AttouchAlvarez2000,BotCsetnek2016,BotCsetnek2016b,BotCsetnek2017b,BotCsetnek2019,BotCsetnekLaszlo2018}, just to name a few. Integration of the associated differential inequality property yields
\[
\mathcal{E}_y(t)-\mathcal{E}_y(s)\leq \int_s^t \bar{e}(\tau)\,d\tau\text{ for all }t\geq s
\]
as before and, in the case of the above decomposition $\mathcal{E}_y(t)=\rho(d(x(t),y))+\mathcal{E}'(t)$, this inequality immediately results in the monotonicity property
\[
\rho(d(x(t),y))\leq \rho(d(x(s),y))+\mathcal{E}'(s)-\mathcal{E}'(t)+\int_s^t \bar{e}(\tau)\,d\tau
\]
for all $t\geq s$ and $y\in\mathrm{zer}F$, so that in case the associated remainder $\mathcal{E}'$ satisfies $\mathcal{E}'(t)\to 0$ for $t\to\infty$, the dynamical system $x(t)$ is $(\rho,\rho)$-quasi-Fej\'er monotone in the above sense with a generalized error function
\[
e(s,t):=\mathcal{E}'(s)-\mathcal{E}'(t)+\int_s^t \bar{e}(\tau)\,d\tau.
\]

Another type of decomposition that sometimes occurs in linear spaces takes the form of
\[
\mathcal{E}_y(t)=\rho(\norm{x(t)-y+v(t)})
\]
for some additional function $v(t)$ with $\lim_{t\to\infty}v(t)= 0$. Such energy functions e.g.\ occur in~\cite{AlecsaLaszloPinta2021,AttouchLaszlo2021,AttouchSvaiter2011}. At least in situations where $\norm{x(t)-y}$, $\norm{x(t)-y+v(t)}$ are bounded (as can generally be established without great difficulty in such situations) and $\rho$ is Lipschitz on bounded sets (as is the case for e.g.\ $\rho=(\cdot)^2$), say with constant $L>0$ for a respective set containing $\norm{x(t)-y}$ and $\norm{x(t)-y+v(t)}$, we can use the associated differential inequality as follows: Integration as before yields
\[
\rho(\norm{x(t)-y+v(t)})\leq \rho(\norm{x(s)-y+v(s)})+\int_s^t \overline{e}(\tau)\,d\tau.
\]
Now, using the Lipschitz continuity of $\rho$ on the respective (bounded) set yields
\begin{align*}
\rho(\norm{x(t)-y})&\leq \rho(\norm{x(t)-y+v(t)})+L\norm{v(t)}\\
&\leq \rho(\norm{x(s)-y+v(s)})+\int_s^t \overline{e}(\tau)\,d\tau+L\norm{v(t)}\\
&\leq \rho(\norm{x(s)-y})+\int_s^t \overline{e}(\tau)\,d\tau+L(\norm{v(t)}+\norm{v(s)}).
\end{align*}
In that way, the dynamical system $x(t)$ then is $(\rho,\rho)$-quasi-Fej\'er monotone in the above sense with a generalized error function
\[
e(s,t):=\int_s^t \overline{e}(\tau)\,d\tau+L(\norm{v(t)}+\norm{v(s)}).
\]

Naturally, choosing such an energy function appropriately is quite difficult and lies at the heart of a convergence analysis of various continuous time methods. We in that way do not at all want to argue that the present results make such an approach obsolete but rather that once such a suitable function with a corresponding differential inequality has been isolated, it allows for a uniform quantitative analysis of the associated convergence results via the above notion of quasi-Fej\'er monotonicity for continuous-time dynamical systems and it is that quantitative perspective, as already discussed in the introduction, that justifies its use.

\subsection{Quantitative results for differential inequalities}

We have seen above how quasi-Fej\'er monotonicity for continuous-time dynamical systems often comes in the guise of differential inequalites. These inequalities have various key asymptotic results attached to them that are fundamental for most analyses in the literature. As such, we now discuss these results here, and in particular provide quantitative variants, on which we later rely in the context of our applications.

The main results with which we are concerned are the following fundamental lemmas (see e.g.\ Lemmas 5.1 and 5.2 in \cite{AbbasAttouchSvaiter2014}):

\begin{lemma}
Assume $\mathcal{E}:[0,\infty)\to\mathbb R$ is locally absolutely continuous and bounded below such that
\[
\frac{d}{dt}\mathcal{E}(t)\leq \overline{e}(t)
\]
almost everywhere for some $\overline{e}\in L^1([0,\infty))$. Then $\lim_{t\to\infty}\mathcal{E}(t)$ exists.

If further $\mathcal{E}$ is nonnegative and $\mathcal{E}\in L^p([0,\infty))$ for $p\in [1,+\infty)$, then $\lim_{t\to\infty}\mathcal{E}(t)=0$. In that case, we may also have $\overline{e}\in L^r([0,\infty))$ for $r\in [1,+\infty]$.
\end{lemma}

We begin with an analysis of the first part of the above lemma:

\begin{lemma}\label{lem:AAS1}
Assume $\mathcal{E}:[0,\infty)\to\mathbb R$ is locally absolutely continuous and bounded below with $\mathcal{E}(t)\geq b$ for all $t\in[0,\infty)$ and $c\geq\mathcal{E}(0)$. Further assume we have 
\[
\frac{d}{dt}\mathcal{E}(t)\leq \overline{e}(t)
\]
almost everywhere for some $\overline{e}\in L^1([0,\infty))$ with $\norm{\overline{e}}_1\leq B$. Then for all $\varepsilon>0$ and $f:\mathbb{N}\to\mathbb{N}$, it holds that
\begin{equation*}
\exists n\leq \tilde{f}^{(\omega_{b,c,B}(\varepsilon))}(0)\ \forall s,t\in [n,n+f(n)]\left( \vert \mathcal{E}(s)-\mathcal{E}(t)\vert\leq\varepsilon\right),
\end{equation*}
where $\tilde{f}(n):=n+f(n)$ and 
\[
\omega_{b,c,B}(\varepsilon):=\left\lceil\frac{2B}{\varepsilon}\right\rceil\cdot\left\lceil\frac{c+B-b}{\varepsilon}\right\rceil.
\]
\end{lemma}
\begin{proof}
By the proof of \cite[Lemma~5.1]{AbbasAttouchSvaiter2014}, the map~$t\mapsto \mathcal{E}(t)-\int_0^t\overline{e}(\tau)\,d\tau$ is non-increasing with
\begin{equation*}
D(t):=\mathcal{E}(t)-\int_0^t\overline{e}(\tau)\,d\tau\in[b-B,c]\text{ for all }t\in[0,\infty).
\end{equation*}
We define $H(m)=\tilde{f}^{\left(\left\lceil\frac{2B}{\varepsilon}\right\rceil\cdot m\right)}(0)$. It is clearly impossible for~$D$ to descend by more than~$\varepsilon/2$ in all intervals~$[H(m),H(m+1)]$ with~$m<\left\lceil 2(c+B-b)/\varepsilon\right\rceil$. So for one of the indicated~$m$, we have
    \begin{equation*}
        0\leq D(s)-D(t)\leq\frac{\varepsilon}{2}\quad\text{whenever }H(m)\leq s\leq t\leq H(m+1).
    \end{equation*}
    On the other hand, the integral over~$|\overline{e}|$ cannot exceed~$\frac{\varepsilon}{2}$ on all intervals~$[\tilde f^{(i)}(H(m)),\tilde f^{(i+1)}(H(m))]$ with $i<\left\lceil\frac{2B}{\varepsilon}\right\rceil$. For an~$i$ where the integral is bounded by~$\frac{\varepsilon}{2}$, we put
    \begin{equation*}
        r=\left\lceil\frac{2B}{\varepsilon}\right\rceil\cdot m+i\quad\text{and}\quad n=\tilde h^{(r)}(0).
    \end{equation*}
    Observe that we have
    \begin{equation*}
        [n,n+f(n)]=\left[\tilde f^{(r)}(0),\tilde f^{(r+1)}(0)\right]=\left[\tilde f^{(i)}(H(m)),\tilde f^{(i+1)}(H(m))\right]\subseteq[H(m),H(m+1)].
    \end{equation*}
    For $n\leq s\leq t\leq n+f(n)$, we thus get
    \begin{equation*}
        |\mathcal{E}(s)-\mathcal{E}(t)|\leq D(s)-D(t)+\int_s^t|G(\tau)|\,d\tau\leq\frac{\varepsilon}{2}+\frac{\varepsilon}{2}=\varepsilon,
    \end{equation*}
    as desired.
\end{proof}

We continue with an analysis of~\cite[Lemma~5.2]{AbbasAttouchSvaiter2014}:

\begin{lemma}\label{lem:AAS2}
Assume $\mathcal{E}\in L^p([0,\infty))$ is locally absolutely continuous and non-negative, with $1\leq p<\infty$, and where $c\geq\mathcal{E}(0)$. Further assume we have 
\[
\frac{d}{dt}\mathcal{E}(t)\leq \overline{e}(t)
\]
almost everywhere, for some $\overline{e}\in L^r([0,\infty))$ with $1\leq r\leq\infty$. Let $\norm{\mathcal{E}}_p\leq A$ and $\norm{\overline{e}}_r\leq B$. Then for all $\varepsilon>0$ and $f:\mathbb{N}\to\mathbb{N}$, it holds that
\begin{equation*}
\exists n\leq \tilde{f'}^{(\varpi_{c,A,B}(\varepsilon))}(0)\ \forall t\in[n,n+f(n)]\left( \mathcal{E}(t)<\varepsilon\right),
\end{equation*}
where $\tilde{f'}(n):=n+f'(n)$ as before, with $f'(n):=\max\{f(n),\ceil*{(3A/\varepsilon)^p}\}$, and
\begin{equation*}
\varpi_{c,A,B}(\varepsilon):=\left\lceil\frac{2^{q+1}qA^{q-1}B}{\varepsilon^q(2^q-1)}\right\rceil\cdot\left\lceil\frac{2^q(c^q+qA^{q-1}B)}{\varepsilon^q(2^q-1)}\right\rceil\quad\text{with}\quad q=1+p\cdot\left(1-\frac{1}{r}\right).
\end{equation*}
\end{lemma}
\begin{proof}
Let $\varepsilon>0$ and $f:\mathbb{N}\to\mathbb{N}$ be given. We first show that
\begin{equation*}
\exists n\leq \tilde{f}^{(\varpi_{c,A,B}(\varepsilon))}(0)\ \forall t\in[n,n+f(n)]\left( \mathcal{E}(t)<\varepsilon\right)
\end{equation*}
for $\tilde{f}(n):=n+f(n)$, if $f(n)> (2A/\varepsilon)^p$ for all $n\in\mathbb{N}$.
    Since~$r'=\frac{p}{q-1}$ is conjugate to~$r$ and $\mathcal{E}^{q-1}$ is in~$L^{r'}([0,\infty))$, H\"older's inequality yields
    \begin{equation*}
        \norm{\mathcal{E}^{q-1}\cdot \overline{e}}_1\leq A^{q-1}\cdot B.
    \end{equation*}
    Except on a null set, we have
    \begin{equation*}
        \frac{d}{dt}\mathcal{E}^q(t)=q\cdot \mathcal{E}^{q-1}(t)\cdot\frac{d}{dt}\mathcal{E}(t)\leq q\cdot \mathcal{E}^{q-1}(t)\cdot \overline{e}(t).
    \end{equation*}
    As $f(n)> (2A/\varepsilon)^p$ for all $n\in\mathbb{N}$, Lemma \ref{lem:AAS1} shows that
    \begin{equation*}
        \left|\mathcal{E}^q(s)-\mathcal{E}^q(t)\right|\leq\frac{\varepsilon^q(2^q-1)}{2^q}\quad\text{for all }s,t\in[n,n+f(n)]
    \end{equation*}
for some~$n\leq \tilde{f}^{(\varpi_{c,A,B}(\varepsilon))}(0)$ as in the current lemma. Now there is some~$s_0\in[n,n+f(n)]$ with $\mathcal{E}(s_0)<\varepsilon/2$, since we would otherwise get
    \begin{equation*}
        A^p\geq\int_n^{n+f(n)} \mathcal{E}^p(t)\,dt\geq\frac{\varepsilon^p}{2^p}\cdot f(n)>A^p.
    \end{equation*}
    We now show that $\mathcal{E}(t)<\varepsilon$ holds for any~$t\in[n,n+f(n)]$. If this was not true, we would get
    \begin{equation*}
        |\mathcal{E}^q(s_0)-\mathcal{E}^q(t)|>\varepsilon^q-\frac{\varepsilon^q}{2^q}=\frac{\varepsilon^q(2^q-1)}{2^q},
    \end{equation*}
    against the above. If $f:\mathbb{N}\to\mathbb{N}$ is now a general function, define $f'(n):=\max\{f(n),\ceil*{(3A/\varepsilon)^p}\}$. By definition $f'(n)>(2A/\varepsilon)^p$ and so the above implies 
\[
\exists n\leq \tilde{f'}^{(\varpi_{c,A,B}(\varepsilon))}(0)\ \forall t\in[n,n+f'(n)]\left( \mathcal{E}(t)<\varepsilon\right).
\]
The result now follows as $f(n)\leq f'(n)$ and so $[n,n+f(n)]\subseteq [n,n+f'(n)]$.
\end{proof}

\begin{remark}
In the above Lemmas \ref{lem:AAS1} and \ref{lem:AAS2}, we have always quantitatively resolved integrability assumptions via upper bounds on the involved $p$-norms, e.g.\ assuming for a given $\overline{e}\in L^1([0,\infty))$ a bound $B\geq \norm{\overline{e}}_1$. It is clear from the proofs above that it actually suffices to have rates of metastability for the integrals in question, e.g.\ in the case of $\overline{e}\in L^1([0,\infty))$ above taking the form of a function $\xi(\varepsilon,f)$ such that for all $\varepsilon>0$ and $f:\mathbb{N}\to\mathbb{N}$, one has
\[
\exists n\leq \xi(\varepsilon,f)\left(\int_{n}^{n+f(n)}\vert\overline{e}\vert<\varepsilon\right).
\]
But then the results become combinatorially more involved. As one is commonly supplied with such a norm upper bound, we keep the results at that generality.
\end{remark}

\section{A quantitative convergence theory for quasi-Fej\'er monotone dynamical systems}

We now turn to our main abstract results, presenting a quantitative theory of convergence for quasi-Fej\'er monotone dynamical systems in continuous time. As discussed before, we will be working over general metric spaces. Convergence results in these rather general contexts commonly fall in two categories:
\begin{enumerate}
\item Strong (metric) convergence results under a (local) compactness assumption, for general problems $\mathrm{zer}F$.
\item Strong (metric) convergence results in the absence of compactness, but for that under an additional (local) regularity assumption on the problem $\mathrm{zer}F$.
\end{enumerate}
Weak convergence results, which are quite characteristic for Fej\'er monotonicity, in general are not available for general metric spaces and require additional structure, commonly working over Hilbert spaces or Hadamard spaces (complete geodesic spaces of nonpositive curvature). In that way, these results shall not be our concern here, although the present work can be easily adapted also to these cases.

In either case, the only other property next to the quasi-Fej\'er monotonicity that is required to induce convergence is the rather weak approximation property that the sequence in question contains arbitrarily good approximate solutions along its run.

We here consider quantitative variants of both items (1) and (2) above, extending the previous seminal works in that regard \cite{KohlenbachLeusteanNicolae2018} and \cite{KohlenbachLopezAcedoNicolae2019} (see also \cite{Pischke2023}) from discrete time to continuous time. We begin with the former.

\subsection{Convergence under relative compactness assumptions}

We recall the basic setup. Let $(X,d)$ be a given metric space and let $F:X\to [0,+\infty]$ be a given function, so that $\mathrm{zer}F:=\{x\in X\mid F(x)=0\}$ serves as an abstract representation of a solution set and so that each sublevel set
\[
\mathrm{lev}_{\leq\varepsilon} F:=\{x\in X\mid F(x)\leq\varepsilon\},
\]
with $\varepsilon>0$, serves as an abstract representation of a set of $\varepsilon$-approximate solutions.\\

We will be concerned with a quantitative convergence result for quasi-Fej\'er monotone dynamical systems in the context where $X$ is totally bounded. To arrive at such a result, we require three central ingredients.

The first is, as mentioned before, a (quantitative) approximation property for the sequence in question that suffices to guarantee convergence.

\begin{definition}
Let $x:[0,\infty)\to X$ be a dynamical system. We say that $x$ has the $\liminf$-property w.r.t.\ $F$ if $\liminf_{t\to\infty} F(x(t))=0$. Quantitatively, we say that $\varphi:(0,\infty)\times\mathbb{N}\to \mathbb{N}$ is a $\liminf$-bound for $x$ w.r.t.\ $F$ if
\[
\forall \varepsilon>0\ \forall n\in\mathbb{N} \exists t\in [n,\varphi(\varepsilon,n)]\left( F(x(t))<\varepsilon\right).
\]
\end{definition}

\begin{remark}
We will later commonly assume that $\varphi$ is monotone in the sense that $\varphi(\varepsilon,n)\geq \varphi(\delta,n)$ whenever $\varepsilon\leq\delta$. This can actually be assumed without loss of generality as if a general $\varphi(\varepsilon,n)$ is given, we can define
\[
\widehat{\varphi}(\varepsilon,n):=\max\{\varphi(1/(k+1),n)\mid k\leq \ceil*{1/\varepsilon}\}.
\]
Indeed, then $\widehat{\varphi}$ is still a $\liminf$-bound as given $\varepsilon>0$ and $n\in\mathbb{N}$, we have
\[
\exists t\in [n,\varphi(1/(\ceil*{1/\varepsilon}+1),n)]\left( F(x(t))<\frac{1}{\ceil*{1/\varepsilon}+1}\leq \varepsilon\right)
\]
and $\varphi(1/(\ceil*{1/\varepsilon}+1),n)\leq \widehat{\varphi}(\varepsilon,n)$, and if $\varepsilon\leq\delta$, then $\ceil*{1/\varepsilon}\geq\ceil*{1/\delta}$ and so
\[
\widehat{\varphi}(\varepsilon,n)=\max\{\varphi(1/(k+1),n)\mid k\leq \ceil*{1/\varepsilon}\}\geq \max\{\varphi(1/(k+1),n)\mid k\leq \ceil*{1/\delta}\}=\widehat{\varphi}(\delta,n).
\]
Hence $\widehat{\varphi}$ is also monotone in the above sense.
\end{remark}

The second ingredient we will need is a uniform quantitative strengthening of the quasi-Fej\'er property. For that, consider the property of $(G,H)$-quasi-Fej\'er monotonicity unwound into its logical structure:
\[
\forall y\in X\left( \forall \delta>0 (y\in\mathrm{lev}_{\leq\delta}F)\to \forall \varepsilon>0\ \forall s\leq t (H(d(x(t),y))\leq G(d(x(s),y))+e(s,t)+\varepsilon)\right).
\]
With the quantifiers suitably prenexed (and rewriting $\forall s\leq t$ into $\forall n,m\in\mathbb{N}\ \forall s\leq t\in [n,n+m]$) the above is equivalent to
\begin{gather*}
\forall \varepsilon>0\ \forall n,m\in\mathbb{N}\ \forall y\in X\ \exists \delta>0(  y\in\mathrm{lev}_{\leq\delta}F\\
\to  \forall s\leq t\in [n,n+m]\left(H(d(x(t),y))\leq G(d(x(s),y))+e(s,t)+\varepsilon\right)).
\end{gather*}
We now call a sequence uniformly $(G,H)$-quasi-Fej\'er monotone if we can bound (and hence witness) the above $\delta$ in terms of $\varepsilon,n,m$, uniformly for all $y\in X$:

\begin{definition}
Let $x:[0,\infty)\to X$ be a dynamical system. We say that $x$ is uniformly $(G,H)$-quasi-Fej\'er monotone w.r.t.\ $F$ with modulus $\chi:(0,\infty)\times\mathbb{N}^2\to (0,\infty)$ if
\begin{gather*}
\forall \varepsilon>0\ \forall n,m\in\mathbb{N}\ \forall y\in \mathrm{lev}_{\leq\chi(\varepsilon,n,m)}F\ \forall s\leq t\in [n,n+m]\\
\left(H(d(x(t),y))\leq G(d(x(s),y))+e(s,t)+\varepsilon\right)
\end{gather*}
for some error function $e$ with $\lim_{t\geq s\to\infty}e(s,t)=0$. If $e\equiv 0$, then we call $x$ uniformly $(G,H)$-Fej\'er monotone w.r.t.\ $F$.
\end{definition}

We have two remarks on this notion:

\begin{remark}\label{rem:unifFejerComp}
A standard compactness argument shows that if $X$ is compact and $F$ as well as $G,H$ are continuous, then any continuous dynamical system $x$ which is $(G,H)$-quasi-Fej\'er monotone w.r.t.\ $F$ and some continuous error function $e$ is also already uniformly $(G,H)$-quasi-Fej\'er monotone w.r.t.\ $F$ and that same error function.
\end{remark}

\begin{remark}[For logicians]
As in the case of the associated notion of uniform $(G,H)$-quasi-Fej\'er monotonicity for discrete-time sequences given in \cite{KohlenbachLeusteanNicolae2018} (see Remark 4.7 therein), the two notions above, plain and uniform quasi-Fej\'er monotonicity for continuous-time sequences, can be proven to be equivalent in theories for bounded metric spaces used in proof mining (see \cite{Kohlenbach2005}) augmented with a nonstandard uniform boundedness principle (see \cite{Kohlenbach2006}). In particular, associated logical metatheorems for these system due to Kohlenbach (recall the previous references and see further \cite{Kohlenbach2008}) in that way guarantee that such a modulus $\chi$ can be extracted from large classes of proofs of the associated, weaker, quasi-Fej\'er monotonicity property, over bounded metric spaces, making it particularly natural for applications.
\end{remark}

The last central ingredient will be a quantitative rendering of the compactness assumption on the space. Concretely, in the following we will work under a total boundedness assumption, which we quantitatively resolve following the work of \cite{Gerhardy2008}:

\begin{definition}[essentially \cite{Gerhardy2008}]\label{def:mod-boundedness}
Let $A\subseteq X$ be nonempty. A function $\gamma:(0,\infty)\to\mathbb{N}$ is a modulus of total boundedness for $A$ if for any $\varepsilon>0$ and any sequence $(x_k)\subseteq A$:
\[
\exists 0\leq i<j\leq\gamma(\varepsilon)\left( d(x_i,x_j)\leq \varepsilon\right).
\]
\end{definition}

It follows rather immediately that a set is totally bounded if, and only if, it has a modulus of total boundedness. We refer to \cite{Gerhardy2008} and \cite{KohlenbachLeusteanNicolae2018} for further discussions of this property. 

Besides these three central ingredients, we are left with quantitatively resolving the minor assumptions around our sequence. 

Let us first consider the perturbation functions $G,H:[0,\infty)\to [0,\infty)$. By assumption, we have $a\to 0$ whenever $H(a)\to 0$, and this is equivalent with assuming the existence of a quantitative modulus $h:(0,\infty)\to (0,\infty)$ such that
\[
H(a)<h(\varepsilon)\to a<\varepsilon
\]
for all $a\geq 0$ and $\varepsilon>0$. Similarly, the property that $G(a)\to 0$ whenever $a\to 0$ is resolved by considering as modulus $g$ with
\[
a<g(\varepsilon)\to G(a)<\varepsilon
\]
for all $a\geq 0$ and $\varepsilon>0$. 

The last minor quantitative detail we need to resolve is the error property $\lim_{t\geq s\to\infty} e(s,t)=0$. Most conveniently, we would resolve this assumption with a rate of convergence (similar to how errors for quasi-Fej\'er monotone sequences are resolved in \cite{KohlenbachLeusteanNicolae2018}), that is a function $\eta(\varepsilon)$ such that
\[
\forall \varepsilon>0\ \forall t\geq s\geq \eta(\varepsilon)\left( \vert e(s,t)\vert\leq\varepsilon\right).
\]
This assumption is particularly mild in cases where the errors are effectively composed of parameters of a problem which can be freely chosen, so that such a rate can be readily supplied, as is often the case.

However, this will crucially not be the case for some applications later where we will only be able to derive a rate of metastability for the above limit, that is a function $\eta(\varepsilon,f)$ such that
\[
\forall \varepsilon>0\ \forall f:\mathbb{N}\to \mathbb{N}\ \exists n\leq \eta(\varepsilon,f)\ \forall s\leq t\in [n,n+f(n)]\left( \vert e(s,t)\vert\leq\varepsilon\right).
\]
This generality however comes with the price of some added combinatorial complexity of the final rate.

Combining all of these ingredients, we then get the following central quantitative result on the convergence of quasi-Fej\'er monotone dynamical systems in totally bounded spaces.

\begin{theorem}\label{thm:generalFejer}
Let $X$ be totally bounded with a modulus $\gamma$. Assume that $G,H:[0,\infty)\to [0,\infty)$ are such that we have moduli $g,h$ with
\[
H(a)<h(\varepsilon)\to a<\varepsilon \text{ and }a<g(\varepsilon)\to G(a)<\varepsilon
\]
for all $a\geq 0$ and $\varepsilon>0$. Also, let $e(s,t)\to 0$ as $s\leq t\to \infty$ with a rate of metastability $\eta$, i.e.\
\[
\forall \varepsilon>0\ \forall f:\mathbb{N}\to\mathbb{N}\ \exists n\leq \eta(\varepsilon,f)\ \forall s\leq t\in [n,n+f(n)]\left(\vert e(s,t)\vert <\varepsilon\right).
\]
Lastly, let $x:[0,\infty)\to X$ be a dynamical system such that
\begin{enumerate}
\item $x$ has the $\liminf$-property w.r.t.\ $F$ with bound $\varphi:(0,\infty)\times\mathbb{N}\to \mathbb{N}$, i.e.\
\[
\forall \varepsilon>0\ \forall n\in\mathbb{N}\ \exists t\in [n,\varphi(\varepsilon,n)]\left( F(x(t))<\varepsilon\right).
\]
Further, we assume (w.l.o.g.) that $\varphi$ is monotone in the sense that $\varphi(\varepsilon,n)\geq \varphi(\delta,n)$ whenever $\varepsilon\leq\delta$.
\item $x$ is uniformly $(G,H)$-quasi-Fej\'er monotone w.r.t.\ $F$ with error $e$ and with modulus $\chi:(0,\infty)\times\mathbb{N}^2\to (0,\infty)$, i.e.\
\begin{gather*}
\forall \varepsilon>0\ \forall n,m\in\mathbb{N}\ \forall y\in \mathrm{lev}_{\leq\chi(\varepsilon,n,m)}F\ \forall s\leq t\in [n,n+m]\\
\left(H(d(x(t),y))\leq G(d(x(s),y))+e(s,t)+\varepsilon\right).
\end{gather*}
\end{enumerate}
Then $x$ is metastable in the sense that for all $\varepsilon>0$ and $f:\mathbb{N}\to\mathbb{N}$, it holds that 
\[
\exists n\leq \Delta(\varepsilon,f)\ \forall s,t\in [n,n+f(n)]\left( d(x(s),x(t))\leq\varepsilon\right),
\]
where moreover the bound $\Delta(\varepsilon,f)$ can be explicitly described by
\[
\Delta(\varepsilon,f):=\max\{\Delta(i,\varepsilon,f)\mid i\leq P\}+1
\]
where $P:=\gamma(g(h(\varepsilon/2)/3))+1$ and $\Delta(0,\varepsilon,f):=0$ as well as
\[
\Delta(j,\varepsilon,f):=\max\{\varphi(\widehat{\varepsilon}_j,n)\mid n\leq \eta(h(\varepsilon/2)/3,f_{\varphi,\widehat{\varepsilon}_j})\}
\]
with
\[
\widehat{\varepsilon}_j=\min\{\chi^M_f(h(\varepsilon/2)/3,\Delta(i,\varepsilon,f))\mid i<j\}
\]
for $j\geq 1$, where
\begin{gather*}
\chi_f(\varepsilon,n):=\chi(\varepsilon,n,f(n+1)+1),\\
\chi^M_f(\varepsilon,n):=\min\{\chi_f(\varepsilon,m)\mid m\leq n\},
\end{gather*}
as well as
\[
f_{\varphi,\varepsilon}(n):=\max\{m+1+f(m+1)\mid m\leq \varphi(\varepsilon,n)\}\dotdiv n.
\]
\end{theorem}
\begin{proof}
Let $\varphi_0$ be a selection function for the $\liminf$-property which is bounded by $\varphi$, i.e.\ take $\varphi_0:(0,\infty)\times\mathbb{N}\to [0,\infty)$ to be any function such that $x(\varphi_0(\varepsilon,n))\in \mathrm{lev}_{\leq\varepsilon} F$ and $\varphi_0(\varepsilon,n)\in [n,\varphi(\varepsilon,n)]$ for all $\varepsilon>0$ and $n\in\mathbb{N}$. Set $t_0:=0$ and $n_0:=0$. Then, given $f:\mathbb{N}\to\mathbb{N}$ and $\varepsilon>0$, recursively define
\[
\varepsilon_j:=\min\{\chi(g(\varepsilon/2)/3,\floor*{t_i},f(\floor*{t_i}+1)+1)\mid i<j\},
\]
let $n_j\leq \eta(h(\varepsilon/2)/3,f_{\varphi,\widehat{\varepsilon}_j})$ be such that
\[
\forall s\leq t\in [n_j,n_j+f_{\varphi,\widehat{\varepsilon}_j}(n_j)]\left(\vert e(s,t)\vert \leq h(\varepsilon/2)/3\right),
\]
and set $t_j:=\varphi_0(\varepsilon_j,n_j)$. We first note that $\varepsilon_j\geq\widehat{\varepsilon}_j$ for $j>0$ as well as $\Delta(i,\varepsilon,f)\geq t_i$ for $i\geq 0$, which we show by induction. Clearly $\Delta(0,\varepsilon,f)=0= t_0$ and if $\Delta(i,\varepsilon,f)\geq t_i$ for all $i<j$, then
\begin{align*}
\widehat{\varepsilon}_j&=\min\{\chi^M_f(h(\varepsilon/2)/3,\Delta(i,\varepsilon,f))\mid i<j\}\\
&\leq \min\{\chi_f(h(\varepsilon/2)/3,\floor*{t_i})\mid i<j\}=\varepsilon_j
\end{align*}
as $\floor*{t_i}\leq t_i\leq \Delta(i,\varepsilon,f)$. Then we get
\begin{align*}
t_j=\varphi_0(\varepsilon_j,n_j)\leq \varphi(\varepsilon_j,n_j)&\leq \varphi(\widehat{\varepsilon}_j,n_j)\leq \max\{\varphi(\widehat{\varepsilon}_j,n)\mid n\leq \eta(h(\varepsilon/2)/3,f_{\varphi,\widehat{\varepsilon}_j})\}=\Delta(j,\varepsilon,f)
\end{align*}
as $n_j\leq \eta(h(\varepsilon/2)/3,f_{\varphi,\widehat{\varepsilon}_j})$.

Now, using the modulus of total boundedness $\gamma$ for the sequence $(x(t_j))$, we get indices
\[
0<I<J\leq\gamma(g(h(\varepsilon/2)/3))+1
\]
such that $d(x(t_I),x(t_J))\leq g(h(\varepsilon/2)/3)$. By definition, as $J>0$ and $J>I$, we have 
\[
x(t_J)\in \mathrm{lev}_{\leq \varepsilon_J}F\subseteq \mathrm{lev}_{\leq \chi(h(\varepsilon/2)/3,\floor*{t_I},f(\floor*{t_I}+1)+1)}F.
\]
Also, as $I>0$, we have $t_I\geq n_I$. The former now implies
\[
H(d(x(t),x(t_J)))\leq G(d(x(s),x(t_J)))+e(s,t)+h(\varepsilon/2)/3
\]
for all $s\leq t\in [\floor*{t_I},\floor*{t_I}+1+f(\floor*{t_I}+1)]$. At the same time, we have
\[
\forall s\leq t\in [n_I,n_I+f_{\varphi,\widehat{\varepsilon}_I}(n_I)]\left( \vert e(s,t)\vert \leq h(\varepsilon/2)/3\right).
\]
Now, we have $n_I\leq \floor*{t_I}$ as $n_I\leq t_I$ and $n_I$ is a natural number, and so
\begin{align*}
n_I\leq \floor*{t_I}\leq t_I\leq \floor*{t_I}+1&\leq \floor*{t_I}+1+f(\floor*{t_I}+1)\\ &\leq \max\{m+1+f(m+1)\mid m\leq\varphi(\widehat{\varepsilon}_I,n_I)\}\\ &\leq n_I+f_{\varphi,\widehat{\varepsilon}_I}(n_I)
\end{align*}
as $t_I=\varphi_0(\varepsilon_I,n_I)\leq \varphi(\varepsilon_I,n_I)\leq \varphi(\widehat{\varepsilon}_I,n_I)$ since $\widehat{\varepsilon}_I\leq \varepsilon_I$. Thus, in the above, we may pick $s=t_I$ to get
\[
H(d(x(t),x(t_J)))\leq G(d(x(t_I),x(t_J)))+e(t_I,t)+h(\varepsilon/2)/3\text{ and }\vert e(t_I,t)\vert\leq h(\varepsilon/2)/3
\]
for all $t\in [\floor*{t_I}+1,\floor*{t_I}+1+f(\floor*{t_I}+1)]$ (as then also $t\geq t_I$). As $d(x(t_I),x(t_J))\leq g(h(\varepsilon/2)/3)$, we have $G(d(x(t_I),x(t_J)))\leq h(\varepsilon/2)/3$ and so
\[
H(d(x(t),x(t_J)))\leq \vert e(t_I,t)\vert+2h(\varepsilon/2)/3\leq h(\varepsilon/2)
\]
for all such $t$. This implies $d(x(t),x(t_J))\leq \varepsilon/2$ for all $t\in [n,n+f(n)]$ with $n=\floor*{t_I}+1$, so that we obtain the result by the triangle inequality.

As seen above, we in particular have $\Delta(i,\varepsilon,f)\geq\floor*{t_i}$ so that $\Delta(\varepsilon,f)\geq \Delta(I,\varepsilon,f)+1\geq \floor*{t_i}+1=n$ as $I\leq P$.
\end{proof}

The above result extends Theorem 6.4 in \cite{KohlenbachLeusteanNicolae2018}, the analogous result for discrete-time sequences, and in fact closely follows the proof strategy given there.

If $F$ is uniformly continuous, then the above result can be further upgraded to also effectively guarantee that the dynamical system consists of approximate solutions.\footnote{In fact, Theorem \ref{thm:generalFejer-uc} can be established under the milder quantitative condition that $\mathrm{zer}F$ is uniformly closed w.r.t.\ the level sets $\mathrm{lev}_{\leq\varepsilon}F$ in the sense of \cite{KohlenbachLeusteanNicolae2018}, that is using moduli $\omega$ and $\delta$ such that for any $x,y\in X$, $d(x,y)\leq\omega(\varepsilon)$ and $F(x)\leq\delta(\varepsilon)$ implies $F(y)\leq\varepsilon$. For simplicity, and since it suffices for all our applications, we here restrict ourselves to the stronger assumption that $F$ is uniformly continuous.}

\begin{theorem}\label{thm:generalFejer-uc}
Under the assumptions of Theorem \ref{thm:generalFejer}, assume that $F$ is uniformly continuous with a modulus $\omega$, that is $\vert F(x)-F(y)\vert\leq\varepsilon$ for all $x,y\in X$ with $d(x,y)\leq\omega(\varepsilon)$.

Then for all $\varepsilon>0$ and $f:\mathbb{N}\to\mathbb{N}$, it holds that 
\[
\exists n\leq \Delta(\min\{\varepsilon,\omega(\varepsilon/2)\},f)\ \forall s,t\in [n,n+f(n)]\left( d(x(s),x(t))\leq\varepsilon\text{ and }F(x(t))\leq\varepsilon\right),
\]
where moreover the bound $\Delta(\varepsilon,f)$ can be explicitly described as in Theorem \ref{thm:generalFejer}, except using $\widehat{\chi}_\varepsilon(\delta,n,m):=\min\{\varepsilon/2,\chi(\delta,n,m)\}$ in place of $\chi$.
\end{theorem}
\begin{proof}
Note that $\widehat{\chi}_\varepsilon$ is also a modulus of uniform $(G,H)$-quasi-Fej\'er monotonicity. Now follow the proof of Theorem \ref{thm:generalFejer} with the respective bound. Writing $\varepsilon_0:=\min\{\varepsilon,\omega(\varepsilon/2)\}$, as therein we then derive $d(x(t),x(t_J))\leq \varepsilon_0/2$ for all $t\in [n,n+f(n)]$ with $n=\floor*{t_I}+1$, where 
\[
x(t_J)\in \mathrm{lev}_{\leq \widehat{\chi}_\varepsilon(h(\varepsilon_0/2)/3,\floor*{t_I},f(\floor*{t_I}+1)+1)}F\subseteq\mathrm{lev}_{\leq\varepsilon/2}F.
\]
The former gives $d(x(t),x(s))\leq \varepsilon$ for all $s,t\in [n,n+f(n)]$ as before, and additionally that $d(x(t),x(t_J))\leq \varepsilon_0/2\leq \omega(\varepsilon/2)$. We hence have $\vert F(x(t))-F(x(t_J))\vert\leq \varepsilon/2$. Now, the latter yields $F(x(t_J))\leq \varepsilon/2$, so that we have 
\[
F(x(t))\leq F(x(t_J))+\vert F(x(t))-F(x(t_J))\vert\leq\varepsilon
\]
for all $t\in [n,n+f(n)]$.
\end{proof}

If we are supplied with a rate of convergence for the error function $e$, or the errors disappear altogether, then the above construction becomes considerably simpler, which we record in the following theorem.

\begin{theorem}\label{thm:generalFejer-rates}
Under the assumptions of Theorem \ref{thm:generalFejer}, assume that $e(s,t)\to 0$ as $s\leq t\to \infty$ with a rate of convergence $\eta$, i.e.\
\[
\forall \varepsilon>0\ \forall t\geq s\geq \eta(\varepsilon)\left(\vert e(s,t)\vert <\varepsilon\right).
\]
Then $x$ is metastable in the sense that for all $\varepsilon>0$ and $f:\mathbb{N}\to\mathbb{N}$, it holds that 
\[
\exists n\leq \Delta(\varepsilon,f)\ \forall s,t\in [n,n+f(n)]\left( d(x(s),x(t))\leq\varepsilon\right),
\]
where moreover the bound $\Delta(\varepsilon,f)$ can be explicitly described by
\[
\Delta(\varepsilon,f):=\max\{\Delta(i,\varepsilon,f)\mid i\leq P\}+1
\]
with $P:=\gamma(g(h(\varepsilon/2)/3))+1$ and $\Delta(0,\varepsilon,f):=0$ as well as
\[
\Delta(j,\varepsilon,f):=\max\{\varphi(\widehat{\varepsilon}_j,n)\mid n\leq \eta(h(\varepsilon/2)/3)\}
\]
where $\widehat{\varepsilon}_j$ for $j\geq 1$ is as in Theorem \ref{thm:generalFejer}.

If $e\equiv 0$, then it suffices that $x$ has the approximate $F$-point property (instead of the $\liminf$-property w.r.t.\ $F$), that is
\[
\forall \varepsilon>0\ \exists t\leq \varphi(\varepsilon)\left( F(x(t))<\varepsilon\right)
\]
with a monotone bound $\varphi:(0,\infty)\to \mathbb{N}$, where $\Delta$ can be defined similar to above, with $\Delta(0,\varepsilon,f):=0$ and $\Delta(j,\varepsilon,f):=\varphi(\widehat{\varepsilon}_j)$.

Lastly, if in any of the above cases $F$ is uniformly continuous with a modulus $\omega$, then for all $\varepsilon>0$ and $f:\mathbb{N}\to\mathbb{N}$, it holds that 
\[
\exists n\leq \Delta(\min\{\varepsilon,\omega(\varepsilon/2)\},f)\ \forall s,t\in [n,n+f(n)]\left( d(x(s),x(t))\leq\varepsilon\text{ and }F(x(t))\leq\varepsilon\right),
\]
where the bound $\Delta(\varepsilon,f)$ is constructed as above, except using $\widehat{\chi}_\varepsilon(\delta,n,m):=\min\{\varepsilon/2,\chi(\delta,n,m)\}$ in place of $\chi$.
\end{theorem}
\begin{proof}
As $\eta$ is a rate of convergence, it follows immediately that $\widehat{\eta}(\varepsilon,f):=\eta(\varepsilon)$ is also a rate of metastability. If the errors are constantly $0$, then $\eta\equiv 0$ is a corresponding rate of convergence. In any case, we can apply Theorem~\ref{thm:generalFejer} to derive the above rates. That, in the case of vanishing errors, an approximate $F$-point bound suffices immediately follows from the proof of Theorem~\ref{thm:generalFejer}. The extensions for approximate solutions in the context of a uniformly continuous $F$ can be accounted for as in the proof of Theorem \ref{thm:generalFejer-uc}.
\end{proof}

Moreover, the above quantitative result trivially implies the following ``usual'' (that is, nonquantitative) convergence result for quasi-Fej\'er monotone dynamical systems in compact spaces.

\begin{theorem}\label{thm:qualThm}
Let $X$ be compact and let $F:X\to [0,+\infty]$ be continuous. Assume that $G,H:[0,\infty)\to [0,\infty)$ are continuous such that $a\to 0$ whenever $H(a)\to 0$ and $G(a)\to 0$ whenever $a\to 0$. Let $x:[0,\infty)\to X$ be a continuous dynamical system such that $x$ has the $\liminf$-property w.r.t.\ $F$ and which is $(G,H)$-quasi-Fej\'er monotone w.r.t.\ $F$ and a continuous error function.

Then $x$ converges to some point in $\mathrm{zer}F$.
\end{theorem}
\begin{proof}
All the present assumptions trivially (but nonconstructively) entail their quantitative variant. In particular, as $X$ is compact, $x$ is uniformly $(G,H)$-quasi-Fej\'er monotone w.r.t.\ $F$ (recall Remark \ref{rem:unifFejerComp}). Hence, Theorem \ref{thm:generalFejer} guarantees that $x$ is metastable, i.e.
\[
\forall \varepsilon>0\ \forall f:\mathbb{N}\to\mathbb{N}\ \exists n\in\mathbb{N}\ \forall s,t\in [n,n+f(n)]\left( d(x(s),x(t))\leq\varepsilon\right).
\]
Trivially, but noneffectively, this is equivalent to $x(t)$ being Cauchy for $t\to\infty$. Indeed, suppose for a contradiction that there is an $\varepsilon>0$ such that for all $n\in\mathbb{N}$, there is an $m\in\mathbb{N}$ with
\[
\exists s,t\in [n,n+m]\left( d(x(s),x(t))>\varepsilon\right).
\]
Define $f(n):=m$ for such an $m$. Then we have found an $f$ such that 
\[
\forall n\in\mathbb{N}\ \exists s,t\in [n,n+f(n)]\left( d(x(s),x(t))>\varepsilon\right).
\]
This contradicts the above property. As $x$ is Cauchy and $X$ is compact, and hence complete, it has a limit $z$. We now show that $z\in\mathrm{zer}F$. For that, note that $F(z)=\lim_{t\to\infty} F(x(t))$ as $F$ is continuous and since $x$ has the $\liminf$-property w.r.t.\ $F$, we get $\liminf_{t\to\infty} F(x(t))=0$. Hence $F(z)=0$.
\end{proof}

\begin{remark}\label{rem:bounded}
Assume that $G$ is inverse bounded, i.e.\
\[
G(a)\leq b\to a\leq g'(b)
\]
for all $a\geq 0$ and $b>0$ and additionally that 
\[
e^M:=\sup_{t\geq 0}e(0,t)<\infty.
\]
Then if $x$ is $(G,H)$-quasi-Fej\'er monotone, we have
\[
G(d(x(t),z))\leq H(d(x(0),z))+e(0,t)
\]
for all $t\geq 0$ and so $G(d(x(t),z))$ and hence $x(t)$ is bounded with $d(x(t),z)\leq g'(b_0+e^M)$ for $b_0\geq H(d(x(0),z))$. A particular case where $G$ is inverse bounded is for example when $G$ is strictly increasing and continuous, where then $g'$ is simply the inverse of $G$. The assumption that $\sup_{t\geq 0}e(0,t)<\infty$ is classically satisfied when $e$ is of standard form, i.e.\ $e(s,t)=\int_s^t\overline{e}(\tau)\,d\tau$ for $\overline{e}\in L^1$, but further easily true in many other situations, such as those later discussed in this paper. In any case, in the context where the dynamical system is bounded, it of course thereby suffices that $X$ is only proper in the above Theorem \ref{thm:qualThm}.
\end{remark}

\begin{remark}\label{rem:rateRemark}
If one has asymptotic regularity w.r.t.\ $F$, that is $\lim_{t\to\infty}F(x(t))=0$, and for simplicity assuming that $e\equiv 0$, then the above qualitative theorem can be directly derived from the associated result for discrete-time dynamical systems (see e.g.\ Proposition 4.3 in \cite{KohlenbachLeusteanNicolae2018}): Indeed, $x(t)$ converges to some point $z\in \mathrm{zer}F$ for $t\to\infty$ if, and only if, $x(t_n)$ converges to $z$ for any (increasing) sequence $(t_n)$ with $t_n\to\infty$ for $n\to\infty$. Using the asymptotic regularity w.r.t.\ $F$, we get that any such sequence has approximate $F$-points and $(G,H)$-Fej\'er monotonicity w.r.t.\ $F$ for $x$ transfers onto the respective sequence $x(t_n)$. The convergence of $x(t_n)$ now follows from the associated discrete-time result (recall Proposition 4.3 in \cite{KohlenbachLeusteanNicolae2018}). This extends suitably to error sequences, using the $\liminf$-property w.r.t.\ $F$ instead.

Exactly this usual reduction to the discrete-time setting under asymptotic regularity seems to be the reason why a qualitative theory for continuous-time Fej\'er monotone dynamical systems has received relatively little attention. However, this reduction proves to be problematic quantitatively, as one cannot immediately transfer rates of metastability for the resulting Cauchy property from the discrete-time subsequences to the full continuous-time sequence (as the respective point of metastability might move out further and further, and moreover, the resulting regions might not overlap).

Instead, the only quantitative approach known to us is to essentially redo the quantitative approach to discrete-time dynamical systems given in \cite{KohlenbachLeusteanNicolae2018} for the continuous setting, and the proof for the above result essentially implements this argument, with various optimisations built in. However, while for metastability there seems to be no efficient way of bootstrapping the quantitative theory of continuous-time dynamical systems onto discrete-time dynamical systems, at least not in the generality of Theorem \ref{thm:generalFejer}, we will later discuss how results on rates of convergence under regularity assumptions can indeed be transferred from the discrete to the continuous, at least in special cases.
\end{remark}

\subsection{Convergence under metric regularity assumptions}

We will now be concerned with (quantitative) convergence results in the absence of compactness. Naturally, such results in turn require additional restrictions on the problem formulation given by $F$.

Such ``regularity'' conditions take various forms in the literature, and to cover these in a uniform manner that is moreover amenable to our quantitative pursuits, we follow the approach of the fundamental work of Kohlenbach, Lop\'ez-Acedo and Nicolae \cite{KohlenbachLopezAcedoNicolae2019} and consider the rather broad notion of a modulus of regularity introduced therein.

\begin{definition}[essentially \cite{KohlenbachLopezAcedoNicolae2019}]\label{def:mod-reg}
Let $z\in\mathrm{zer}F$ and $r>0$. A modulus of regularity for $F$ w.r.t.\ $\overline{B}_r(z)$ is a function $\tau:(0,\infty)\to (0,\infty)$ such that 
\[
F(x)<\tau(\varepsilon)\to \mathrm{dist}(x,\mathrm{zer}F)<\varepsilon
\]
for all $\varepsilon>0$ and all $x\in\overline{B}_r(z)$.
\end{definition}

Note that such a function $\tau$ always exists when $X$ is proper and $F$ is continuous (see Proposition 3.3 in \cite{KohlenbachLopezAcedoNicolae2019}), although the description of $\tau$ is generally noneffective in this case.

Indeed, this notion is quite broad and, for various types of problems, covers a variety of different notions usually considered in the literature to induce (strong) convergence of Fej\'er monotone sequences, with simple moduli. We give a brief overview in the following examples as it pertains to our present work:

\begin{example}[\cite{KohlenbachLopezAcedoNicolae2019}]\label{ex:mappings}
Let $(X,d)$ be a metric space and $T:X\to X$ be a given mapping. We are interested in the problem of finding a fixed point $Tx=x$ of $T$. This problem can be equivalently expressed as a zero problem by setting $F(x):=d(x,Tx)$, where we then have $\mathrm{zer}F=\mathrm{Fix}T$, which we assume to be non-empty.

Situations under which we then can obtain a particularly simple and explicit moduli of regularity for the above $F$ in particular include the following:
\begin{enumerate}
\item Let $T$ be a quasi-contraction, i.e.\ for some $z\in\mathrm{Fix}T$ and $c\in [0,1)$ it holds that 
\[
d(Tx,z)\leq cd(x,z)
\]
for all $x\in X$. Then $\tau(\varepsilon):=(1-c)\varepsilon$ is a modulus of regularity (and in fact of uniqueness) for $F$ w.r.t.\ $\overline{B}_r(z)$, for any $r>0$. 
\item Let $T$ be a continuous orbital contraction, i.e.\ $T$ is continuous and there is a $c\in [0,1)$ such that
\[
d(Tx,T^2x)\leq cd(x,Tx)
\]
for all $x\in X$. Then $\tau(\varepsilon):=(1-c)\varepsilon$ is a modulus of regularity for $F$ w.r.t.\ $\overline{B}_r(z)$, for any $r>0$ and $z\in\mathrm{Fix}T$. In particular, there are examples of orbital contractions that do not have unique fixed points (see \cite{KohlenbachLopezAcedoNicolae2019}).
\item Let $T:X\to C\subseteq X$ be a retraction. Then $\tau(\varepsilon)=\varepsilon$ is a modulus of regularity for $F$ w.r.t.\ $\overline{B}_r(z)$, for any $r>0$ and $z\in\mathrm{Fix}T$. 
\end{enumerate}
\end{example}

\begin{example}[\cite{KohlenbachLopezAcedoNicolae2019}]\label{ex:operators}
Let $(X,\norm{\cdot})$ be a Banach space and $A:X\to 2^X$ be a given set-valued operator with 
\[
\mathrm{dist}(0,Ax)=0\to x\in\mathrm{zer}A
\]
for all $x\in X$ (such as when $A$ is closed, which is e.g.\ the case if $A$ is m-accretive). We are interested in the problem of solving the inclusion $0\in Ax$ for $A$. Due to the above assumption, this problem can be equivalently expressed as a zero problem by setting $F(x):=\mathrm{dist}(0,Ax)$, where we then have $\mathrm{zer}F=\mathrm{zer}A$, which we assume to be non-empty.

Situations under which we then can obtain a modulus of regularity for the above $F$ in particular include the following:
\begin{enumerate}
\item Let $A$ be a $\tau$-strongly accretive operator, i.e.\
\[
\langle x-y,u-v\rangle_+\geq\tau(\norm{x-y})\norm{x-y}
\]
for all $(x,u),(y,v)\in A$, where $\tau$ is strictly increasing and $\tau(0)=0$, and with
\[
\langle x,y\rangle_+:=\max\{j(y)\mid j\in J(x)\}
\]
where $J$ is the normalized duality mapping of $X$ (see e.g.\ \cite{Takahashi2000}). Then $\tau$ is a modulus of regularity for $F$ w.r.t.\ $\overline{B}_r(z)$, for any $r>0$ and $z\in\mathrm{zer}A$. 

This in particular includes the case of $\beta$-strongly accretive operators, i.e.\ operators $A$ where
\[
\langle x-y,u-v\rangle_+\geq\beta\norm{x-y}^2
\]
for all $(x,u),(y,v)\in A$ for some $\beta>0$, in which case the modulus takes the form of $\tau(\varepsilon):=\beta\varepsilon$.
\item Let $A$ be metrically subregular at $z\in\mathrm{zer}A$, i.e.\ there are $k,r>0$ such that
\[
\mathrm{dist}(x,\mathrm{zer}A)\leq k\mathrm{dist}(0,A(x))
\]
for all $x\in\overline{B}_r(z)$. Then $\tau(\varepsilon)=\varepsilon/k$ is a modulus of regularity for $F$ w.r.t.\ $\overline{B}_r(z)$.
\end{enumerate}
\end{example}

\begin{example}[\cite{KohlenbachLopezAcedoNicolae2019}]\label{ex:functions}
Let $(X,d)$ be a metric space and $\phi:X\to (-\infty,+\infty]$ be a given function. We are interested in the problem of finding a minimizer of $\phi$. Assuming that $\mu:=\inf_{x\in X}\phi(x)$ exists, this problem can be equivalently expressed as a zero problem by setting $F(x):=\phi(x)-\mu$ where we then have $\mathrm{zer}F=\argmin \phi$, which we assume to be non-empty.

Situations under which we then can obtain a modulus of regularity for the above $F$ in particular include functions $\phi$ that have $\tau$-global weak sharp minima, i.e.\ where it holds that
\[
\phi(x)\geq m+\tau(\mathrm{dist}(x,\argmin \phi))
\]
for all $x\in X$, where $\tau:[0,\infty)\to [0,\infty)$ is strictly increasing with $\tau(0)=0$. Then $\tau$ is also a modulus of regularity for $F$ w.r.t.\ $\overline{B}_r(z)$, for any $r>0$ and $z\in \argmin \phi$. This notion of weak sharp minima is intimately related to the notion of error bounds (see e.g.\ the discussion in \cite{KohlenbachLopezAcedoNicolae2019}).

A particular situation where a function $\phi$ has weak sharp minima is if it is also uniformly quasiconvex. For this, let $X$ be a uniquely geodesic metric space and, given $x,y\in X$ and $\lambda\in [0,1]$, write $\lambda x\oplus (1-\lambda)y$ for the unique point $z$ on the geodesic connecting $x$ and $y$ that satisfies $d(x,z)=\lambda d(x,y)$ and $d(y,z)=(1-\lambda)d(x,y)$. Assuming that $\phi$ satisfies 
\[
\phi(\lambda x\oplus(1-\lambda)y)\leq\max\{\phi(x),\phi(y)\}-\lambda(1-\lambda)\psi(d(x,y))
\]
for all $x,y\in X$ and $\lambda\in [0,1]$ for a $\tau$ as above, one derives 
\[
\mu\leq \phi\left(\frac{1}{2}x\oplus \frac{1}{2}z\right)\leq\max\{\phi(x),\phi(z)\}-\frac{1}{4}\psi(d(x,z))\leq \phi(x)-\frac{1}{4}\tau(d(x,z))
\]
for $z\in \argmin  \phi$ and $x\in X$, which implies that $\phi$ has $\frac{1}{4}\tau$-global weak sharp minimas. In the case of a strongly quasiconvex function, i.e.\ where
\[
\phi(\lambda x\oplus (1-\lambda)y)\leq\max\{\phi(x),\phi(y)\}-\lambda(1-\lambda)\frac{\rho}{2}d(x,y)^2,
\]
we simply have $\tau(\varepsilon):=\frac{\rho}{2}\varepsilon^2$.
\end{example}

Beyond these examples, the general formulation via a function $F$ further accommodates various other problems (see \cite{KohlenbachLopezAcedoNicolae2019}). In particular, as also illustrated by the above examples, the above moduli of regularity generalize quantitative renderings of uniqueness of solutions, so-called moduli of uniqueness, to situations with multiple solutions.

Next to the regularity of the problem and the quasi-Fej\'er monotonicity of the dynamical system, the only further property needed to induce convergence is the $\liminf$-property, as before. In that context, a modulus of regularity carries further benefits than just removing an otherwise necessary compactness assumption: it in principle allows for the construction of rates of convergence, not just rates of metastability, for the dynamical system.

However, we say ``in principle'' here as such a construction is only possible if all involved data also already come equipped with strong enough quantitative data. In the general case considered in Theorem \ref{thm:generalFejer} above, where the convergence of the error function $e$ was resolved \emph{only} with a rate of metastability, we will not be able to derive a rate of convergence, just a rate of metastability, even under the assumption of a modulus of regularity. We nevertheless begin with this case since, as mentioned before, for some applications later we will only be able to derive a rate of metastability for the errors. 

In that way, we arrive at the following first quantitative convergence result.

\begin{theorem}\label{thm:met-reg}
Let $x:[0,\infty)\to X$ be a dynamical system and assume that $F$ has a modulus of regularity $\tau$ w.r.t.\ $\overline{B}_b(z)$, where $z\in\mathrm{zer}F$ and $b\geq d(x(t),z)$ for all $t\geq 0$. Assume that $G,H:[0,\infty)\to [0,\infty)$ are such that we have moduli $g,h$ with
\[
a<g(\varepsilon)\to G(a)<\varepsilon\text{ and }H(a)<h(\varepsilon)\to a<\varepsilon
\]
for all $a\geq 0$ and $\varepsilon>0$. 
Assume that $x$ has the $\liminf$-property w.r.t.\ $F$ with bound $\varphi:(0,\infty)\times\mathbb{N}\to \mathbb{N}$, i.e.\
\[
\forall \varepsilon>0\ \forall n\in\mathbb{N} \exists t\in [n,\varphi(\varepsilon,n)]\left( F(x(t))<\varepsilon\right),
\]
and further that $x$ is $(G,H)$-quasi-Fej\'er monotone w.r.t.\ $F$, i.e.
\[
\forall y\in\mathrm{zer}F\ \forall t\geq s\left( H(d(x(t),y))\leq G(d(x(s),y))+e(s,t)\right).
\]
Lastly, let $e(s,t)\to 0$ as $s\leq t\to \infty$ with a rate of metastability $\eta$, i.e.\
\[
\forall \varepsilon>0\ \forall f:\mathbb{N}\to\mathbb{N}\ \exists n\leq \eta(\varepsilon,f)\ \forall s\leq t\in [n,n+f(n)]\left(\vert e(s,t)\vert <\varepsilon\right).
\]

Then $x$ satisfies
\[
\forall \varepsilon>0\ \forall f:\mathbb{N}\to\mathbb{N}\ \exists n\leq \rho(\varepsilon,f)\ \forall t\in [n,n+f(n)]\left( \mathrm{dist}(x(t),\mathrm{zer}F)<\varepsilon\right)
\]
and further $x$ is Cauchy with
\[
\forall \varepsilon>0\ \forall f:\mathbb{N}\to\mathbb{N}\ \exists n\leq \rho(\varepsilon/2,f)\ \forall s,t\in [n,n+f(n)]\left( d(x(t),x(s))<\varepsilon\right).
\]

Here, $\rho$ is defined by
\[
\rho(\varepsilon,f):=\max\{\varphi(\tau(g(h(\varepsilon)/2)),n)\mid n\leq \eta(h(\varepsilon)/2,f_{\varphi,h(\varepsilon)})\}+1
\]
where 
\[
f_{\varphi,\delta}(n)=\max\{m+1+f(m+1)\mid m\leq\varphi(\tau(g(h(\varepsilon)/2)),n)\}\dotdiv n.
\]
\end{theorem}
\begin{proof}
Let $n_0\leq \eta(h(\varepsilon)/2,f_{\varphi,h(\varepsilon)})$ be such that
\[
\forall s\leq t\in [n_0,n_0+f_{\varphi,h(\varepsilon)}(n_0)]\left( \vert e(s,t)\vert < h(\varepsilon)/2\right).
\]
Take now $t_0\in [n_0,\varphi(\tau(f(h(\varepsilon)/2)),n_0)]$ such that $F(x(t_0))<\tau(g(h(\varepsilon)/2))$. Define $n=\floor*{t_0}$ and note that $n_0\leq n$ as $n_0\leq t_0$ and $n_0$ is a natural number. Further, we have $n\leq t_0\leq \varphi(\tau(g(h(\varepsilon)/2)),n_0)$ and so
\[
n+1+f(n+1)\leq n_0+f_{\varphi,h(\varepsilon)}(n_0).
\]
Thus, we have $\vert e(s,t)\vert < h(\varepsilon)/2$ for all 
\[
s\leq t\in [n,n+1+f(n+1)]\subseteq [n_0,n_0+f_{\varphi,h(\varepsilon)}(n_0)].
\]
Using that $\tau$ is a modulus of regularity over $\overline{B}_b(z)$, which contains $x$, we have $\mathrm{dist}(x(t_0),\mathrm{zer}F)<g(h(\varepsilon)/2)$. Thus, there exists a $y\in\mathrm{zer}F$ such that $d(x(t_0),y)<g(h(\varepsilon)/2)$ and so we have $G(d(x(t_0),y))<h(\varepsilon)/2$. Using the quasi-Fej\'er monotonicity, we have
\[
H(d(x(t),y))\leq G(d(x(t_0),y))+e(t_0,t)< h(\varepsilon)/2+\vert e(t_0,t)\vert 
\]
for all $t\geq t_0\geq \floor*{t_0}=n$. Thus, we get $H(d(x(t),y))<h(\varepsilon)$ and so
\[
\mathrm{dist}(x(t),\mathrm{zer}F)\leq d(x(t),y)<\varepsilon
\]
for all $t\in [n+1,n+1+f(n+1)]$. This yields the first claim. For the second, note simply that choosing $\varepsilon/2$ instead of $\varepsilon$ in the above then yields $d(x(t),y)<\varepsilon/2$ and so
\[
d(x(t),x(s))\leq d(x(t),y)+d(x(s),y)<\varepsilon
\]
for all $t,s\in [n+1,n+1+f(n+1)]$.
\end{proof}

The above result extends (parts of) Theorem 4.1 in \cite{KohlenbachLopezAcedoNicolae2019} as well as Theorem 4.1 in \cite{Pischke2023}, the analogous results for discrete-time sequences without and with errors, respectively, and in fact closely follows the proof strategy given there.

The following result now improves the conclusion of the above to a full rate of convergence, provided $e$ already comes equipped with one.

\begin{theorem}\label{thm:met-reg-rates}
Under the assumptions of Theorem \ref{thm:met-reg}, assume that $e(s,t)\to 0$ as $s\leq t\to \infty$ with a rate of convergence $\eta$, i.e.\
\[
\forall \varepsilon>0\ \forall t\geq s\geq \eta(\varepsilon)\left(\vert e(s,t)\vert <\varepsilon\right).
\]
Then $x$ satisfies
\[
\forall \varepsilon>0\ \forall t\geq \rho(\varepsilon) \left( \mathrm{dist}(x(t),\mathrm{zer}F)<\varepsilon\right)
\]
and further $x$ is Cauchy with
\[
\forall \varepsilon>0\ \forall s,t\geq \rho(\varepsilon/2)\left( d(x(t),x(s))<\varepsilon\right),
\]
where $\rho$ is now defined by $\rho(\varepsilon):=\varphi(\tau(g(h(\varepsilon)/2)),\eta(h(\varepsilon)/2))+1$.

If $e\equiv 0$, then it suffices that $x$ has the approximate $F$-point property with bound $\varphi:(0,\infty)\to \mathbb{N}$, i.e.
\[
\forall \varepsilon>0\ \exists t\leq \varphi(\varepsilon)\left( F(x(t))<\varepsilon\right),
\]
where then $\rho(\varepsilon):=\varphi(\tau(g(h(\varepsilon))))+1$ suffices for the above claims.

Lastly, if $X$ is complete and $\mathrm{zer}F$ is closed, then in each of the above cases, $x(t)$ converges to some point in $\mathrm{zer}F$ with rate $\rho(\varepsilon/2)$, with the respective $\rho$ above.
\end{theorem}
\begin{proof}
As $\eta$ is a rate of convergence, it follows immediately that $\widehat{\eta}(\varepsilon,f):=\eta(\varepsilon)$ is also a rate of metastability. Theorem \ref{thm:met-reg} then yields a rate of metastability of the form
\[
\rho(\varepsilon,f):=\max\{\varphi(\tau(g(h(\varepsilon)/2)),n)\mid n\leq \eta(h(\varepsilon)/2)\}+1.
\]
As $\rho$ is independent of $f$, it follows that the resulting function is also a rate of convergence in both cases (see e.g.\ Proposition 2.6 in \cite{KohlenbachPinto2022}). That the above optimized rate $\rho(\varepsilon):=\varphi(\tau(g(h(\varepsilon)/2)),\eta(h(\varepsilon)/2))+1$, which dispenses of the maximum, is already sufficient follows by a simple inspection of the proof of Theorem \ref{thm:met-reg}.

If the errors are constantly $0$, then $\eta\equiv 0$ is a corresponding rate of convergence. The result again follows from Theorem \ref{thm:met-reg} as above. That, in the case of vanishing errors, an approximate $F$-point bound suffices immediately follows from the proof of Theorem \ref{thm:met-reg}, as does the fact that the above optimized rate $\rho(\varepsilon):=\varphi(\tau(g(h(\varepsilon))))+1$, which dispenses of the division by $2$, suffices. The last claim is immediate.
\end{proof}

We could have similarly slightly optimized the bounds in Theorem \ref{thm:generalFejer-rates} as we have done for the above Theorem \ref{thm:met-reg-rates}, but have chosen to only carry this out in the context of rates of convergence where optimality is much more critical.

We also here want to remind of Remark \ref{rem:bounded}, where it is discussed how the boundedness assumption featuring in Theorem \ref{thm:met-reg} can be derived from the $(G,H)$-quasi-Fej\'er monotonicity under suitable assumptions.

We now want to recall Remark \ref{rem:rateRemark}, were  the relationship between quantitative results for the discrete-time setting and the continuous-time setting have been discussed before. Contrary to the previous general results on rates of metastability, at least parts of the present results on rates of convergence under regularity assumptions can indeed be derived from corresponding results for rates of convergence in the discrete-time setting as given in \cite{KohlenbachLopezAcedoNicolae2019} (see also \cite{Pischke2023}), at least under additional assumptions (such as suitable, e.g.\ vanishing, errors or in the context of rates of asymptotic regularity w.r.t.\ $F$). 

As all these alternative proofs nevertheless retain a certain amount of pre- and postprocessing, which makes them practically not much more optimal than the tailored arguments presented above, we do not spell this out in detail for the previous theorems, and instead focus on one particular additional result for fast rates of convergence that we can derive in that way, which will be useful for our applications later.

\begin{lemma}[essentially Theorem 4.5 in \cite{KohlenbachLopezAcedoNicolae2019}]\label{lem:fastDiscrete}
Let $(x_n)$ be a sequence in $X$ and assume that $F$ has a modulus of regularity $\tau(\varepsilon)=k\varepsilon$ for $k>0$ w.r.t.\ $\overline{B}_b(z)$, where $z\in\mathrm{zer}F$ and $b\geq d(x_0,z)$. Assume further
\[
\forall y\in\mathrm{zer}F\ \forall n\in\mathbb{N}\left( d^p(x_{n+1},y)\leq d^p(x_n,y)-\beta (F(x_{n+1}))^p\right).
\]
Then
\[
\mathrm{dist}(x_n,\mathrm{zer}F)\leq c^n\mathrm{dist}(x_0,\mathrm{zer}F)\text{ and }d(x_n,x_{n+k})\leq 2c^n\mathrm{dist}(x_0,\mathrm{zer}F)
\]
for all $n,k\in\mathbb{N}$, where $c:=(1+\beta k^p)^{-1/p}\in (0,1)$.
\end{lemma}

Note that the above result is slightly different to Theorem 4.5 presented in \cite{KohlenbachLopezAcedoNicolae2019}, which features $\beta (F(x_{n}))^p$ instead of $\beta (F(x_{n+1}))^p$. By inspecting the proof given in \cite{KohlenbachLopezAcedoNicolae2019}, it can however easily be seen that the above Lemma holds true as well.

In any way, a lift to the continuous setting yields the following:

\begin{theorem}\label{thm:fastRatesCont}
Let $x:[0,+\infty)\to X$ be a dynamical system and assume that $F$ has a modulus of regularity $\tau(\varepsilon)=k\varepsilon$ for $k>0$ w.r.t.\ $\overline{B}_b(z)$, where $z\in\mathrm{zer}F$ and $b\geq d(x(0),z)$. Assume further
\[
\forall y\in\mathrm{zer}F\ \forall n\in\mathbb{N}\left( d^p(x(n+1),y)\leq d^p(x(n),y)-\beta (F(x(n+1)))^p\right)
\]
and that $x$ is Fej\'er monotone w.r.t.\ $F$. Then
\[
\mathrm{dist}(x(t),\mathrm{zer}F)\leq c^{\floor*{t}}\mathrm{dist}(x_0,\mathrm{zer}F)\text{ and }d(x(t),x(t+s))\leq 2c^{\floor*{t}}\mathrm{dist}(x_0,\mathrm{zer}F)
\]
for all $t,s\geq 0$, where $c:=(1+\beta k^p)^{-1/p}\in (0,1)$.
\end{theorem}
\begin{proof}
The above Lemma \ref{lem:fastDiscrete} in combination with the first assumption in particular yields $\mathrm{dist}(x(n),\mathrm{zer}F)\leq c^n\mathrm{dist}(x(0),\mathrm{zer}F)$ for all $n\in\mathbb{N}$. Now, for arbitrary $y\in \mathrm{zer}F$, the Fej\'er monotonicity yields
\[
\mathrm{dist}(x(t),\mathrm{zer}F)\leq d(x(t),y)\leq d(x(\floor*{t}),y),
\]
so that
\[
\mathrm{dist}(x(t),\mathrm{zer}F)\leq \mathrm{dist}(x(\floor*{t}),\mathrm{zer}F)\leq c^{\floor*{t}}\mathrm{dist}(x(0),\mathrm{zer}F).
\]
Further, we now also have
\[
d(x(t),x(t+s))\leq d(x(t),y)+d(x(t+s),y)\leq 2d(x(t),y)
\]
and so get 
\[
d(x(t),x(t+s))\leq 2\mathrm{dist}(x(t),\mathrm{zer}F)\leq 2c^{\floor*{t}}\mathrm{dist}(x(0),\mathrm{zer}F).\qedhere
\]
\end{proof}

The above result can also be extended to incorporate sufficiently fastly decreasing errors, in which case it can no longer be lifted from the discrete-time results but has to be re-derived. We do not spell this out here any further.

\section{Applications to dynamical systems}

To illustrate the breadth of the above rather general approach, we are now concerned with applying the previous results to concrete dynamical systems from the literature.

Two examples will focus on dynamical systems over Hilbert spaces. The first presents a quantitative analysis of the work \cite{BotCsetnek2017} by Bo\c{t} and Csetnek for a first-order dynamical system over a nonexpansive operator. In particular, we derive rates of convergence under a metric regularity in that case, which in the special case of a continuous variant of a forward-backward scheme extend those derived in \cite{BotCsetnek2018} using quite different methods to a broader setting. The second presents a quantitative analysis of the work \cite{BotCsetnek2016} by Bo\c{t} and Csetnek for a second-order dynamical system over a cocoercive operator. In particular, this system presents an example where our approach, even in the context of metric regularity assumptions, cannot provide rates of convergence as we can only provide a rate of metastability for the respective error function. This, we think, in particular presents a novel structural reason for why no rates of convergence have appeared for that particular scheme, at that level of generality, which provides deeper insight into the nature of the scheme presented in \cite{BotCsetnek2016}. In particular, the present rates of metastability seem to be the first structural quantitative results for the method presented in~\cite{BotCsetnek2016}. To illustrate the metric generality, our last case study is concerned with (generalized) gradient flows and associated semigroups in Hadamard spaces, that is geodesic metric spaces of nonpositive curvature. In particular, we study both the classical example of the gradient flow semigroup in these spaces, following the well-known work of Mayer \cite{Mayer1998} and Ba\v{c}\'ak \cite{Bacak2013}, and then further provide a first quantitative study of a related nonlinear semigroup generated by a nonexpansive mapping due to Stojkovic \cite{Stojkovic2012} and later studied by Ba\v{c}\'ak and Reich \cite{BacakReich2014}. In both cases, the quantitative results we provide are either completely novel, in particular relating to general rates of metastability, or extend previously known results on rates of convergence from specific ad-hoc assumption to broader metric regularity conditions.

\subsection{First-order systems over nonexpansive operators}

We begin with the essential setup from \cite{BotCsetnek2017}. Let $X$ be a Hilbert space with norm $\norm{\cdot}$ and inner product $\langle\cdot,\cdot\rangle$ and let $T:X\to X$ be a nonexpansive map. For a Lebesgue measurable function $\lambda:[0,+\infty)\to [0,1]$ and $x_0\in X$, we consider the dynamical system
\[
\begin{cases}
\dot{x}(t)=\lambda(t)\left( T(x(t))-x(t)\right),\\
x(0)=x_0.
\end{cases}\tag*{$(*)_1$}\label{firstOrder}
\]

This system is motivated by the first-order dynamical system 
\[
\dot{x}+x=\mathrm{prox}_{\mu f}(x-\mu B(x))
\]
for a given proper, convex and lower-semicontinuous function $f$, where $\mathrm{prox}_{\mu f}$ is the corresponding proximal map, and a cocoercive operator $B$, as studied by Abbas and Attouch \cite{AbbasAttouch2015} (and in fact captures that system, if extended appropriately to averaged maps). This system in turn arose as an abstraction of the first-order dynamical system 
\[
\dot{x}+x=P_C(x-\mu \nabla \phi(x))
\]
for a given closed convex nonempty set $C$ and a convex $C^1$ function $\phi$, first studied by Bolte \cite{Bolte2003} (see also \cite{Antipin1994}), and indeed captures an even more general continuous-time variant of the forward-backward method as discussed later. We refer to the survey \cite{Csetnek2020} for further discussions along these lines.

Following \cite{BotCsetnek2017} (see Definition 2 therein), we call a dynamical system $x:[0,+\infty)\to X$ a strong global solution of \ref{firstOrder} if $x$ satisfies the respective equation almost everywhere on $[0,+\infty)$ and $x$ is absolutely continuous on each interval $[0,b]$ with $b\in (0,\infty)$. The asymptotic behavior of such strong global solutions is studied in \cite{BotCsetnek2017} (see Theorem 6 therein), in particular establishing a continuous form of asymptotic regularity of the dynamical system together with weak convergence towards a fixed point of $T$. In this section, we want to study these results from a quantitative perspective.

As discussed in Remark 8 in \cite{BotCsetnek2017}, a discretization of the above equation w.r.t.\ the time variable in particular yields the well-known Krasnosel'skii-Mann method \cite{Krasnoselskii1955,Mann1953} (see also \cite{BauschkeCombettes2017}). Also, the above system immediately allows for extensions to averaged mappings (see Corollary 9 in \cite{BotCsetnek2017} and the surrounding discussions).

At first, it follows from the above assumptions that the system \ref{firstOrder} has a unique strong global solution (see the discussion given in Section 2 in \cite{BotCsetnek2017}). Further, that solution is immediately bounded (see item (i) of \cite[Theorem 6]{BotCsetnek2017}). We recall that result here, and in particular provide a more quantitative perspective on it in the following lemma:

\begin{lemma}\label{lem:boundedness}
Let $y\in\mathrm{Fix}T$ and let $b\in\mathbb{N}$ with $\norm{x_0-y}\leq b$. Let $x$ be the unique strong global solution to \ref{firstOrder}. Then $x$ is bounded with $\norm{x(t)-y}\leq b$.
\end{lemma}
\begin{proof}
As in \cite{BotCsetnek2017}, Eq.\ (8), we have
\[
\frac{d}{dt}\norm{x(t)-y}^2+\lambda(t)(1-\lambda(t))\norm{T(x(t))-x(t)}^2+\norm{\dot{x}(t)}^2\leq 0
\]
so that, as $\lambda(t)(1-\lambda(t))\norm{T(x(t))-x(t)}^2$ and $\norm{\dot{x}(t)}^2$ are nonnegative, we have 
\[
\frac{d}{dt}\norm{x(t)-y}^2\leq 0
\]
so that $t\mapsto \norm{x(t)-y}$ is decreasing. In particular, we have
\[
\norm{x(t)-y}\leq \norm{x(0)-y}\leq b
\]
for all $t\geq 0$.
\end{proof}

To capture the corresponding fixed point problem in the setup of the previous sections, we set $X_0=\overline{B}_{b}(y)$ and 
\[
F(z)=\begin{cases}\norm{z-Tz}&\text{if }z\in X_0,\\+\infty&\text{otherwise},\end{cases}
\]
for a fixed solution $y\in\mathrm{Fix}T$ so that $\mathrm{zer}F=\mathrm{Fix}T\cap X_0$ and 
\[
\mathrm{lev}_{\leq\varepsilon} F=\{z\in X_0\mid \norm{z-Tz}\leq \varepsilon\}.
\]

We now turn to the asymptotic regularity of the dynamical system, that is the property that
\[
\lim_{t\to\infty}\norm{T(x(t))-x(t)}=0.
\]
A quantitative version of that property given in the following will in particular ensure that the dynamical system $x$ has approximate $F$-points for the $F$ defined above, so that in particular our previous general and abstract results apply to it. Generally of course, results of the above form and their quantitative renderings are also of independent interest, and rates of asymptotic regularity for the dynamical system $x$ as above in particular provide continuous-time analogs of respective rates for the so-called $T$-asymptotic regularity of the Krasnosel'skii-Mann method $(x_n)$, that is rates for the respective limit $\lim_{n\to\infty}\norm{Tx_n-x_n}=0$ (see e.g.\ \cite{Kohlenbach2001,Kohlenbach2003}). Further, note that by $x$ solving the equations \ref{firstOrder}, we in particular always get
\[
\norm{\dot{x}(t)}\leq \norm{T(x(t))-x(t)},
\]
so that any quantitative rendering of the latter also provides a similar such result for the former.
 
In \cite{BotCsetnek2017} (see item (ii) of Theorem 6 therein), this result is now established under two different (and in fact independent \cite[Remark 7]{BotCsetnek2017}) conditions:
\[
\int_0^\infty\lambda(t)(1-\lambda(t))dt=+\infty\text{ and }\inf_{t\geq 0}\lambda(t)>0.
\]
We now provide quantitative versions of these results by giving rates of convergence in each case. We begin with the former assumption:

\begin{lemma}\label{lem:assym-reg-bot}
Let $y\in\mathrm{Fix}T$ and let $b\in\mathbb{N}$ with $\norm{x_0-y}\leq b$. Let further $\int_0^{+\infty}\lambda(t)(1-\lambda(t))\,dt=+\infty$ with a corresponding divergence modulus $\eta:(0,\infty)\to \mathbb{N}$, i.e.\ where
\[
\int_0^{r}\lambda(t)(1-\lambda(t))\,dt\geq K
\]
for any $r\geq \eta(K)$ and $K>0$. Then $x$ is $T$-asymptotically regular, i.e.
\[
\forall \varepsilon>0\ \forall t\geq \varphi(\varepsilon)\left(\norm{T(x(t))-x(t)}\leq \varepsilon\right)
\]
with rate $\varphi(\varepsilon):=\eta(b^2/\varepsilon^2)$. In particular, $x$ has approximate $F$-points with modulus $\varphi$.
\end{lemma}
\begin{proof}
As in \cite{BotCsetnek2017}, Eq.\ (8), we have
\[
\frac{d}{dt}\norm{x(t)-y}^2+\lambda(t)(1-\lambda(t))\norm{T(x(t))-x(t)}^2+\norm{\dot{x}(t)}^2\leq 0.
\]
This yields, as $\norm{\dot{x}(t)}^2$ is nonnegative, that
\[
\lambda(t)(1-\lambda(t))\norm{T(x(t))-x(t)}^2\leq -\frac{d}{dt}\norm{x(t)-y}^2
\]
and so that
\begin{align*}
\int_0^r\lambda(t)(1-\lambda(t))\norm{T(x(t))-x(t)}^2\,dt&\leq -\int_0^r\frac{d}{dt}\norm{x(t)-y}^2\,dt\\
&=\norm{x(0)-y}^2-\norm{x(r)-y}^2\\
&\leq \norm{x(0)-y}^2\leq b^2
\end{align*}
for $r\geq 0$. As $t\mapsto\norm{T(x(t))-x(t)}^2$ is nonincreasing, we get
\[
\norm{T(x(r))-x(r)}^2\int_0^r\lambda(t)(1-\lambda(t))\,dt\leq \int_0^r\lambda(t)(1-\lambda(t))\norm{T(x(t))-x(t)}^2\,dt\leq b^2.
\]
For $r\geq \eta(b^2/\varepsilon^2)=\varphi(\varepsilon)$, we get
\[
\norm{T(x(r))-x(r)}^2\cdot \frac{b^2}{\varepsilon^2}\leq \norm{T(x(r))-x(r)}^2\int_0^{r}\lambda(t)(1-\lambda(t))\,dt\leq b^2
\]
and so $\norm{T(x(r))-x(r)}\leq\varepsilon$. That $x$ has approximate zeros w.r.t.\ $F$ with the same modulus follows from Lemma \ref{lem:boundedness}.
\end{proof}

A particular assumption under which not only $\int_0^{+\infty}\lambda(t)(1-\lambda(t))\,dt=+\infty$ can be guaranteed but where a corresponding divergence modulus $\eta$ can be explicitly constructed is if 
\[
\underline{\tau}=\inf_{t\geq 0}\lambda(t)(1-\lambda(t))>0.
\]
Then $\eta(K):=\frac{K}{\underline{\tau}}$ is a corresponding divergence modulus. In particular, the above Lemma \ref{lem:assym-reg-bot} yields
\[
\norm{T(x(t))-x(t)}\leq \frac{b}{\sqrt{\underline{\tau} t}}
\]
for any $t\geq 0$ and any $b\geq \norm{x_0-y}$. In particular, for $b= \norm{x_0-y}$, and taking the infimum over $y$, we obtain 
\[
\norm{T(x(t))-x(t)}\leq \frac{d(x_0,\mathrm{Fix}T)}{\sqrt{\underline{\tau} t}},
\]
which was already obtained as Theorem 10 in \cite{BotCsetnek2017}. In fact, Bo\c{t} and Csetnek afterwards show that the convergence can be improved to $o(\frac{1}{\sqrt{t}})$ (see Theorem 11 in \cite{BotCsetnek2017}), but we will not be concerned with this further here. 

We now turn to the alternative assumption $\inf_{t\geq 0}\lambda(t)>0$ for the asymptotic regularity of the dynamical system, which thereby essentially presents itself as a weakening of the above property.

\begin{lemma}\label{lem:assym-reg-botSecond}
Let $y\in\mathrm{Fix}T$ and let $b\in\mathbb{N}$ with $\norm{x_0-y}\leq b$. Let further $\inf_{t\geq 0}\lambda(t)>0$ with a corresponding witness $\underline{\lambda}>0$, i.e.\ where
\[
\lambda(t)\geq\underline{\lambda}\text{ for all }t\in[0,\infty).
\]
Then $x$ is $T$-asymptotically regular, i.e.
\[
\forall \varepsilon>0\ \forall t\geq \varphi(\varepsilon)\left(\norm{T(x(t))-x(t)}\leq \varepsilon\right)
\]
with rate $\varphi(\varepsilon):=4b^4/\underline{\lambda}^2\varepsilon^2$. In particular, $x$ has approximate $F$-points with modulus $\varphi$.
\end{lemma}
\begin{proof}
    Completely parallel to the first part of the proof of Lemma~\ref{lem:assym-reg-bot}, we get
    \begin{equation*}
        \int_0^\infty\norm{\dot x(t)}\,dt\leq b^2.
    \end{equation*}
    Given that~$x$ is the unique global solution of~\ref{firstOrder}, this yields
    \begin{equation*}
        \int_0^\infty\frac12\norm{T(x(t))-x(t)}^2\,dt=\int_0^\infty\frac{\norm{\dot x(t)}^2}{2\lambda(t)^2}\,dt\leq\frac{b^2}{2\underline{\lambda}^2}.
    \end{equation*}
    As in the proof of Theorem 6 in \cite{BotCsetnek2017}, for almost all~$t$ we also have
    \begin{equation*}
        \frac{d}{dt}\left(\frac12\norm{T(x(t))-x(t)}^2\right)\leq 0.
    \end{equation*}
    For any~$\varepsilon>0$ and with $f\equiv 0$, we can now apply Lemma~\ref{lem:AAS2} to get
\begin{equation*}
\exists n\leq \tilde{f'}^{(\varpi(\varepsilon))}(0)\left( \frac12\norm{T(x(n))-x(n)}^2<\frac{\varepsilon^2}{2}\right),
\end{equation*}
where $\tilde{f'}(n):=n+\ceil*{3b^2/2\underline{\lambda}^2\varepsilon}$ as $f\equiv 0$, and $\varpi(\varepsilon):=\left\lceil b^2/\varepsilon\right\rceil$. In particular, we have
\[
\tilde{f'}^{(\varpi(\varepsilon))}(0)=\varpi(\varepsilon)\ceil*{\frac{3b^2}{2\underline{\lambda}^2\varepsilon}}\leq \frac{2b^2}{\varepsilon}\cdot\frac{2b^2}{\underline{\lambda}^2\varepsilon}=\frac{4b^4}{\underline{\lambda}^2\varepsilon^2}.
\]
As $t\mapsto\norm{T(x(t))-x(t)}^2$ is nonincreasing, the latter even holds for all~$t\geq n$, which yields the result.
\end{proof}

Fej\'er monotonicity of the unique strong global solution $x$ to the system \ref{firstOrder} w.r.t.\ $F$ now comes to pass in exactly the way discussed in Section \ref{sec:motivation}, through a differential inequality with an associated energy functional. Indeed, in \cite{BotCsetnek2017} (see eq.\ (8) therein), it is established that
\[
\frac{d}{dt}\norm{x(t)-y}^2\leq 0\text{ for all }y\in\mathrm{Fix}T,
\]
so that the inequality here is even of the most simple form discussed in Section \ref{sec:motivation}. We now analyse the proof of this inequality as given in \cite{BotCsetnek2017} to establish the associated uniform Fej\'er monotonicity and extract an associated modulus. Indeed, the following result now provides a quantitative variant of the above differential inequality, relating approximate solutions to approximate inequalities.

\begin{lemma}\label{lem:approximateDerivative}
Let $y\in\mathrm{Fix}T$ and let $b\in\mathbb{N}$ with $\norm{x_0-y}\leq b$. For any $\varepsilon>0$, if $z\in\mathrm{lev}_{\leq \varepsilon/4b}F$, then $\frac{d}{dt}\norm{x(t)-z}^2\leq \varepsilon$ for all $t\geq 0$.
\end{lemma}
\begin{proof}
Take $z\in\mathrm{lev}_{\leq \varepsilon/4b}F$. Analogous to the proof of Theorem 6 in \cite{BotCsetnek2017}, we have
\begin{align*}
\frac{d}{dt}\norm{x(t)-z}^2&=2\langle\dot{x}(t),x(t)-z\rangle\\
&=\norm{\dot{x}(t)+x(t)-z}^2-\norm{x(t)-z}^2-\norm{\dot{x}(t)}^2\\
&\leq \norm{\lambda(t)\left(Tx(t)-z\right)+(1-\lambda(t))\left(x(t)-z\right)}^2-\norm{x(t)-z}^2\\
&\leq \norm{Tx(t)-z}^2-\norm{x(t)-z}^2\\
&=(\norm{Tx(t)-z}-\norm{x(t)-z})(\norm{Tx(t)-z}+\norm{x(t)-z}).
\end{align*}
As $z\in X_0$, we have 
\[
\norm{x(t)-z}\leq \norm{x(t)-z}+\norm{z-y}\leq \norm{x_0-z}+\norm{z-y}\leq 2b
\]
as well as 
\[
\norm{Tx(t)-z}\leq \norm{Tx(t)-y}+\norm{z-y}\leq \norm{x(t)-y}+\norm{z-y}\leq \norm{x_0-z}+\norm{z-y}\leq 2b
\]
using the nonexpansivity of $T$ and Lemma \ref{lem:boundedness}. Thus, we have 
\begin{align*}
\frac{d}{dt}\norm{x(t)-z}^2&\leq (\norm{Tx(t)-z}-\norm{x(t)-z})4b\\
&\leq (\norm{Tx(t)-Tz}+\norm{Tz-z}-\norm{x(t)-z})4b\\
&\leq \norm{Tz-z}4b\leq \varepsilon
\end{align*}
as $\norm{Tz-z}\leq \varepsilon/4b$.
\end{proof}

An approximate differential inequality like the above then yields a modulus of uniform Fej\'er monotonicity (with $G=H=(\cdot)^2$) via the mean value theorem:

\begin{lemma}\label{lem:firstOrderUniFejer}
Let $y\in\mathrm{Fix}T$ and let $b\in\mathbb{N}$ with $\norm{x_0-y}\leq b$. Then $x$ is uniformly $(G,H)$-Fej\'er monotone w.r.t.\ $F$ and with $G=H=(\cdot)^2$, i.e.
\[
\forall \varepsilon>0\ \forall n,m\in\mathbb{N}\ \forall z\in \mathrm{lev}_{\leq\chi(\varepsilon,n,m)}F\ \forall s\leq t\in [n,n+m]\left(\norm{x(t)-z}^2\leq \norm{x(s)-z}^2+\varepsilon\right),
\]
where $\chi(\varepsilon,n,m):=\varepsilon/(4bm)$.
\end{lemma}
\begin{proof}
Fix $\varepsilon,n,m$ as well as $z\in \mathrm{lev}_{\leq\chi(\varepsilon,n,m)}F$ and $s,t\in [n,n+m]$. If $s=t$, the result is clear. If $s<t$, then the mean value theorem yields an $r\in (s,t)$ such that
\[
\frac{d}{dt}\norm{x(r)-z}^2=\frac{\norm{x(t)-z}^2-\norm{x(s)-z}^2}{t-s}.
\]
Thus, as $z\in \mathrm{lev}_{\leq\chi(\varepsilon,n,m)}F$, we have
\[
\norm{x(t)-z}^2-\norm{x(s)-z}^2\leq \frac{\varepsilon}{m}(t-s)\leq \frac{\varepsilon}{m}\cdot m=\varepsilon
\]
by Lemma \ref{lem:approximateDerivative}.
\end{proof}

A direct application of Theorem \ref{thm:generalFejer-rates} now yields the following quantitative result on the behavior of the solution $x$ to \ref{firstOrder} under a finite-dimensionality (that is, local compactness) assumption on $X$, giving a rate of metastability for the convergence result established in Theorem 6 of \cite{BotCsetnek2017} in that case:

\begin{theorem}\label{thm:firstOrderMeta}
Let $X$ be a finite-dimensional Hilbert space with dimension $d$ and let $T:X\to X$ be a nonexpansive map. Let $y\in\mathrm{Fix}T$ and let $b\in\mathbb{N}$ with $\norm{x_0-y}\leq b$. Further, assume that $\lambda:[0,+\infty)\to [0,1]$ is Lebesgue measurable and fix $x_0\in X$. Let $x$ be the unique strong global solution to \ref{firstOrder}. Assume either:
\begin{enumerate}
\item $\int_0^{+\infty}\lambda(t)(1-\lambda(t))\,dt=+\infty$ with a divergence modulus $\eta:(0,\infty)\to \mathbb{N}$ with $\eta(K)\geq \eta(L)$ for $K\geq L$, in which case we set $\varphi(\varepsilon):=\eta(b^2/\varepsilon^2)$.
\item $\inf_{t\geq 0}\lambda(t)>0$ with a witness $\underline{\lambda}>0$, in which case we set $\varphi(\varepsilon):=4b^4/\underline{\lambda}^2\varepsilon^2$.
\end{enumerate}

Then, in either case, $x$ is metastable in the sense that for all $\varepsilon>0$ and $f:\mathbb{N}\to\mathbb{N}$, it holds that 
\[
\exists n\leq \Delta(\varepsilon/4,f)\ \forall s,t\in [n,n+f(n)]\left( \norm{x(s)-x(t)}\leq\varepsilon\text{ and }\norm{x(t)-T(x(t))}\leq\varepsilon\right),
\]
where moreover the bound $\Delta(\varepsilon,f)$ can be explicitly described by
\[
\Delta(\varepsilon,f):=\max\{\Delta(i,\varepsilon,f)\mid i\leq P\}+1
\]
where $P:=\ceil*{2(\ceil*{\sqrt{12}\varepsilon^{-1}}+1)\sqrt{d}b}^d+1$ and $\Delta(0,\varepsilon,f):=0$ and $\Delta(j,\varepsilon,f):=\varphi(\widehat{\varepsilon}_j)$
with
\[
\widehat{\varepsilon}_j=\min\left\{\frac{\varepsilon}{2},\frac{\varepsilon^2/12}{4b(f(m+1)+1)}\;\Big\vert\; m\leq \Delta(i,\varepsilon,f),i<j\right\}
\]
for $j\geq 1$.
\end{theorem}
\begin{proof}
The result immediately follows by instantiating Theorem \ref{thm:generalFejer-rates} over $X_0=\overline{B}_{b}(y)$ and simplifying the resulting bound (recall that $x(t)\in X_0$ for all $t\geq 0$ by Lemma \ref{lem:boundedness}).  In particular, we set $G:=H:=(\cdot)^2$ as well as $h(\varepsilon):=\varepsilon^2$ and $g(\varepsilon):=\sqrt{\varepsilon}$. That $x$ is uniformly $(G,H)$-Fej\'er monotone w.r.t.\ $F$ follows by Lemma \ref{lem:firstOrderUniFejer}, which in particular yields that $\chi(\varepsilon,n,m):=\varepsilon/(4bm)$ is an associated modulus. The definition of $\widehat{\varepsilon}_j$ is a resulting simplification. As we do not deal with errors, it suffices that $x$ has the approximate $F$-point property. The above assumptions (1) and (2) guarantee that with respective bounds, using Lemmas \ref{lem:assym-reg-bot} and \ref{lem:assym-reg-botSecond}. Note that the respective $\varphi$ is immediately monotone under our assumptions. Also, note that $X_0$ is totally bounded with a modulus $\gamma(\varepsilon):=\ceil*{2(\ceil*{\varepsilon^{-1}}+1)\sqrt{d}b}^d$ by Example 2.8 of \cite{KohlenbachLeusteanNicolae2018}. Lastly, note that $F$ is uniformly continuous on $X_0$ with modulus $\omega(\varepsilon):=\varepsilon/2$, since we have $\vert \norm{z-Tz}-\norm{z'-Tz'}\vert\leq 2\norm{z-z'}$ using the nonexpansivity of $T$.
\end{proof}

Further, a direct application of Theorem \ref{thm:met-reg-rates} yields the following general construction of a rate of convergence under a regularity assumption on the mapping $T$ and its associated fixed point set $\mathrm{Fix}T$:

\begin{theorem}
Let $X$ be a Hilbert space and let $T:X\to X$ be a nonexpansive map. Let $y\in\mathrm{Fix}T$ and let $b\in\mathbb{N}$ with $\norm{x_0-y}\leq b$. Assume that $F$ has a modulus of regularity $\tau$ w.r.t.\ $\overline{B}_b(y)$. Further, assume that $\lambda:[0,+\infty)\to [0,1]$ is Lebesgue measurable and fix $x_0\in X$. Let $x$ be the unique strong global solution to \ref{firstOrder}. Assume either:
\begin{enumerate}
\item $\int_0^{+\infty}\lambda(t)(1-\lambda(t))\,dt=+\infty$ with a divergence modulus $\eta:(0,\infty)\to \mathbb{N}$ with $\eta(K)\geq \eta(L)$ for $K\geq L$, in which case we set $\varphi(\varepsilon):=\eta(b^2/\varepsilon^2)$.
\item $\inf_{t\geq 0}\lambda(t)>0$ with a witness $\underline{\lambda}>0$, in which case we set $\varphi(\varepsilon):=4b^4/\underline{\lambda}^2\varepsilon^2$.
\end{enumerate}

Then, in either case, $x$ satisfies
\[
\forall \varepsilon>0\ \forall t\geq \rho(\varepsilon) \left( \mathrm{dist}(x(t),\mathrm{Fix}T)<\varepsilon\right)
\]
and further $x$ is Cauchy with
\[
\forall \varepsilon>0\ \forall s,t\geq \rho(\varepsilon/2)\left( \norm{x(t)-x(s)}<\varepsilon\right),
\]
where $\rho$ is defined by $\rho(\varepsilon):=\varphi(\tau(\varepsilon)/2)+1$. In particular, $x$ converges to some fixed point of $T$ with that rate.
\end{theorem}
\begin{proof}
The result immediately follows by instantiating Theorem \ref{thm:met-reg-rates} and simplifying the resulting bound. In particular, we set $G:=H:=(\cdot)^2$ as well as $h(\varepsilon):=\varepsilon^2$ and $g(\varepsilon):=\sqrt{\varepsilon}$. That $x$ is $(G,H)$-Fej\'er monotone w.r.t.\ $F$ follows by eq.\ (8) in \cite{BotCsetnek2017} (recall also Lemma \ref{lem:firstOrderUniFejer}). As we do not deal with errors, it suffices that $x$ has the approximate $F$-point property. The above assumptions (1) and (2) guarantee that with respective bounds $\varphi(\varepsilon/2)$, using Lemmas \ref{lem:assym-reg-bot} and \ref{lem:assym-reg-botSecond} as before. 
\end{proof}

In the special case of operators $T$ with linear moduli of regularity, we can further employ the previous Theorem \ref{thm:fastRatesCont} to derive exponentially fast rates of convergence for the solution of \ref{firstOrder}:

\begin{theorem}\label{thm:firstOrderFast}
Let $X$ be a Hilbert space and let $T:X\to X$ be a nonexpansive map. Let $y\in\mathrm{Fix}T$ and let $b\in\mathbb{N}$ with $\norm{x_0-y}\leq b$. Further, assume that $\lambda:[0,+\infty)\to [0,1]$ is Lebesgue measurable and fix $x_0\in X$. Let $x$ be the unique strong global solution to \ref{firstOrder}. Assume that $F$ has a modulus of regularity $\tau(\varepsilon)=k\varepsilon$ for $k>0$ w.r.t.\ $\overline{B}_b(y)$ and further assume that $\underline{\tau}=\inf_{t\geq 0}\lambda(t)(1-\lambda(t))>0$.

Then 
\[
\mathrm{dist}(x(t),\mathrm{Fix}T)\leq c^{\floor*{t}}\mathrm{dist}(x_0,\mathrm{Fix}T)\text{ and }d(x(t),x(t+s))\leq 2c^{\floor*{t}}\mathrm{dist}(x_0,\mathrm{Fix}T)
\]
for all $t,s\geq 0$, where $c:=(1+\underline{\tau} k^p)^{-1/p}\in (0,1)$.
\end{theorem}
\begin{proof}
As in \cite{BotCsetnek2017} (see eq.\ (12) therein), for an arbitrary $y\in\mathrm{Fix}T$, we can derive
\[
\frac{d}{dt}\norm{x(t)-y}^2+\lambda(t)(1-\lambda(t))\norm{T(x(t))-x(t)}^2\leq 0
\]
for all $t\geq 0$. As before, this yields the Fej\'er monotonicity of $x$, but also integrating from $n$ to $n+1$ yields
\[
\norm{x(n+1)-y}^2\leq \norm{x(n)-y}^2-\underline{\tau}\int_{n}^{n+1}\norm{T(x(\tau))-x(\tau)}^2\,d\tau.
\]
As established in the proof of Theorem 6 in \cite{BotCsetnek2017}, $\norm{T(x(\tau))-x(\tau)}^2$ is decreasing. Thus, we have
\begin{align*}
\norm{x(n+1)-y}^2&\leq \norm{x(n)-y}^2-\underline{\tau}\int_{n}^{n+1}\norm{T(x(\tau))-x(\tau)}^2\,d\tau\\
&\leq \norm{x(n)-y}^2-\underline{\tau}\norm{T(x(n+1))-x(n+1)}^2\\
&\leq \norm{x(n)-y}^2-\underline{\tau}(F(x(n+1)))^2.
\end{align*}
Using Theorem \ref{thm:fastRatesCont}, we get the result.
\end{proof}

Combined with Example \ref{ex:mappings}, the above result immediately allows for the construction of fast rates of convergence in the case where $T$ is a quasi-contraction or a continuous orbital-contraction. However, the above is certainly not limited to that.

Let us now consider a particular instance of the system \ref{firstOrder} to which the above results apply: the continuous-time variant of the forward-backward method. We first recall the method from \cite{BotCsetnek2017}: Over a Hilbert space $X$, consider a maximally monotone set-valued operator $A:X\to 2^X$ and a single-valued map $B:X\to X$ which is $\beta$-cocoercive, for some $\beta>0$. For $\gamma\in (0,2\beta)$ and $\delta:=\min\{1,\beta/\gamma\}+1/2$ as well as a Lebesgue measurable function $\lambda:[0,+\infty)\to [0,\delta]$ and a starting point $x_0\in X$, we consider the dynamical system 
\[
\begin{cases}
\dot{x}(t)=\lambda(t)\left( J_{\gamma A}(x(t)-\gamma B(x(t)))-x(t)\right),\\
x(0)=x_0,
\end{cases}\tag*{$(*)'_1$}\label{FB}
\]
where $J_{\gamma A}:=(\mathrm{Id}+\gamma A)^{-1}$ is the (single-valued and total) resolvent of $\gamma A$. Indeed, the above system can be conceived of as an instantiation of the system \ref{firstOrder} using the map $T:=J_{\gamma A}\circ (\mathrm{Id}-\gamma B)$. As in the proof of Theorem 12 in \cite{BotCsetnek2017}, one can show that $T$ is in fact $\frac{1}{\delta}$-averaged, that is by definition that there is a nonexpansive operator $R$ such that
\[
T=\left(1-\frac{1}{\delta}\right)\mathrm{Id}+\frac{1}{\delta}R.
\]
Thereby, the system \ref{FB} is further an instantiation of \ref{firstOrder} using the map $R$ and the parameter sequence $\frac{1}{\delta}\lambda$ instead of $\lambda$ (see also the proof of Corollary 9 in \cite{BotCsetnek2017}). Note also that $\mathrm{Fix}T=\mathrm{zer}(A+B)$.

In that way, we immediately obtain the following result on the quantitative asymptotic behavior of the associated solution under a finite-dimensionality assumption on $X$, a quantitative variant of parts of Theorem 12 in \cite{BotCsetnek2017}:

\begin{theorem}\label{thm:metaFBFirstOrder}
Let $X$ be a finite-dimensional Hilbert space with dimension $d$ and let $A:X\to 2^X$ be maximally monotone and $B:X\to X$ be $\beta$-cocoercive for $\beta>0$. Let $\gamma\in (0,2\beta)$ and write $\delta:=\min\{1,\beta/\gamma\}+1/2$. Let $y\in\mathrm{zer}(A+B)$ and let $b\in\mathbb{N}$ with $\norm{x_0-y}\leq b$. Further, assume that $\lambda:[0,+\infty)\to [0,\delta]$ is Lebesgue measurable and fix $x_0\in X$. Let $x$ be the unique strong global solution to \ref{FB}. Assume either:
\begin{enumerate}
\item $\int_0^{+\infty}\lambda(t)(\delta-\lambda(t))\,dt=+\infty$ with a divergence modulus $\eta:(0,\infty)\to \mathbb{N}$ with $\eta(K)\geq \eta(L)$ for $K\geq L$, in which case we set $\varphi(\varepsilon):=\eta(\delta b^2/\varepsilon^2)$.
\item $\inf_{t\geq 0}\lambda(t)>0$ with a witness $\underline{\lambda}>0$, in which case we set $\varphi(\varepsilon):=4b^4\delta^2/\underline{\lambda}^2\varepsilon^2$.
\end{enumerate}

Then, in either case, $x$ is metastable in the sense that for all $\varepsilon>0$ and $f:\mathbb{N}\to\mathbb{N}$, it holds that 
\[
\exists n\leq \Delta(\varepsilon/4,f)\ \forall s,t\in [n,n+f(n)]\left( \norm{x(s)-x(t)}\leq\varepsilon\text{ and }\norm{x(t)-T(x(t))}\leq\varepsilon\right),
\]
where moreover the bound $\Delta(\varepsilon,f)$ can be explicitly described as in Theorem \ref{thm:firstOrderMeta}, using the above $\varphi$ in either case.
\end{theorem}

\begin{remark}\label{rem:approximateZeros}
Note that since $T(x):=J_{\gamma A}(x-\gamma B(x))$, we have $x-\gamma B(x)\in T(x)+\gamma A(T(x))$ by definition of the resolvent. In particular, we have $\gamma^{-1}(x-T(x))- B(x)\in A(T(x))$. Combined, we hence obtain
\[
\gamma^{-1}(x-T(x))+B(T(x))-B(x)\in (A+B)(T(x)).
\]
Write $w:=\gamma^{-1}(x-T(x))+B(T(x))-B(x)$ and $v:= T(x)$, so that the above translates to $w\in (A+B)(v)$. Hence, if $\norm{x-T(x)}\leq \min\{\varepsilon,\varepsilon/( \gamma^{-1}+\beta^{-1})\}$, we get $\norm{x-v}=\norm{x-T(x)}\leq \varepsilon$ as well as
\begin{align*}
\norm{w}&=\norm{\gamma^{-1}(x-T(x))+B(T(x))-B(x)}\\
&\leq \gamma^{-1}\norm{x-T(x)}+\norm{B(T(x))-B(x)}\\
&\leq \left( \gamma^{-1}+\beta^{-1}\right)\norm{x-T(x)}\leq \varepsilon
\end{align*}
using that $B$ is $\frac{1}{\beta}$-cocoercive. As such, if the above Theorem \ref{thm:metaFBFirstOrder} is applied with $\min\{\varepsilon,\varepsilon/( \gamma^{-1}+\beta^{-1})\}$ in place of $\varepsilon$, so that the property $\norm{x(t)-T(x(t))}\leq \min\{\varepsilon,\varepsilon/( \gamma^{-1}+\beta^{-1})\}$ is maintained along the region of metastability, the dynamical system $x$ in particular contains approximate zeros of $A+B$ in the above sense in that region, that is for any $t\in [n,n+f(n)]$, there are $w\in (A+B)(v)$ such that $\norm{x(t)-v},\norm{w}\leq \varepsilon$. This argument already appears abstractly in \cite{Findling2026} in the context of the discrete-time forward-backward method.
\end{remark}

As is well-studied in the context of ``ordinary'' splitting methods, strong convergence can be guaranteed under a uniform monotonicity assumption for either of the operators $A$ or $B$, in which case rates can also be extracted. For discrete-time dynamical systems, constructions of such rates were recently studied in very general situations in the work of Treusch and Kohlenbach \cite{TreuschKohlenbach2026}. For continuous-time dynamical systems, rates are presented under a strong monotonicity assumption in \cite{BotCsetnek2018}. We here can obtain a result on rates of convergence under a uniform monotonicity assumption in a similar generality as the recent work \cite{TreuschKohlenbach2026}, and in particular following their arguments.

First, we derive a quantitative result on the asymptotic behavior of $B(x(t))$. As such, we rely on the following key inequality also used again later:

\begin{lemma}\label{lem:FB-B-ineq}
Let $X$ be a Hilbert space and let $A:X\to 2^X$ be maximally monotone and $B:X\to X$ be $\beta$-cocoercive for $\beta>0$. Fix $\gamma>0$ and define $T:=J_{\gamma A}\circ (\mathrm{Id}-\gamma B)$. For any $x\in\mathrm{zer}(A+B)$ and $z\in X$:
\[
\gamma\beta\norm{Bz-Bx}^2\leq\left( 1+\frac{\gamma}{\beta}\right)\norm{z-x}\norm{Tz-z}.
\]
\end{lemma}
\begin{proof}
We proceed similar to the proof of Theorem 26.14 in \cite{BauschkeCombettes2017}, abstracting the proof of eq.\ (26.56) therein: For any $y\in X$, define $y^-:=y-\gamma By$. Note that $Ty=J_{\gamma A} y^-$. By assumption, we also have $x=Tx=J_{\gamma A}x^-$. We then have
\begin{align*}
&\langle Tz-x,z-Tz-\gamma(Bz-Bx)\rangle\\
&\qquad=\langle Tz-x,z^-Tz+x-x^-\rangle\\
&\qquad=\langle J_{\gamma A}z^--J_{\gamma A}x^-,(\mathrm{Id}-J_{\gamma A})z^--(\mathrm{Id}-J_{\gamma A})x^-)\rangle\geq 0,
\end{align*}
where the inequality follows from the monotonicity of $A$ and the definition of the resolvent. Using the nonexpansivity of $T$ and the $\beta$-cocoercivity of $B$, we in particular have that
\begin{align*}
\norm{z-x}\norm{Tz-z}&\geq \norm{Tz-x}\norm{Tz-z}\\
&\geq\langle Tz-x,z-Tz\rangle\\
&\geq \gamma\langle Tz-x,Bz-Bx\rangle\\
&=\gamma\left(\langle Tz-z,Bz-Bx\rangle+\langle z-x,Bz-Bx\rangle\right)\\
&\geq -\gamma\norm{Tz-z}\norm{Bz-Bx}+\gamma\beta\norm{Bz-Bx}^2\\
&\geq -\frac{\gamma}{\beta}\norm{Tz-z}\norm{z-x}+\gamma\beta\norm{Bz-Bx}^2,
\end{align*}
where we have used the preceding inequality in the third line and in the last line in particular used that $B$ is also $\frac{1}{\beta}$-Lipschitz. Rearranging the above inequality yields the result.
\end{proof}

We then have the following result (cf.\ Theorem 3.3 in \cite{TreuschKohlenbach2026} for the discrete-time case, whose arguments we follows):

\begin{proposition}\label{pro:FB-B-rate}
Let $X$ be a Hilbert space and let $A:X\to 2^X$ be maximally monotone and $B:X\to X$ be $\beta$-cocoercive for $\beta>0$. Let $\gamma\in (0,2\beta)$ and write $\delta:=\min\{1,\beta/\gamma\}+1/2$. Let $y\in\mathrm{zer}(A+B)$ and let $b\in\mathbb{N}$ with $\norm{x_0-y}\leq b$. Further, assume that $\lambda:[0,+\infty)\to [0,\delta]$ is Lebesgue measurable and fix $x_0\in X$. Let $x$ be the unique strong global solution to \ref{FB}. Assume either:
\begin{enumerate}
\item $\int_0^{+\infty}\lambda(t)(\delta-\lambda(t))\,dt=+\infty$ with a divergence modulus $\eta:(0,\infty)\to \mathbb{N}$ with $\eta(K)\geq \eta(L)$ for $K\geq L$, in which case we set $\varphi(\varepsilon):=\eta(\delta b^2/\varepsilon^2)$.
\item $\inf_{t\geq 0}\lambda(t)>0$ with a witness $\underline{\lambda}>0$, in which case we set $\varphi(\varepsilon):=4b^4\delta^2/\underline{\lambda}^2\varepsilon^2$.
\end{enumerate}
Then, in either case, $\lim_{t\to\infty}B(x(t))= B(y)$ with
\[
\forall \varepsilon>0\ \forall s,t\geq \psi(\varepsilon)\left( \norm{B(x(t))-B(y)}\leq \varepsilon\right),
\]
where $\psi(\varepsilon):=\varphi(\gamma\beta\varepsilon^2/3b)$ in either case.
\end{proposition}
\begin{proof}
By Lemma \ref{lem:FB-B-ineq} and $\gamma\leq 2/\beta$, we obtain
\[
\gamma\beta\norm{B(x(t))-B(y)}^2\leq 3\norm{x(t)-y}\norm{T(x(t))-x(t)}.
\]
Lemma \ref{lem:boundedness} then yields 
\[
\norm{B(x(t))-B(y)}^2\leq \frac{3b}{\gamma\beta}\norm{T(x(t))-x(t)}
\]
which implies the results using Lemmas \ref{lem:assym-reg-bot} and \ref{lem:assym-reg-botSecond}. 
\end{proof}

The above result is obtained in \cite{BotCsetnek2017} (see Theorem 12 therein) using quite a different argument, only under the assumption that $\inf_{t\geq 0}\lambda(t)>0$ and in particular without any rates.

We now obtain a result on rates of convergence in the context of a uniform monotonicity assumption. Indeed, recall that a set-valued operator $M:X\to 2^X$ is called uniformly monotone if 
\[ 
\langle x-y,u-v\rangle \geq\phi(\norm{x-y})
\]
for all $(x,u),(y,v)\in A$, where $\phi:[0,\infty)\to [0,\infty)$ is an increasing function which vanishes only at $0$. Modelled after Theorem 3.4 in \cite{TreuschKohlenbach2026}, the resulting construction of a rate of convergence relies on the following well-known result, which we only present here in the abstract as we reuse it later:\footnote{Note that Theorem 3.4 in \cite{TreuschKohlenbach2026} is formulated using a so-called modulus of uniform monotonicity, which we chose to avoid in the presentation of this paper.}

\begin{lemma}\label{lem:FB-error}
Let $X$ be a Hilbert space and let $A:X\to 2^X$ be maximally monotone and $B:X\to X$ be $\beta$-cocoercive for $\beta>0$. Let $\gamma\in (0,2\beta)$ and let $x\in X$ and $y\in\mathrm{zer}(A+B)$ with $b\in\mathbb{N}$ such that $\norm{x-y}\leq b$. Let $\varepsilon>0$ be arbitrary.
\begin{enumerate}[(i)]
\item If $A$ is uniformly monotone with function $\phi$ and 
\[
\norm{x-T(x)}\leq\min\left\{\frac{\gamma\phi(\varepsilon/2)}{2b},\frac{\varepsilon}{2}\right\}\text{ and }\norm{B(x)-B(y)}\leq\frac{\phi(\varepsilon/2)}{2b},
\]
then $\norm{x-y}\leq \varepsilon$.
\item If $B$ is uniformly monotone with function $\phi$ and
\[
\norm{B(x)-B(y)}\leq \frac{\phi(\varepsilon)}{b},
\]
then $\norm{x-y}\leq \varepsilon$.
\end{enumerate}
\end{lemma}
\begin{proof}
For the case where $A$ is uniformly monotone with function $\phi$, note first that using Cauchy-Schwarz, and the nonexpansivity of $T$, we have
\begin{align*}
&\langle T(x)-y,\gamma^{-1}(x-T(x))-B(x)+B(y)\rangle\\
&\qquad\leq \norm{T(x)-y}(\gamma^{-1}\norm{x-T(x)}+\norm{B(x)-B(y)})\\
&\qquad\leq \frac{b}{\gamma}\norm{x(t)-T(x)}+b\norm{B(x)-B(y)}.
\end{align*}
As we have
\[
\norm{x-T(x)}\leq\min\left\{\frac{\gamma\phi(\varepsilon/2)}{2b},\frac{\varepsilon}{2}\right\}\text{ and }\norm{B(x)-B(y)}\leq\frac{\phi(\varepsilon/2)}{2b}
\]
by assumption, we get
\[
\langle T(x)-y,\gamma^{-1}(x-T(x))-B(x)+B(y)\rangle\leq \phi(\varepsilon/2).
\]
Now note that $-B(y)\in A(y)$ as $y\in\mathrm{zer}(A+B)$ and that
\[
\gamma^{-1}(x-T(x))-B(x)\in A(T(x))
\]
so that we get
\[
\phi(\norm{T(x)-y})\leq \langle T(x)-y,\gamma^{-1}(x-T(x))-B(x)+B(y)\rangle\leq \phi(\varepsilon/2).
\]
As $\phi$ is increasing, we get $\norm{T(x)-y}\leq \varepsilon/2$. Finally, we get
\[
\norm{x-y}\leq \norm{x-T(x)}+\norm{T(x)-y}\leq\varepsilon.
\]

For the case where $B$ is uniformly monotone with function $\phi$, simply note that using Cauchy-Schwarz, we have
\[
\langle x-y,B(x)-B(y)\rangle\leq \norm{x-y}\norm{B(x)-B(y)}
\]
so that using the uniform monotonicity of $B$, we get
\[
\phi(\norm{x-y})\leq \langle x-y,B(x)-B(y)\rangle\leq b\norm{B(x)-B(y)}\leq \phi(\varepsilon)
\]
since $\norm{B(x)-B(y)}\leq \phi(\varepsilon)/b$. As $\phi$ is increasing, the result follows.
\end{proof}

Combined with the above, this immediately yields the following rates of convergence:

\begin{theorem}
Let $X$ be a Hilbert space and let $A:X\to 2^X$ be maximally monotone and $B:X\to X$ be $\beta$-cocoercive for $\beta>0$. Let $\gamma\in (0,2\beta)$ and write $\delta:=\min\{1,\beta/\gamma\}+1/2$. Let $y\in\mathrm{zer}(A+B)$ and let $b\in\mathbb{N}$ with $\norm{x_0-y}\leq b$. Further, assume that $\lambda:[0,+\infty)\to [0,\delta]$ is Lebesgue measurable and fix $x_0\in X$. Let $x$ be the unique strong global solution to \ref{FB}. Assume either:
\begin{enumerate}
\item $\int_0^{+\infty}\lambda(t)(\delta-\lambda(t))\,dt=+\infty$ with a divergence modulus $\eta:(0,\infty)\to \mathbb{N}$ with $\eta(K)\geq \eta(L)$ for $K\geq L$, in which case we set $\varphi(\varepsilon):=\eta(\delta b^2/\varepsilon^2)$ and $\psi$ as in Proposition \ref{pro:FB-B-rate}.
\item $\inf_{t\geq 0}\lambda(t)>0$ with a witness $\underline{\lambda}>0$, in which case we set $\varphi(\varepsilon):=4b^4\delta^2/\underline{\lambda}^2\varepsilon^2$ and $\psi$ as in Proposition \ref{pro:FB-B-rate}.
\end{enumerate}
Then, if either $A$ or $B$ is uniformly monotone with function $\phi$, and in either case (1) or (2) above, $x$ converges to $y$. Further, we have the rates
\[
\forall \varepsilon>0\ \forall t\geq \rho(\varepsilon)\left( \norm{x(t)-y}\leq \varepsilon\right),
\]
where $\rho$ is defined as follows: 
\begin{enumerate}[(i)]
\item If $A$ is uniformly monotone with function $\phi$, then
\[
\rho(\varepsilon):=\min\left\{\varphi\left(\min\left\{\frac{\gamma\phi(\varepsilon/2)}{2b},\frac{\varepsilon}{2}\right\}\right),\psi\left(\frac{\phi(\varepsilon/2)}{2b}\right)\right\}.
\]
\item If $B$ is uniformly monotone with function $\phi$, then
\[
\rho(\varepsilon):=\psi\left(\frac{\phi(\varepsilon)}{b}\right).
\]
\end{enumerate}
\end{theorem}
\begin{proof}
For item (i), given $t\geq\rho(\varepsilon)$, we have 
\[
\norm{x(t)-T(x(t))}\leq\min\left\{\frac{\gamma\phi(\varepsilon/2)}{2b},\frac{\varepsilon}{2}\right\}\text{ and }\norm{B(x(t))-B(y)}\leq\frac{\phi(\varepsilon/2)}{2b}
\]
using Lemmas \ref{lem:assym-reg-bot} or \ref{lem:assym-reg-botSecond}, respectively, and Proposition \ref{pro:FB-B-rate}. Noting Lemma \ref{lem:boundedness}, we get $\norm{x(t)-y}\leq \varepsilon$ for any $t\geq\rho(\varepsilon)$ by Lemma \ref{lem:FB-error}. Item (ii) is similarly a combination of Proposition \ref{pro:FB-B-rate}, Lemma \ref{lem:boundedness} and Lemma \ref{lem:FB-error}.
\end{proof}

If $A$ or $B$ are strongly monotone, that is
\[
\langle x-y,u-v\rangle \geq \rho\norm{x-y}^2
\]
for some $\rho>0$ and all $(x,u),(y,v)\in A$, or similarly for $B$, then they are immediately uniformly monotone with function $\phi(\varepsilon):=\rho\varepsilon^2$. In particular, the above rates apply in that case, complementing the rates given in \cite{BotCsetnek2018}. Moreover, using the previous result on fast rates for \ref{firstOrder} given by Theorem \ref{thm:firstOrderFast}, we can also derive exponentially fast rates under special parameter selections in that case. In particular, our result is slightly different from that given in \cite{BotCsetnek2018}, working over a slightly stronger (yet still canonical) assumption on the operators but for that following the traditional argument known from the discrete-time setting, leveraging the previous results of this section instead of relying on arguments related to differential inequalities:

\begin{theorem}
Let $X$ be a Hilbert space and let $A:X\to 2^X$ be maximally monotone and $B:X\to X$ be $\beta$-cocoercive for $\beta>0$. Let $\gamma\in (0,2\beta)$ and write $\delta:=\min\{1,\beta/\gamma\}+1/2$. Further, assume that $\lambda:[0,+\infty)\to [0,\delta]$ is Lebesgue measurable and fix $x_0\in X$. Let $x$ be the unique strong global solution to \ref{FB}. Further assume that $\underline{\tau}=\inf_{t\geq 0}\lambda(t)(\delta-\lambda(t))>0$ and that $A$ is strongly monotone with constant $\alpha$.

Then, for the unique point $y\in\mathrm{zer}(A+B)$, we have
\[
d(x(t),y)\leq c^{\floor*{t}}d(x(0),y)
\]
for all $t\geq 0$, where $c:=(1+\underline{\tau} k^p)^{-1/p}\in (0,1)$ and $k:=1/(\alpha\gamma+1)$.
\end{theorem}
\begin{proof}
It is well-known (see e.g.\ the proof of Proposition 26.16 in \cite{BauschkeCombettes2017}) that under the present assumptions, the map $T:=J_{\gamma A}\circ (\mathrm{Id}-\gamma B)$ is a strict contraction with parameter $k\in [0,1)$. The result now follows from Theorem \ref{thm:firstOrderFast}.
\end{proof}

Following Remark 15 in \cite{BotCsetnek2017}, also we want to remark here that in the same way as the Douglas–Rachford splitting method arises as an instantiation of the Krasnosel'skii-Mann method in discrete-time contexts, the above system \ref{firstOrder} gives rise to a continuous-time variant. This dynamical system can hence be studied using the present results, and we think that it is moreover quite likely that, following a similar approach as in \cite{TreuschKohlenbach2026}, one can also obtain particularly broad and uniform rates in that case.

\subsection{Second-order dynamical systems for cocoercive operators}

We again begin with the essential setup, now from \cite{BotCsetnek2016}. Let $X$  be a Hilbert space with norm $\norm{\cdot}$ and inner product $\langle\cdot,\cdot\rangle$ and let $B:X\to X$ be $\beta$-cocoercive, i.e.\ $\beta\norm{Bx-By}^2\leq \langle x-y,Bx-By\rangle$ for all $x,y\in X$. We write $\mathrm{Zer}B:=\{x\in X\mid B(x)=0\}$. Further, let $\lambda,\gamma:[0,+\infty)\to [0,+\infty)$ be Lebesgue measurable and $u_0,v_0\in X$ be given. For these data, we consider the dynamical system
\[
\begin{cases}
\ddot{x}(t)+\gamma(t)\dot{x}(t)+\lambda(t)B(x(t))=0,\\
x(0)=u_0, \dot{x}(0)=v_0.
\end{cases}\tag*{$(*)_2$}\label{secondOrder}
\]

This system arises naturally as an extension of the second-order dynamical system 
\[
\ddot{x}+\gamma\dot{x}+x-Tx=0
\]
formulated for a nonexpansive operator $T:X\to X$ as considered in \cite{AttouchAlvarez2000}, extending previous work on a gradient-projection second-order dynamical system
\[
\ddot{x}+\gamma\dot{x}+x-P_C(x-\eta\nabla f(x))=0
\]
studied by Antipin \cite{Antipin1994} in Euclidean spaces and by Attouch and Alvarez \cite{AttouchAlvarez2000} in general Hilbert spaces, which in turn is motivated by a study of the heavy ball with friction dynamical system
\[
\ddot{x}+\gamma\dot{x}+\nabla f (x)=0
\]
studied by various authors (see e.g.\ \cite{Alvarez2000,Antipin1994,AttouchGoudouRedont2000,HarauxJendoubi1998}, among others). We again refer to the survey \cite{Csetnek2020} for further discussions along these lines.

Similar to before (see also Definition 3 in \cite{BotCsetnek2016}), we call a dynamical system $x:[0,+\infty)\to X$ a strong global solution of \ref{secondOrder} if $x$ satisfies the respective equation almost everywhere on $[0,+\infty)$ and $x$ is absolutely continuous on each interval $[0,b]$ with $b\in (0,\infty)$. By Theorem 4 in \cite{BotCsetnek2016}, the system \ref{secondOrder} has a unique strong global solution under the assumptions above, provided that additionally $\lambda,\gamma\in L^1([0,b])$ for every $b>0$.

As discussed in \cite{BotCsetnek2016}, we also here want to highlight that a discretization of the above equation w.r.t.\ the time variable leads to methods which include inertia terms, that is which make use of the distance between the previous two iterates, which commonly yields accelerated schemas (we refer e.g.\ to \cite{Alvarez2000,AlvarezAttouch2001,MoudafiOliny2003} for works explicitly discussing this in connection to dynamical systems, as well as to the further references provided in \cite{BotCsetnek2016}).

Similar to before, we are now interested in providing a quantitative account of the results given in \cite{BotCsetnek2016} (see Theorem 8 therein) on the asymptotic behavior of such strong global solutions. For that, following \cite{BotCsetnek2016} (see assumption (A1) therein), we now throughout assume further that $\lambda$ and $\gamma$ are locally absolutely continuous and that there exists a $\theta>0$ such that
\[
\dot{\gamma}(t)\leq 0\leq \dot{\lambda}(t)\text{ and }\frac{\gamma^2(t)}{\lambda(t)}\geq\frac{1+\theta}{\beta}
\]
for almost every $t\in [0,+\infty)$. In particular, there are $0<\underline{\lambda},\overline{\gamma},\overline{\lambda},\underline{\gamma}$ such that $\underline{\lambda}\leq \lambda(t)\leq\overline{\lambda}$ and $\underline{\gamma}\leq \gamma(t)\leq\overline{\gamma}$. We refer to \cite{BotCsetnek2016} for further motivation and discussion on these conditions, including natural examples of parameters satisfying them, and here just note that in such a case the solution to the above system \ref{secondOrder} is a $C^2$-function (see p.\ 1427 in \cite{BotCsetnek2016}).

We now begin with the boundedness of both the solution dynamical system as well as the associated derivatives and the image under $B$ (as established in item (i) of \cite[Theorem 8]{BotCsetnek2016}). We begin with the boundedness of the dynamical system itself:

\begin{lemma}\label{lem:SecondBoundedness}
Let $z\in\mathrm{Zer}B$ and let $b,c\in\mathbb{N}$ with $\norm{u_0-z}\leq b$ and $\norm{v_0}\leq c$. Let $x$ be the unique strong global solution to \ref{secondOrder}. Then $x$ is bounded with $\norm{x(t)-z}\leq \sqrt{b^2+2\underline{\gamma}^{-1}M}$, where we set $M:=bc+\frac{\overline{\gamma}}{2}b^2+\beta\overline{\gamma}\underline{\lambda}^{-1}c^2$.
\end{lemma}
\begin{proof}
Write $h(t)=\frac{1}{2}\norm{x(t)-z}^2$. As in the proof of \cite[Theorem 8]{BotCsetnek2016}, we have $\dot{h}(t)=\langle x(t)-z,\dot{x}(t)\rangle$ and it follows from Eq.\ (13) therein that
\begin{align*}
\dot{h}(t)+\gamma(t)h(t)+\beta\frac{\gamma(t)}{\lambda(t)}\norm{\dot{x}(t)}^2&\leq \dot{h}(0)+\gamma(0)h(0)+\beta\frac{\gamma(0)}{\lambda(0)}\norm{\dot{x}(0)}^2\\
&\leq \norm{u_0-z}\norm{v_0}+\frac{\overline{\gamma}}{2}\norm{u_0-z}^2+\beta\overline{\gamma}\underline{\lambda}^{-1}\norm{v_0}^2\\
&\leq bc+\frac{\overline{\gamma}}{2}b^2+\beta\overline{\gamma}\underline{\lambda}^{-1}c^2=M.
\end{align*}
In particular, we have $\dot{h}(t)+\gamma(t)h(t)\leq K$. Continuing as in the proof of \cite[Theorem 8]{BotCsetnek2016} until Eq.\ (14), we get that
\[
h(T)\leq h(0)\exp(-\underline{\gamma}T)+M\underline{\gamma}^{-1}(1-\exp(-\underline{\gamma}T))
\]
for all $T>0$, and so
\[
h(t)\leq \frac{1}{2}b^2+\underline{\gamma}^{-1}M.
\]
In particular, we have
\[
\norm{x(t)-z}\leq \sqrt{b^2+2\underline{\gamma}^{-1}M}.\qedhere
\]
\end{proof}

Before we carry on with the boundedness of the derivative, we require the following simple lemma on real inequalities:

\begin{lemma}\label{lem:SecondRealIneq}
Let $a,b,c> 0$ be given. For any $x\geq 0$, if $ax^2\leq bx+c$, then $x \leq \ceil{\frac{b+c}{a}}$.
\end{lemma}
\begin{proof}
Suppose for a contradiction that $x> \ceil{\frac{b+c}{a}}\geq 1$. Then
\[
ax^2\leq bx+c\leq (b+c)x
\]
and so $ax\leq b+c$. But this is a contradiction as then $b+c\leq a\ceil{\frac{b+c}{a}}< ax\leq b+c$.
\end{proof}

The boundedness of the first derivative is now readily derived:

\begin{lemma}\label{lem:SecondDerBoundedness}
Let $z\in\mathrm{Zer}B$ and let $b,c\in\mathbb{N}$ with $\norm{u_0-z}\leq b$ and $\norm{v_0}\leq c$. Let $x$ be the unique strong global solution to \ref{secondOrder}. Then $\dot{x}$ is bounded with $\norm{\dot{x}(t)}\leq \ceil*{\frac{K+M}{\beta\underline{\gamma}\overline{\lambda}^{-1}}}$, where we set $M:=bc+\frac{\overline{\gamma}}{2}b^2+\beta\overline{\gamma}\underline{\lambda}^{-1}c^2$ and $K:=\sqrt{b^2+2\underline{\gamma}^{-1}M}$.
\end{lemma}
\begin{proof}
As in the proof of \cite[Theorem 8]{BotCsetnek2016} (see the last equation before Eq.\ (16)), we derive
\[
\langle x(t)-z,\dot{x}(t)\rangle +\beta\underline{\gamma}\overline{\lambda}^{-1}\norm{\dot{x}(t)}^2\leq M
\]
and so
\[
\beta\underline{\gamma}\overline{\lambda}^{-1}\norm{\dot{x}(t)}^2\leq M + \norm{x(t)-z}\norm{\dot{x}(t)}\leq M+K\norm{\dot{x}(t)}
\]
by Lemma \ref{lem:SecondBoundedness}. Lemma \ref{lem:SecondRealIneq} yields
\[
\norm{\dot{x}(t)}\leq \ceil*{\frac{K+M}{\beta\underline{\gamma}\overline{\lambda}^{-1}}}
\]
as claimed.
\end{proof}

We now turn to the $L^2$-boundedness of the derivatives and the image under the operator $B$:

\begin{lemma}\label{lem:SecondIntBoundedness}
Let $z\in\mathrm{Zer}B$ and let $b,c\in\mathbb{N}$ with $\norm{u_0-z}\leq b$ and $\norm{v_0}\leq c$. Let $x$ be the unique strong global solution to \ref{secondOrder}. Then 
\[
\norm{\dot{x}}_2\leq \sqrt{\frac{KL}{\theta}}, \quad\norm{\ddot{x}}_2\leq \sqrt{\frac{KL}{\beta\overline{\lambda}^{-1}}}, \quad\norm{Bx}_2\leq \underline{\lambda}^{-1}(\norm{\ddot{x}}_2+\overline{\gamma}\norm{\dot{x}}_2)
\]
where we set $M:=bc+\frac{\overline{\gamma}}{2}b^2+\beta\overline{\gamma}\underline{\lambda}^{-1}c^2$,  $K:=\sqrt{b^2+2\underline{\gamma}^{-1}M}$ and $L:=\ceil*{(K+M)\beta\underline{\gamma}\overline{\lambda}^{-1}}$.
\end{lemma}
\begin{proof}
Write $h(t)=\frac{1}{2}\norm{x(t)-z}^2$ and recall $\dot{h}(t)=\langle x(t)-z,\dot{x}(t)\rangle$. As in Eq.\ (12) in the proof of \cite[Theorem 8]{BotCsetnek2016}, we derive
\[
\ddot{h}(t)+\frac{d}{dt}(\gamma h)(t)+\beta\frac{d}{dt}\left(\frac{\gamma(t)}{\lambda(t)}\norm{\dot{x}(t)}^2\right)+\theta\norm{\dot{x}(t)}^2+\beta\overline{\lambda}^{-1}\norm{\ddot{x}(t)}^2\leq 0
\]
so that integration yields
\[
\theta\int_0^t\norm{\dot{x}(s)}^2\,ds+\beta\overline{\lambda}^{-1}\int_0^t\norm{\ddot{x}(s)}^2\,ds\leq \vert \dot{h}(t)\vert \leq \norm{x(t)-z}\norm{\dot{x}(t)}\leq KL
\]
for all $t\geq 0$. This yields the bounds on the $L^2$-norms for $\dot{x}$ and $\ddot{x}$. For $Bx$, note that as $x$ satisfies the equation \ref{secondOrder}, we have
\[
\underline{\lambda} \norm{Bx}_2\leq \norm{\lambda Bx}_2\leq \norm{\ddot{x}}_2+\norm{\gamma\dot{x}}_2\leq \norm{\ddot{x}}_2+\overline{\gamma}\norm{\dot{x}}_2
\]
using Minkowski's inequality, which yields the claim.
\end{proof}

To capture the corresponding problem in the setup of the previous sections, we set $X_0=\overline{B}_{K}(z)$ and 
\[
F(x^*)=\begin{cases}\norm{B(x^*)}&\text{if }x^*\in X_0,\\+\infty&\text{otherwise},\end{cases}
\]
for a fixed solution $z\in\mathrm{Zer}B$, and where $K$ is as in Lemma \ref{lem:SecondDerBoundedness}. In particular, we have $\mathrm{zer}F=\mathrm{Zer}B\cap X_0$ and 
\[
\mathrm{lev}_{\leq\varepsilon} F=\{x^*\in X_0\mid \norm{B(x^*)}\leq \varepsilon\}.
\]

Our next step is a quantitative result on the asymptotic regularity of the dynamical system. Concretely, in item (ii) of Theorem 8 in \cite{BotCsetnek2016}, it is established that 
\[
\lim_{t\to\infty}\norm{\dot{x}(t)}=\lim_{t\to\infty}\norm{\ddot{x}(t)}=\lim_{t\to\infty}\norm{B(x(t))}=0.
\]
We now provide quantitative versions of these results by giving rates of metastability for each case. Our quantitative rendering of the limit $\lim_{t\to\infty}B(x(t))=0$ then in particular allows us to derive the $\liminf$-property w.r.t.\ the above $F$ for the dynamical system $x$, together with a respective bound for it.

\begin{lemma}\label{lem:SecondIntMeta}
Let $z\in\mathrm{Zer}B$ and let $b,c,d\in\mathbb{N}$ with $\norm{u_0-z}\leq b$, $\norm{v_0}\leq c$ and $\norm{B(u_0)}\leq d$. Let $x$ be the unique strong global solution to \ref{secondOrder}. Then 
\[
\lim_{t\to\infty}\dot{x}(t)=\lim_{t\to\infty}\ddot{x}(t)=\lim_{t\to\infty}B(x(t))= 0 
\]
where further, for any $f:\mathbb{N}\to\mathbb{N}$ and $\varepsilon>0$, we have
\[
\exists n\leq \Lambda(\varepsilon,f)\ \forall t\in[n,n+f(n)]\left( \norm{\dot{x}(t)},\norm{B(x(t))}<\varepsilon\right)
\]
as well as
\[
\exists n\leq \Lambda(\varepsilon/2\max\{\overline{\gamma},\overline{\lambda}\},f)\ \forall t\in[n,n+f(n)]\left( \norm{\ddot{x}(t)}<\varepsilon\right).
\]
Here, $\Lambda(\varepsilon,f):=\tilde{f'}^{(\varpi_{C,A,B}(\varepsilon))}(0)$ where $\varpi_{C,A,B}$ is defined via 
\begin{equation*}
\varpi_{C,A,B}(\varepsilon):=\left\lceil\frac{16AB}{3\varepsilon^2}\right\rceil\cdot\left\lceil\frac{4(C^2+2AB)}{3\varepsilon^2}\right\rceil,
\end{equation*}
where $\tilde{f'}(n):=n+f'(n)$ as before, with $f'(n):=\max\{f(n),\ceil*{(3A/\varepsilon)^2}\}$, and where we set $M:=bc+\frac{\overline{\gamma}}{2}b^2+\beta\overline{\gamma}\underline{\lambda}^{-1}c^2$,  $K:=\sqrt{b^2+2\underline{\gamma}^{-1}M}$ and $L:=\ceil*{(K+M)\beta\underline{\gamma}\overline{\lambda}^{-1}}$ as well as $a_0:=\sqrt{KL/\theta}$, $a_1:=\sqrt{KL/\beta\overline{\lambda}^{-1}}$ and $a_2:=\underline{\lambda}^{-1}(a_1+\overline{\gamma}a_0)$
together with $A=\frac{1}{2}\max\{a_0,a_2\}$, $B=\frac{1}{2}\max\{a_0+a_1,a_2+\frac{1}{\beta^2}a_0\}$ and $C=\frac{1}{2}\max\{c,d\}$.

In particular, $x$ has the $\liminf$-property w.r.t.\ $F$ with bound
\[
\varphi(\varepsilon,n):=\varpi_{C,A,B}(\varepsilon)\max\{n,\ceil*{(3A/\varepsilon)^2}\}.
\]
\end{lemma}
\begin{proof}
Note that as in the proof of item (ii) of Theorem 8 in \cite{BotCsetnek2016}, we have
\[
\frac{d}{dt}\left(\frac{1}{2}\norm{\dot{x}(t)}^2\right)\leq \frac{1}{2}\norm{\dot{x}(t)}^2+\frac{1}{2}\norm{\ddot{x}(t)}^2
\]
as well as
\[
\frac{d}{dt}\left(\frac{1}{2}\norm{B(x(t))}^2\right)\leq \frac{1}{2}\norm{B(x(t))}^2+\frac{1}{2\beta^2}\norm{\dot{x}(t)}^2.
\]
Lemma \ref{lem:SecondIntBoundedness} now in particular yields $\norm{\dot{x}}_2\leq a_0$, $\norm{\ddot{x}}_2\leq a_1$ and $\norm{Bx}_2\leq a_2$. The two respective rates for $\lim_{t\to\infty}\dot{x}(t)=\lim_{t\to\infty}B(x(t))= 0$ now follow from Lemma \ref{lem:AAS2}. For the rate for $\ddot{x}$, note similar to before that $x$ satisfies the equation \ref{secondOrder} and so we have
\[
\norm{\ddot{x}(t)}\leq \overline{\gamma}\norm{\dot{x}(t)}+\overline{\lambda}\norm{B(x(t))}.
\]
We obtain the $\liminf$-property by setting $f:\equiv n$ for a given $n\in\mathbb{N}$ in the above rate of metastability for $\lim_{t\to\infty}B(x(t))= 0$.
\end{proof}

In similarity to the previous application, we now also obtain the quasi-Fej\'er monotonicity of the unique strong global solution $x$ to the system \ref{secondOrder} w.r.t.\ $F$. Again, this follows the pattern discussed in Section \ref{sec:motivation}, that is through a differential inequality for a suitable energy functional which decomposes in the right way. Indeed, in \cite{BotCsetnek2016} (see eq.\ (12) therein), it is established that
\[
\ddot{h}(t)+\frac{d}{dt}(\gamma h)(t)+\beta\frac{d}{dt}\left(\frac{\gamma(t)}{\lambda(t)}\norm{\dot{x}(t)}^2\right)\leq 0,
\]
where again $h(t)=\frac{1}{2}\norm{x(t)-z}^2$ for a solution $z\in\mathrm{Zer}B$. In that way, distinct from the previous application, this inequality is no longer of the most simple form discussed in Section \ref{sec:motivation}. Nevertheless, it fits into another pattern discussed therein, whereby the energy functional additively decomposes into a distance to a solution as well as associated errors terms. We now analyse the proof of this inequality as given in \cite{BotCsetnek2016} to establish the associated uniform quasi-Fej\'er monotonicity and extract an associated modulus. The following result first provides a quantitative variant of the above differential inequality:

\begin{lemma}\label{lem:SecondApproximateDerivative}
Let $z\in\mathrm{Zer}B$ and let $b,c\in\mathbb{N}$ with $\norm{u_0-z}\leq b$ and $\norm{v_0}\leq c$. Let $x$ be the unique strong global solution to \ref{secondOrder}. For any $\varepsilon>0$, if $x^*\in\mathrm{lev}_{\leq \varepsilon/(\overline{\lambda}4K)}F$, then 
\[
\ddot{h}(t)+\frac{d}{dt}(\gamma h)(t)+\beta\frac{d}{dt}\left(\frac{\gamma(t)}{\lambda(t)}\norm{\dot{x}(t)}^2\right)\leq \varepsilon
\]
for all $t\geq 0$, where $h(t):=\frac{1}{2}\norm{x(t)-x^*}^2$. Here $M:=bc+\frac{\overline{\gamma}}{2}b^2+\beta\overline{\gamma}\underline{\lambda}^{-1}c^2$ and $K:=\sqrt{b^2+2\underline{\gamma}^{-1}M}$.
\end{lemma}
\begin{proof}
Given $x^*\in\mathrm{lev}_{\leq  r(\varepsilon)}F$, write $h(t)=\frac{1}{2}\norm{x(t)-x^*}^2$. As in the proof of item (i) of Theorem 8 in \cite{BotCsetnek2016}, we get
\[
\ddot{h}(t)+\gamma(t)\dot{h}(t)+\lambda(t)\langle x(t)-x^*,B(x(t))\rangle=\norm{\dot{x}(t)}^2.
\]
In particular, we have
\[
\ddot{h}(t)+\gamma(t)\dot{h}(t)+\lambda(t)\langle x(t)-x^*,B(x(t))-B(x^*)\rangle+\lambda(t)\langle x(t)-x^*,B(x^*)\rangle=\norm{\dot{x}(t)}^2
\]
and so, using the cocoercivity of $B$ and Lemma \ref{lem:SecondBoundedness}, we get
\[
\ddot{h}(t)+\gamma(t)\dot{h}(t)+\lambda(t)\beta\norm{B(x(t))-B(x^*)}^2\leq\norm{\dot{x}(t)}^2+\overline{\lambda}2K\norm{B(x^*)}.
\]
Further, using the reverse triangle inequality, we have
\[
\ddot{h}(t)+\gamma(t)\dot{h}(t)+\lambda(t)\beta\norm{B(x(t))}^2\leq\norm{\dot{x}(t)}^2+\overline{\lambda}2K\norm{B(x^*)}+2\lambda(t)\beta\norm{B(x(t))}\norm{B(x^*)}.
\]
As $B$ is $\frac{1}{\beta}$-Lipschitz and as $z$ is a zero of $B$, we get
\[
\norm{B(x(t))}=\norm{B(x(t))-B(z)}\leq \frac{1}{\beta}\norm{x(t)-z}\leq \frac{K}{\beta}
\]
using Lemma \ref{lem:SecondBoundedness}. In particular, we have
\[
\ddot{h}(t)+\gamma(t)\dot{h}(t)+\lambda(t)\beta\norm{B(x(t))}^2\leq\norm{\dot{x}(t)}^2+\overline{\lambda}4K\norm{B(x^*)}.
\]
Proceeding as in the proof of item (i) of Theorem 8 in \cite{BotCsetnek2016}, we get
\[
\ddot{h}(t)+\frac{d}{dt}(\gamma h)(t)+\beta\frac{d}{dt}\left(\frac{\gamma(t)}{\lambda(t)}\norm{\dot{x}(t)}^2\right)\leq \overline{\lambda}4K\norm{B(x^*)}
\]
which yields the claim.
\end{proof}

Again, an application of the mean value theorem allows us to derive a modulus of uniform $(G,H)$-quasi-Fej\'er monotonicity (with $H=(\cdot)^2$ and $G=\overline{\gamma}\underline{\gamma}^{-1}(\cdot)^2$) from the above approximate differential inequality.

\begin{lemma}\label{lem:secondOrderUniFejer}
Let $z\in\mathrm{Zer}B$ and let $b,c\in\mathbb{N}$ with $\norm{u_0-z}\leq b$ and $\norm{v_0}\leq c$. Let $x$ be the unique strong global solution to \ref{secondOrder}. Then $x$ is uniformly $(G,H)$-quasi-Fej\'er monotone w.r.t.\ $\mathrm{zer}F$ and $F$ and with $H=(\cdot)^2$ and $G=\overline{\gamma}\underline{\gamma}^{-1}(\cdot)^2$, i.e.
\begin{gather*}
\forall \varepsilon>0\ \forall n,m\in\mathbb{N}\ \forall x^*\in \mathrm{lev}_{\leq\chi(\varepsilon,n,m)}F\ \forall s,t\in [n,n+m]\\
\left(s\leq t \rightarrow \norm{x(t)-x^*}^2\leq \overline{\gamma}\underline{\gamma}^{-1}\norm{x(s)-x^*}^2+\varepsilon(s,t)+\varepsilon\right),
\end{gather*}
where $\chi(\varepsilon,n,m):=(\underline{\gamma}\varepsilon)/(m\overline{\lambda}8K)$ and
\[
\varepsilon(s,t)=2\beta\underline{\gamma}^{-1}\frac{\overline{\gamma}}{\underline{\lambda}}\norm{\dot{x}(s)}^2+2\underline{\gamma}^{-1}K\norm{\dot{x}(s)}+2\underline{\gamma}^{-1}K\norm{\dot{x}(t)}.
\]
Here $M:=bc+\frac{\overline{\gamma}}{2}b^2+\beta\overline{\gamma}\underline{\lambda}^{-1}c^2$ and $K:=\sqrt{b^2+2\underline{\gamma}^{-1}M}$.
\end{lemma}
\begin{proof}
Fix $\varepsilon,n,m$ as well as $x^*\in \mathrm{lev}_{\leq\chi(\varepsilon,n,m)}F$ and $s,t\in [n,n+m]$. If $s=t$, the result is clear. If $s<t$, then since $\dot{h}(t)+\gamma(t) h(t)+\beta\frac{\gamma(t)}{\lambda(t)}\norm{\dot{x}(t)}^2$ is locally absolutely continuous, the mean value theorem yields an $r\in (s,t)$ such that
\begin{align*}
&\ddot{h}(r)+\frac{d}{dt}(\gamma h)(r)+\beta\frac{d}{dt}\left(\frac{\gamma(r)}{\lambda(r)}\norm{\dot{x}(r)}^2\right)\\
&\qquad=\frac{d}{dt}\left(\dot{h}(r)+\gamma(r) h(r)+\beta\frac{\gamma(r)}{\lambda(r)}\norm{\dot{x}(r)}^2\right)\\
&\qquad=\frac{\dot{h}(t)+\gamma(t) h(t)+\beta\frac{\gamma(t)}{\lambda(t)}\norm{\dot{x}(t)}^2-(\dot{h}(s)+\gamma(s) h(s)+\beta\frac{\gamma(s)}{\lambda(s)}\norm{\dot{x}(s)}^2)}{t-s}.
\end{align*}
Thus, as $x^*\in \mathrm{lev}_{\leq\chi(\varepsilon,n,m)}F$, we have
\begin{align*}
\gamma(t) h(t)-\gamma(s) h(s)&\leq\beta\frac{\gamma(s)}{\lambda(s)}\norm{\dot{x}(s)}^2-\beta\frac{\gamma(t)}{\lambda(t)}\norm{\dot{x}(t)}^2+\dot{h}(s)-\dot{h}(t)+\varepsilon\\
&\leq\beta\frac{\overline{\gamma}}{\underline{\lambda}}\norm{\dot{x}(s)}^2+K\norm{\dot{x}(s)}+K\norm{\dot{x}(t)}+\underline{\gamma}\frac{\varepsilon}{2}
\end{align*}
by Lemma \ref{lem:approximateDerivative}. In particular, we get
\[
h(t)\leq \overline{\gamma}\underline{\gamma}^{-1}h(s)+\beta\underline{\gamma}^{-1}\frac{\overline{\gamma}}{\underline{\lambda}}\norm{\dot{x}(s)}^2+\underline{\gamma}^{-1}K\norm{\dot{x}(s)}+\underline{\gamma}^{-1}K\norm{\dot{x}(t)}+\frac{\varepsilon}{2}
\]
which yields the claim.
\end{proof}

Another direct application of Theorem \ref{thm:generalFejer-rates} now yields the following quantitative result on the behavior of the solution $x$ to \ref{secondOrder}, again under a finite-dimensionality assumption on $X$, giving a rate of metastability for the convergence result established in Theorem 8 of \cite{BotCsetnek2016} in that case:

\begin{theorem}\label{thm:secondOrderMeta}
Let $X$ be a finite-dimensional Hilbert space with dimension $d$ and let $B:X\to X$ be $\beta$-cocoercive. Let $u_0,v_0\in X$ be given and $z\in\mathrm{Zer}B$ and let $b,c,d\in\mathbb{N}$ with $\norm{u_0-z}\leq b$, $\norm{v_0}\leq c$ and $\norm{B(u_0)}\leq d$. Further, let $\lambda,\gamma:[0,+\infty)\to [0,+\infty)$ be Lebesgue measurable and locally absolutely continuous and assume that there exists a $\theta>0$ such that
\[
\dot{\gamma}(t)\leq 0\leq \dot{\lambda}(t)\text{ and }\frac{\gamma^2(t)}{\lambda(t)}\geq\frac{1+\theta}{\beta}
\]
for almost every $t\in [0,+\infty)$. Let $0<\underline{\lambda},\overline{\gamma},\overline{\lambda},\underline{\gamma}$ such that $\underline{\lambda}\leq \lambda(t)\leq\overline{\lambda}$ and $\underline{\gamma}\leq \gamma(t)\leq\overline{\gamma}$. Let $x$ be the unique strong global solution to \ref{secondOrder}.

Then $x$ is metastable in the sense that for all $\varepsilon>0$ and $f:\mathbb{N}\to\mathbb{N}$, it holds that 
\[
\exists n\leq \Delta(\min\{\varepsilon,\beta\varepsilon/2\},f)\ \forall s,t\in [n,n+f(n)]\left( \norm{x(s)-x(t)}\leq\varepsilon\text{ and }\norm{B(x(t))}\leq\varepsilon\right),
\]
where moreover the bound $\Delta(\varepsilon,f)$ can be explicitly described by
\[
\Delta(\varepsilon,f):=\max\{\Delta(i,\varepsilon,f)\mid i\leq P\}+1
\]
where $\Delta(0,\varepsilon,f):=0$ as well as
\[
\Delta(j,\varepsilon,f):=\max\{\varphi(\widehat{\varepsilon}_j,n)\mid n\leq \eta(\varepsilon^2/12,f_{\varphi,\widehat{\varepsilon}_j})\}
\]
with
\[
\widehat{\varepsilon}_j=\min\{\varepsilon/2,(\underline{\gamma}\varepsilon^2)/12((f(m+1)+1)\overline{\lambda}8K)\mid m\leq \Delta(i,\varepsilon,f),i<j\}
\]
for $j\geq 1$ as well as
\[
f_{\varphi,\varepsilon}(n):=\max\{m+1+f(m+1)\mid m\leq \varphi(\varepsilon,n)\}\dotdiv n
\]
and
\[
P:=\ceil*{2\left(\ceil*{\left(\varepsilon/\sqrt{12\overline{\gamma}\underline{\gamma}^{-1}}\right)^{-1}}+1\right)\sqrt{d}b}^d+1
\]
as well as
\[
\eta(\varepsilon,f):=\Lambda\left(\min\left\{\frac{\varepsilon}{6\underline{\gamma}^{-1}K},\sqrt{\frac{\varepsilon}{6\beta\underline{\gamma}^{-1}\frac{\overline{\gamma}}{\underline{\lambda}}}}\right\},f\right)\text{ and }\varphi(\varepsilon,n):=\varpi_{C,A,B}(\varepsilon)\max\{n,\ceil*{(3A/\varepsilon)^2}\}.
\]
Here, $\Lambda(\varepsilon,f):=\tilde{f'}^{(\varpi_{C,A,B}(\varepsilon))}(0)$ and $\varpi_{C,A,B}$ as well as the other constants are defined as in Lemma \ref{lem:SecondIntMeta}. 
\end{theorem}
\begin{proof}
The result immediately follows by instantiating Theorem \ref{thm:generalFejer-uc} over $X_0=\overline{B}_{K}(z)$ and simplifying the resulting bound (recall that $x(t)\in X_0$ for all $t\geq 0$ by Lemma \ref{lem:SecondBoundedness}). In particular, we set $H:=(\cdot)^2$ and $G:=\overline{\gamma}\underline{\gamma}^{-1}(\cdot)^2$ as well as $h(\varepsilon):=\varepsilon^2$ and $g(\varepsilon):=\sqrt{\varepsilon/\overline{\gamma}\underline{\gamma}^{-1}}$. That $x$ is uniformly $(G,H)$-quasi-Fej\'er monotone w.r.t.\ $F$ and errors 
\[
\varepsilon(s,t)=2\beta\underline{\gamma}^{-1}\frac{\overline{\gamma}}{\underline{\lambda}}\norm{\dot{x}(s)}^2+2\underline{\gamma}^{-1}K\norm{\dot{x}(s)}+2\underline{\gamma}^{-1}K\norm{\dot{x}(t)}
\]
follows by Lemma \ref{lem:secondOrderUniFejer}, which in particular yields that $\chi(\varepsilon,n,m):=(\underline{\gamma}\varepsilon)/(m\overline{\lambda}8K)$ is an associated modulus. The definition of $\widehat{\varepsilon}_j$ is a resulting simplification. The errors in particular satisfy $\varepsilon(s,t)\to 0$ as $s\leq t\to \infty$ with a rate of metastability $\eta$ as defined above by Lemma \ref{lem:SecondIntMeta}. Also by Lemma \ref{lem:SecondIntMeta}, $x$ has the $\liminf$-property w.r.t.\ $F$ with a bound $\varphi(\varepsilon,n):=\varpi_{C,A,B}(\varepsilon)\max\{n,\ceil*{(3A/\varepsilon)^2}\}$, which can immediately be seen to be suitably monotone. As before, $X_0$ is totally bounded with a modulus $\gamma(\varepsilon):=\ceil*{2(\ceil*{\varepsilon^{-1}}+1)\sqrt{d}b}^d$ by Example 2.8 of \cite{KohlenbachLeusteanNicolae2018}. 
Lastly, note that $F$ is uniformly continuous on $X_0$ with modulus $\omega(\varepsilon):=\beta\varepsilon$, since we have $\vert \norm{B(x)}-\norm{B(y)}\vert\leq \frac{1}{\beta}\norm{x-y}$ using the fact that $B$ is $\frac{1}{\beta}$-Lipschitz.
\end{proof}

In the preceding application, Theorem \ref{thm:met-reg-rates} yielded rather general constructions for rates of convergence under a regularity assumption. While Theorem \ref{thm:met-reg-rates} can similarly be applied to the present case study to yield quantitative convergence guarantees for the strong convergence of the dynamical system under a regularity assumption, it does not yield rates in this general case.

This is due to the fact that we cannot supply the respective errors in the quasi-Fej\'er monotonicity property for the dynamical system at hand with a rate of convergence, but only with a rate of metastability. While this does not rule out constructions of rates of convergence for problem \ref{secondOrder} under regularity assumptions, it perhaps provides a structural reason for why such rates seem to be currently absent from the literature (to the best of our knowledge) at this level of generality. This issue seems to persist already for special cases of the above system \ref{secondOrder}, such as the gradient-projection second-order dynamical system  \cite{AttouchAlvarez2000,Antipin1994} or its extensions to general nonexpansive mappings as discussed at the beginning of this subsection (indeed, the paper \cite{BotCsetnek2016} uses essentially the same approach as \cite{AttouchAlvarez2000}, appropriately extended to its more general setting).

By comparison, Bo\c{t} and Csetnek \cite{BotCsetnek2018} derive exponentially fast rates for the second-order forward-backward scheme, arising from the above by instantiating $B$ appropriately and discussed in more detail in the following, under a strong monotonicity assumption on either of the involved operators and under suitable restrictions of the other parameters. Indeed, their argument essentially proceeds by utilizing the extremely strong premise of strong monotonicity to derive an alternative differential inequality that dispenses of the error induced by $\norm{\dot{x}(t)}$ above. However, this inequality is then already so strong that it allows one to immediately derive exponential rates, so that no additional arguments beyond that are necessary. While it can be rather immediately seen that these arguments can be slightly generalized to the setting of \ref{secondOrder} for general $B$ which are strict contractions, we do not record this here as no new arguments from this paper would be required for that. Instead, we focus on the general case of \ref{secondOrder} which however then restricts us to metastability.

Of course, by the modularity of our approach, if one would be supplied by such a rate by any other means, then Theorem \ref{thm:met-reg-rates} could be immediately used to derive corresponding rates for the strong convergence of the dynamical system under a regularity assumption. Outside of assumptions guaranteeing this however, the following theorem is the best we can currently provide.

\begin{theorem}
Let $X$ be a Hilbert space and let $B:X\to X$ be $\beta$-cocoercive. Let $u_0,v_0\in X$ be given and $z\in\mathrm{Zer}B$ and let $b,c,d\in\mathbb{N}$ with $\norm{u_0-z}\leq b$, $\norm{v_0}\leq c$ and $\norm{B(u_0)}\leq d$. Assume that $F$ has a modulus of regularity $\tau$ w.r.t.\ $\overline{B}_K(z)$. Further, let $\lambda,\gamma:[0,+\infty)\to [0,+\infty)$ be Lebesgue measurable and locally absolutely continuous and assume that there exists a $\theta>0$ such that
\[
\dot{\gamma}(t)\leq 0\leq \dot{\lambda}(t)\text{ and }\frac{\gamma^2(t)}{\lambda(t)}\geq\frac{1+\theta}{\beta}
\]
for almost every $t\in [0,+\infty)$. Let $0<\underline{\lambda},\overline{\gamma},\overline{\lambda},\underline{\gamma}$ such that $\underline{\lambda}\leq \lambda(t)\leq\overline{\lambda}$ and $\underline{\gamma}\leq \gamma(t)\leq\overline{\gamma}$. Let $x$ be the unique strong global solution to \ref{secondOrder}.

Then $x$ satisfies
\[
\forall \varepsilon>0\ \forall f:\mathbb{N}\to\mathbb{N}\ \exists n\leq \rho(\varepsilon,f)\ \forall t\in [n,n+f(n)]\left( \mathrm{dist}(x(t),\mathrm{Zer}B)<\varepsilon\right)
\]
and further $x$ is Cauchy with
\[
\forall \varepsilon>0\ \forall f:\mathbb{N}\to\mathbb{N}\ \exists n\leq \rho(\varepsilon/2,f)\ \forall t\in [n,n+f(n)]\left( \norm{x(t)-x(s)}<\varepsilon\right).
\]

Here, $\rho$ is defined by
\[
\rho(\varepsilon,f):=\max\left\{\varphi\left(\tau\left(\varepsilon/\sqrt{2\overline{\gamma}\underline{\gamma}^{-1}}\right),n\right)\mid n\leq \eta(\varepsilon^2,f_{\varphi,\varepsilon^2})\right\}+1
\]
where 
\[
f_{\varphi,\delta}(n)=\max\left\{m+1+f(m+1)\;\big\vert\; m\leq\varphi\left(\tau\left(\varepsilon/\sqrt{2\overline{\gamma}\underline{\gamma}^{-1}}\right),n\right)\right\}\dotdiv n
\]
with
\[
\eta(\varepsilon,f):=\Lambda\left(\min\left\{\frac{\varepsilon}{6\underline{\gamma}^{-1}K},\sqrt{\frac{\varepsilon}{6\beta\underline{\gamma}^{-1}\frac{\overline{\gamma}}{\underline{\lambda}}}}\right\},f\right)\text{ and }\varphi(\varepsilon,n):=\varpi_{C,A,B}(\varepsilon)\max\{n,\ceil*{(3A/\varepsilon)^2}\}.
\]
Here, $\Lambda(\varepsilon,f):=\tilde{f'}^{(\varpi_{C,A,B}(\varepsilon))}(0)$ and $\varpi_{C,A,B}$ as well as the other constants are defined as in Lemma \ref{lem:SecondIntMeta}. 
\end{theorem}
\begin{proof}
The result immediately follows by instantiating Theorem \ref{thm:met-reg} and simplifying the resulting bound. In particular, we set $H:=(\cdot)^2$ and $G:=\overline{\gamma}\underline{\gamma}^{-1}(\cdot)^2$ as well as $h(\varepsilon):=\varepsilon^2$ and $g(\varepsilon):=\sqrt{\varepsilon/\overline{\gamma}\underline{\gamma}^{-1}}$. That $x$ is $(G,H)$-quasi-Fej\'er monotone w.r.t.\ $F$ follows by Lemma \ref{lem:secondOrderUniFejer}. That $x$ has the $\liminf$-property w.r.t.\ $F$ with $\varphi$ as the respective bound follows from Lemma \ref{lem:SecondIntMeta}. The rate of metastability $\eta$ for the errors likewise follows from that lemma as before.
\end{proof}

Similar to before, we here want to also study a particular instance of the system \ref{secondOrder} to which the above results apply. For that consider a continuous-time variant of the forward-backward method, this time defined via a second order equation as in \ref{secondOrder}, as considered already in \cite{BotCsetnek2016}. Let us recall the setup: Assume we are given a maximally monotone set-valued operator $A:X\to 2^X$ and a single-valued map $B:X\to X$ which is $\beta$-cocoercive, for some $\beta>0$, such that $\mathrm{zer}(A+B)\neq\emptyset$. For $\eta\in (0,2\beta)$ (following the notation of \cite{BotCsetnek2016}) and $\delta:=\frac{4\beta-\eta}{2\beta}$, let $\lambda,\gamma:[0,+\infty)\to (0,+\infty)$ be Lebesgue measurable functions which are locally absolutely continuous and satisfy
\[
\dot{\gamma}(t)\leq 0\leq\dot{\lambda}(t)\text{ and }\frac{\gamma^2(t)}{\lambda(t)}\geq\frac{2(1+\theta)}{\delta}
\]
almost everywhere, for some $\theta>0$. For initial values $u_0,v_0\in X$, we then consider the dynamical system 
\[
\begin{cases}
\ddot{x}(t)+\gamma(t)\dot{x}(t)+\lambda(t)\left[x(t)-J_{\eta A}\left(x(t)-\eta B(x(t))\right)\right]=0,\\
x(0)=u_0, \dot{x}(0)=v_0.
\end{cases}\tag*{$(*)'_2$}\label{SFB}
\]
In analogy to before, this system can be conceived of as an instantiation of the system \ref{secondOrder} using the map $T:=J_{\eta A}\circ (\mathrm{Id}-\eta B)$, which is $\frac{1}{\delta}$-averaged and satisfies $\mathrm{Fix}T=\mathrm{zer}(A+B)$ as before (see the beginning of the proof of Theorem 12 in \cite{BotCsetnek2016}). So writing 
\[
T=\left(1-\frac{1}{\delta}\right)\mathrm{Id}+\frac{1}{\delta}R
\]
with $R$ nonexpansive, \ref{SFB} is an instantiation of \ref{secondOrder} using the map $R$ and the parameter sequence $\frac{1}{\delta}\lambda$ instead of $\lambda$ (see also the proof of Corollary 11 in \cite{BotCsetnek2016}). 

We immediately obtain the following result on the quantitative asymptotic behavior of the associated solution under a finite-dimensionality assumption on $X$, a quantitative variant of parts of Theorem 12 in \cite{BotCsetnek2016}:

\begin{theorem}\label{thm:metaFBSecondOrder}
Let $X$ be a finite-dimensional Hilbert space with dimension $d$ and let $A:X\to 2^X$ be maximally monotone and $B:X\to X$ be $\beta$-cocoercive for $\beta>0$. For $\eta\in (0,2\beta)$ and $\delta:=\frac{4\beta-\eta}{2\beta}$, let $\lambda,\gamma:[0,+\infty)\to (0,+\infty)$ be Lebesgue measurable functions which are locally absolutely continuous and satisfy
\[
\dot{\gamma}(t)\leq 0\leq\dot{\lambda}(t)\text{ and }\frac{\gamma^2(t)}{\lambda(t)}\geq\frac{2(1+\theta)}{\delta}
\]
almost everywhere, for some $\theta>0$. Fix $u_0,v_0\in X$ and let $y\in\mathrm{zer}(A+B)$ with $b,c,d\in\mathbb{N}$ such that $\norm{u_0-y}\leq b$, $\norm{v_0}\leq c$ and $\norm{u_0-T(u_0)}\leq d$. Let $x$ be the unique strong global solution to \ref{SFB}. 

Then $x$ is metastable in the sense that for all $\varepsilon>0$ and $f:\mathbb{N}\to\mathbb{N}$, it holds that 
\[
\exists n\leq \Delta(\min\{\varepsilon,\beta\varepsilon/2\},f)\ \forall s,t\in [n,n+f(n)]\left( \norm{x(s)-x(t)}\leq\varepsilon\text{ and }\norm{x(t)-T(x(t))}\leq\varepsilon\right),
\]
where moreover the bound $\Delta(\varepsilon,f)$ can be explicitly described as in Theorem \ref{thm:secondOrderMeta}.
\end{theorem}

\begin{remark}
Akin to Remark \ref{rem:approximateZeros}, if Theorem \ref{thm:metaFBSecondOrder} is applied with $\min\{\varepsilon,\varepsilon/( \gamma^{-1}+\beta^{-1})\}$ in place of $\varepsilon$, then for any $t\in [n,n+f(n)]$ with the respective $n$, there are $w\in (A+B)(v)$ such that $\norm{x(t)-v},\norm{w}\leq \varepsilon$.
\end{remark}

We now move to the associated strong convergence result for \ref{SFB} under a uniform monotonicity assumption for either of the operators $A$ or $B$, again in similarity to before.

First, we have the following quantitative rendering of the associated limit $\lim_{t\to\infty}B(x(t))= B(x^*)$ for $x^*\in\mathrm{zer}(A+B)$ as established in item (iv) of Theorem 12 in \cite{BotCsetnek2016}:

\begin{proposition}\label{pro:SFB-B-rate}
Let $X$ be a Hilbert space and let $A:X\to 2^X$ be maximally monotone and $B:X\to X$ be $\beta$-cocoercive for $\beta>0$. For $\eta\in (0,2\beta)$ and $\delta:=\frac{4\beta-\eta}{2\beta}$, let $\lambda,\gamma:[0,+\infty)\to (0,+\infty)$ be Lebesgue measurable functions which are locally absolutely continuous and satisfy
\[
\dot{\gamma}(t)\leq 0\leq\dot{\lambda}(t)\text{ and }\frac{\gamma^2(t)}{\lambda(t)}\geq\frac{2(1+\theta)}{\delta}
\]
almost everywhere, for some $\theta>0$. Fix $u_0,v_0\in X$ and let $y\in\mathrm{zer}(A+B)$ with $b,c,d\in\mathbb{N}$ such that $\norm{u_0-y}\leq b$, $\norm{v_0}\leq c$ and $\norm{u_0-T(u_0)}\leq d$. Let $x$ be the unique strong global solution to \ref{SFB}. Then $\lim_{t\to\infty}B(x(t))= B(y)$ and for any $f:\mathbb{N}\to\mathbb{N}$ and $\varepsilon>0$, we have
\[
\exists n\leq \Lambda(\varepsilon^2\eta\beta/3K,f)\ \forall t\in[n,n+f(n)]\left( \norm{B(x(t))-B(y)}<\varepsilon\right)
\]
where $\Lambda$ is defined as in Lemma \ref{lem:SecondIntMeta} and $M:=bc+\frac{\overline{\gamma}}{2}b^2+\beta\overline{\gamma}\underline{\lambda}^{-1}c^2$ and $K:=\sqrt{b^2+2\underline{\gamma}^{-1}M}$.
\end{proposition}
\begin{proof}
Using Lemmas \ref{lem:FB-B-ineq} and \ref{lem:SecondBoundedness}, we again obtain
\[
\eta\beta\norm{B(x(t))-B(y)}^2\leq\left( 1+\frac{\eta}{\beta}\right)\norm{x(t)-y}\norm{T(x(t))-x(t)}\leq \left( 1+\frac{\eta}{\beta}\right)K\norm{T(x(t))-x(t)}
\]
and so, using $\eta\leq 2\beta$, we have
\[
\norm{B(x(t))-B(y)}\leq \sqrt{\frac{3K}{\eta\beta}}\sqrt{\norm{T(x(t))-x(t)}}.
\]
Lemma \ref{lem:SecondIntMeta} now yields the result.
\end{proof}

\begin{remark}
As in the case of the preceding case study on the first-order dynamical system, also here our approach towards the convergence of $B(x(t))$ to $B(y)$ for a solution $y\in\mathrm{zer}(A+B)$ deviates from the argument given in the respective paper that we analyze, that is \cite{BotCsetnek2016} in this case. This is purely for reasons of simplicity, and we could have just as well analyzed the concrete proof of this fact given in \cite{BotCsetnek2016} (see the proof of item (iv) of Theorem 12 in \cite{BotCsetnek2016}), which proceeds via a long chain of inequalities to establish the following key inequality (see the end of p.\ 1433 in \cite{BotCsetnek2016}):
\[
\frac{\beta\underline{\lambda}}{2}\int_0^T\norm{B(x(t))-B(y)}^2\,dt+\frac{1}{\eta}\left(\dot{h}(T)-\dot{h}(0)+\gamma(T)h(T)-\gamma(0)h(0)\right)\leq\frac{1}{\eta}\int^T_0\norm{\dot{x}(t)}\,dt,
\]
where $h(t):=\frac{1}{2}\norm{x(t)-y}^2$. By this, it follows in particular that 
\[
\frac{\beta\underline{\lambda}}{2}\int_0^T\norm{B(x(t))-B(y)}^2\,dt\leq\frac{1}{\eta}\left(\int^T_0\norm{\dot{x}(t)}\,dt+\vert \dot{h}(T)\vert+\vert \dot{h}(0)\vert +\gamma(0)h(0)\right)
\]
which can be used to calculate an explicit bound on the $L^2$-norm of the dynamical system $\norm{B(x(t))-B(y)}^2$, using in particular Lemmas \ref{lem:SecondBoundedness} and \ref{lem:SecondDerBoundedness} to bound $\vert \dot{h}(T)\vert$ and Lemma \ref{lem:SecondIntBoundedness} for a bound on the $L^2$-norm of $\norm{\dot{x}(t)}$. Using the differential inequality 
\[
\frac{d}{dt}\left(\frac{1}{2}\norm{B(x(t))-B(y)}^2\right)\leq\frac{1}{2}\norm{B(x(t))-B(y)}^2+\frac{1}{2\beta}\norm{\dot{x}(t)}^2
\]
established in \cite{BotCsetnek2016} (see the first displayed equation on p.\ 1434 therein), the convergence of $B(x(t))$ to $B(y)$ follows. However, an analysis of this argument also ``only'' yields a rate of metastability as we apply Lemma \ref{lem:AAS2}, which generally only comes furnished with a rate of metastability, to exploit the above inequality with the respective bounds. In particular, it will be presumably of quite similar complexity as the rate presented in the preceding Lemma \ref{pro:FB-B-rate}.
\end{remark}

We now obtain a result on rates of metastability in the context of a uniform monotonicity assumption, being a quantitative rendering of item (v) from Theorem 12 in \cite{BotCsetnek2016}:

\begin{theorem}
Let $X$ be a Hilbert space and let $A:X\to 2^X$ be maximally monotone and $B:X\to X$ be $\beta$-cocoercive for $\beta>0$. For $\eta\in (0,2\beta)$ and $\delta:=\frac{4\beta-\eta}{2\beta}$, let $\lambda,\gamma:[0,+\infty)\to (0,+\infty)$ be Lebesgue measurable functions which are locally absolutely continuous and satisfy
\[
\dot{\gamma}(t)\leq 0\leq\dot{\lambda}(t)\text{ and }\frac{\gamma^2(t)}{\lambda(t)}\geq\frac{2(1+\theta)}{\delta}
\]
almost everywhere, for some $\theta>0$. Fix $u_0,v_0\in X$ and let $y\in\mathrm{zer}(A+B)$ with $b,c,d\in\mathbb{N}$ such that $\norm{u_0-y}\leq b$, $\norm{v_0}\leq c$ and $\norm{u_0-T(u_0)}\leq d$. Let $x$ be the unique strong global solution to \ref{SFB}. 

Then, if either $A$ or $B$ is uniformly monotone with function $\phi$, $x$ converges to $y$ and for any $f:\mathbb{N}\to\mathbb{N}$ and $\varepsilon>0$, we have
\[
\exists n\leq \Theta(\varepsilon,f)\ \forall t\in[n,n+f(n)]\left( \norm{x(t)-y}<\varepsilon\right)
\]
where $\Theta(\varepsilon,f)$ is defined as follows: 
\begin{enumerate}[(i)]
\item If $A$ is uniformly monotone with function $\phi$, then
\[
\Theta(\varepsilon,f):=\Lambda\left(\min\left\{\left(\frac{\phi(\varepsilon/2)}{2K}\right)^2\eta\beta/3K,\frac{\eta\phi(\varepsilon/2)}{2K},\frac{\varepsilon}{2}\right\},f\right).
\]
\item If $B$ is uniformly monotone with function $\phi$, then
\[
\Theta(\varepsilon,f):=\Lambda\left(\left(\frac{\phi(\varepsilon)}{K}\right)^2\eta\beta/3K,f\right).
\]
\end{enumerate}
Here, $\Lambda$ is defined as in Lemma \ref{lem:SecondIntMeta} and $M:=bc+\frac{\overline{\gamma}}{2}b^2+\beta\overline{\gamma}\underline{\lambda}^{-1}c^2$ and $K:=\sqrt{b^2+2\underline{\gamma}^{-1}M}$.
\end{theorem}
\begin{proof}
Items (i) and (ii) follow by a combination of Lemmas \ref{lem:FB-error} as well as \ref{lem:SecondBoundedness} and \ref{lem:SecondIntMeta} together with Proposition \ref{pro:SFB-B-rate}.
\end{proof}

\subsection{(Generalized) gradient flows in Hadamard spaces}\label{subsect:Hadamard}

We now move to the topic of flows and associated semigroups in Hadamard spaces, and in particular to our applications to the classic gradient flow semigroup \cite{Bacak2013,Mayer1998} as well as nonlinear semigroup generated by a nonexpansive mapping \cite{BacakReich2014,Stojkovic2012}. Throughout, if not stated otherwise, let $X$ be a Hadamard space with metric $d$, that is $(X,d)$ is a complete geodesic metric space of nonpositive curvature in the sense of Alexandrov. We refer to \cite{Bacak2014,BridsonHaefliger1999} for further information on these spaces, and here just note that Hadamard spaces are uniquely geodesic. In particular, given $x,y\in X$ and $\lambda\in [0,1]$, we write $(1-\lambda)x\oplus \lambda y$ for the unique point $z$ on the geodesic connecting $x$ and $y$ satisfying $d(x,z)=\lambda d(x,y)$ and $d(y,z)=(1-\lambda)d(x,y)$. 

To motivate our case studies, recall the inclusion
\[
\begin{cases}
-\dot{x}(t)\in A(x(t)),\\
x(0)=x_0,
\end{cases}
\]
where $A$ is a maximally monotone operator on a Hilbert space $X$ and $x_0\in\mathrm{dom}(A)$, as already discussed in Section \ref{sec:motivation} before. This inclusion problem not only generalizes the gradient flow equation $\dot{x}=\nabla \phi(x)$ but also captures the parabolic problem $-\dot{x}\in \partial \phi(x)$ associated with a convex lower-semicontinuous (lsc) function $\phi:X\to (-\infty,+\infty]$. By the well-known Crandall-Liggett theorem \cite{CrandallLiggett1971} (see also the preceding work \cite{BrezisPazy1970,CrandallPazy1969,Komura1967}), the solutions of this problem induce a nonexpansive semigroup $S_t(x_0)$ and this semigroup can moreover be explicitly described by the exponential formula $S_t(x_0):=\lim_{n\to \infty}(J_{t/n})^{(n)}(x_0)$, where $J_{\lambda}(x):=(\mathrm{Id}+\lambda A)^{-1}x$ is the resolvent of of $A$. This formula moreover generally produces a semigroup, even outside the solvability of the above problem. The study of the asymptotic behavior of solutions to the above problem can hence be addressed through general theory for these nonexpansive semigroups.

In particular, the above access to the inclusion problem via the exponential formula allows for an extension to the nonlinear context of Hadamard spaces. Given a Hadamard space $X$ and a convex lsc function $\phi:X\to (-\infty,+\infty]$, a nonlinear version of the above gradient flow semigroup can be defined via
\begin{equation*}
S_t(x):=\lim_{n\to\infty}\left(J_{t/n}\right)^{(n)}(x)
\end{equation*}
for $t\geq 0$ and $x\in \overline{\mathrm{dom}\phi}$ for $\mathrm{dom}\phi:=\{x\in X\mid \phi(x)<+\infty\}$. Here, $J_t(x)$ is now the resolvent of $\phi$ defined by
\begin{equation*}
J_t (x):=\argmin_{y\in X}\left(\phi(y)+\frac{1}{2t} d^2(x,y)\right).
\end{equation*}
for $t>0$ (and via $J_t (x):=x$ for $t=0$). The study of this limit and the nonexpansive semigroup it generates in particular goes back to the work of Mayer \cite{Mayer1998} and was later extended by Ba\v{c}\'ak \cite{Bacak2013}. In particular, this semigroup essentially represents a continuous-time analog of the proximal point algorithm in Hadamard spaces (see \cite{Bacak2013}).

Another instance of the previous inclusion problem in Hilbert spaces is given by considering a monotone operator $A:=\mathrm{Id}-F$ defined via a nonexpansive map $F:X\to X$. A nonlinear variant of the resulting semigroup was subsequently considered by Stojkovic \cite{Stojkovic2012}: For a nonexpansive map $F: X\to X$ on a Hadamard space $X$ and given $x\in X$ as well as $t>0$, let $G_{x,t}:X\to X$ be defined by
\begin{equation*}
G_{x,t}(y):=\frac{1}{1+t} x\oplus \frac{t}{1+t} F(y).
\end{equation*}
It can be easily seen that $G_{x,t}$ is a strict contraction, and so it has a unique fixed point~$R_t(x)$. One calls~$R_t: X\to X$ the resolvent of~$F$. For~$t\in[0,\infty)$ and $x\in X$, the resulting exponential formula
\begin{equation*}
T_t(x):=\lim_{n\to\infty}\left(R_{t/n}\right)^{(n)}(x).
\end{equation*}
similarly induces a nonexpansive semigroup on $X$ (see \cite{Stojkovic2012}). Similar definitions were already studied by Reich and Shafrir \cite{ReichShafrir1990} for coaccretive operators on hyperbolic spaces, and the asymptotic study of the semigroup was in particular further developed by Ba\v{c}\'ak and Reich \cite{BacakReich2014}, extending previous results restricted to the Hilbert ball by Reich \cite{Reich1991}. A practical use of the above semigroup is in particular illustrated in \cite{BacakReich2014}, where it is used to provide a novel access to the Hadamard-space-valued version of the $L^2$-Dirichlet problem on measure spaces with a symmetric Markov kernel as pioneered in the seminal work of Sturm \cite{Sturm2001} (see also \cite{Mayer1998} for related results).

The case studies here should be understood to be indicative, and we expect that other semigroups and flows on nonlinear spaces can be quantitatively approached in a similar manner. In particular we want to highlight the recent work of Chaipunya, Kohsaka and Kumam \cite{ChaipunyaKohsakaKumam2021}, where the notion of a monotone vector field known from Hadamard manifolds is extended to Hadamard spaces. Such fields carry a natural notion of a resolvent, and associated semigroups generated via an exponential formula are already discussed in \cite{ChaipunyaKohsakaKumam2021}. Studying this semigroup could in particular provide a uniform (quantitative) study of both special cases studied in the present work, next to the associated monotone inclusions on Hilbert spaces and Hadamard manifolds, and beyond. We here however decided to only focus on the works \cite{Bacak2013,Mayer1998} and \cite{BacakReich2014,Stojkovic2012} as they provide classic nontrivial examples that illustrate our metric generality without requiring too much new theory, in particular as the asymptotic study of the general semigroups from \cite{ChaipunyaKohsakaKumam2021} is still a bit underdeveloped. We hence leave this for future work, as we do the study of flows in different geometric contexts, such as heat flows on Finsler manifolds as in \cite{OhtaSturm2009} or gradient flows on Wasserstein spaces as in \cite{Ohta2009}.

\subsubsection{The gradient flow of a convex function}

We first study the classical gradient flow over a Hadamard space as in \cite{Bacak2013,Mayer1998}. As before, we assume that $X$ is a Hadamard space with metric $d$ and we further fix a convex lsc $\phi: X\to (-\infty,+\infty]$, for which denote the associated gradient flow semigroup, defined as above, by $S_t(x)$. Throughout, we assume that $\argmin \phi\neq\emptyset$ and write $\mu:=\min_{x\in X}\phi(x)$.

The asymptotic behavior of these semigroups is studied in \cite{Mayer1998,Bacak2013}. In particular, strong convergence towards a minimizer of $\phi$ in proper Hadamard spaces is established in \cite{Mayer1998} (see Theorem 2.41 therein), together with strong convergence under a uniform convexity assumption on $\phi$ but for that over general Hadamard spaces (see Theorem 2.42 therein). Weak convergence towards a minimizer of $\phi$ was subsequently shown in \cite{Bacak2013} (see Theorem 1.5 therein). We now study these results from a quantitative perspective.

The key result that establishes both the continuous Fej\'er monotonicity of the semigroup as well as the associated approximation property is the following inequality:

\begin{lemma}[Lemma 2.37 in \cite{Mayer1998}]\label{lem:MayerLem}
For any $x\in \mathrm{dom}\phi$, $z\in X$ and $t\geq s\geq 0$:
\begin{equation*}
d^2(S_t(x),z)\leq d^2(S_s(x),z)-2(t-s)(\phi(S_t(x))-\phi(z)).
\end{equation*}
\end{lemma}

Immediately, as $S_0(x):=x$, this yields the boundedness of the semigroup.

\begin{lemma}\label{lem:firstSGboundedness}
Let $y\in\argmin \phi$, $x\in \mathrm{dom}\phi$ and let $b\in\mathbb{N}$ with $d(x,y)\leq b$. Then $S_t(x)$ is bounded with $d(S_t(x),y)\leq b$.
\end{lemma}

To capture the corresponding minimization problem in the setup of the previous sections, we set $X_0:=\overline{B}_{b}(y)\cap \overline{\mathrm{dom}\phi}$ and 
\[
F(z)=\begin{cases} \phi(z)-\mu&\text{if }z\in X_0,\\+\infty&\text{otherwise},\end{cases}
\]
for a fixed solution $y\in\argmin \phi$ so that $\mathrm{zer}F=\argmin \phi\cap X_0$ and 
\[
\mathrm{lev}_{\leq\varepsilon} F=\{z\in X_0\mid \phi(z)-\mu\leq \varepsilon\}.
\]

Mayer's inequality in particular also yields a modulus of uniform Fej\'er monotonicity (with $G=H=(\cdot)^2$):

\begin{lemma}\label{lem:firstSGFejer}
Let $y\in\argmin \phi$, $x\in \mathrm{dom}\phi$ and let $b\in\mathbb{N}$ with $d(x,y)\leq b$. Then $S_t(x)$ is uniformly $(G,H)$-Fej\'er monotone w.r.t.\ $F$ and with $G=H=(\cdot)^2$, i.e.\
\begin{equation*}
\forall \varepsilon>0\ \forall n,m\in\mathbb{N}\ \forall z\in\mathrm{lev}_{\leq\chi(\varepsilon,n,m)} F\ \forall s\leq t\in [n,n+m]\left( d^2(S_t(x),z)\leq d^2(S_s(x),z)+\varepsilon\right),
\end{equation*}
where $\chi(\varepsilon,n,m):=\varepsilon/2m$.
\end{lemma}
\begin{proof}
Fix $\varepsilon,n,m$ and let $z\in\mathrm{lev}_{\leq\chi(\varepsilon,n,m)} F$, i.e.\ $z\in X_0$ with $\phi(z)-\mu\leq\chi(\varepsilon,n,m)\leq\varepsilon/2m$. Lemma \ref{lem:MayerLem} yields for $s\leq t\in [n,n+m]$ that
\begin{align*}
d^2(S_t(x),z)&\leq d^2(S_s(x),z)-2(t-s)\big(\phi(S_t(x))-\phi(z)\big)\\
{}&\leq d^2(S_s(x),z)-2m\big(\mu-\phi(z)\big)\\
{}&\leq d^2(S_s(x),z)+\varepsilon.\qedhere
\end{align*}
\end{proof}

Also, we get the following approximation property from Mayer's inequality:

\begin{lemma}\label{lem:firstSGApprox}
Let $y\in\argmin \phi$, $x\in \mathrm{dom}\phi$ and let $b\in\mathbb{N}$ with $d(x,y)\leq b$. Then
\[
\forall \varepsilon>0\ \forall t\geq \varphi(\varepsilon)\left( \phi(S_t(x)) - \mu <\varepsilon\right)
\]
with rate $\varphi(\varepsilon):=\lceil b^2/\varepsilon\rceil$. In particular, $S_t(x)$ has approximate $F$-points with modulus $\varphi$.
\end{lemma}
\begin{proof}
For $t\geq 0$, Lemma \ref{lem:MayerLem} yields
\begin{equation*}
0\leq d^2(S_t(x),y)\leq d^2(x,y)-2t\left(\phi(S_t(x))-\mu\right)\leq b^2-2t \left(\phi(S_t(x))-\mu\right),
\end{equation*}
so that we get
\[
\phi(S_t(x))-\mu\leq \frac{b^2}{2t}\text{ for all }t>0.
\]
In particular, for $t\geq \lceil b^2/\varepsilon\rceil$, we have
\[
\phi(S_t(x))-\mu\leq\frac{b^2}{2\lceil b^2/\varepsilon\rceil}\leq\frac\varepsilon2<\varepsilon.\qedhere
\]
\end{proof}

Combined, we can hence give the following result on the quantitative asymptotic behavior of the semigroup in the context of proper Hadamard spaces, a quantitative version of Theorem 2.41 in \cite{Mayer1998}.

\begin{theorem}
Consider an Hadamard space~$X$ and a convex lsc function $\phi: X\to (-\infty,+\infty]$. Let $y\in\argmin \phi$, $x\in \mathrm{dom}\phi$ and let $b\in\mathbb{N}$ with $d(x,y)\leq b$. Define $X_0:=\overline B_b(y)\cap \overline{\mathrm{dom}\phi}$ and assume that $\gamma$ is a modulus of total boundedness for $X_0$.

Then $S_t(x)$ is metastable in the sense that for all $\varepsilon>0$ and all nondecreasing $f:\mathbb{N}\to\mathbb{N}$, it holds that
\[
\exists n\leq \Delta(\varepsilon,f)\ \forall s,t\in [n,n+f(n)]\left( d(S_t(x),S_s(x))\leq\varepsilon\right),
\]
where moreover the bound $\Delta(\varepsilon,f)$ can be explicitly described by $\Delta(\varepsilon,f):=\Delta(\gamma(\varepsilon/\sqrt{12})+1,\varepsilon,f)+1$ with $\Delta(0,\varepsilon,f):=0$ and
\begin{equation*}
\Delta(j+1,\varepsilon,f):=\left\lceil 24b^2\big(f(\Delta(j,\varepsilon,f)+1)+1\big)/\varepsilon^2\right\rceil.
\end{equation*}
for $j\geq 1$.
\end{theorem}
\begin{proof}
The result immediately follows by instantiating Theorem \ref{thm:generalFejer-rates} over the given $X_0$ and simplifying the resulting bound (recall that $S_t(x)\in X_0$ for all $t\geq 0$ by Lemma \ref{lem:boundedness} and since $S_t(x)\in\overline{\mathrm{dom}\phi}$). In particular, we set $G:=H:=(\cdot)^2$ as well as $h(\varepsilon):=\varepsilon^2$ and $g(\varepsilon):=\sqrt{\varepsilon}$. That $S_t(x)$ is uniformly $(G,H)$-Fej\'er monotone w.r.t.\ $F$ follows by Lemma \ref{lem:firstSGFejer}, which in particular yields that $\chi(\varepsilon,n,m):=\varepsilon/2m$ is an associated modulus. As we do not deal with errors, it suffices that $S_t(x)$ has the approximate $F$-point property, which follows by Lemma \ref{lem:firstSGApprox} with a respective bound $\varphi(\varepsilon):=\lceil b^2/\varepsilon\rceil$ that is immediately monotone under our assumptions. 
\end{proof}

If we assume that $\phi$ is uniformly continuous with a respective modulus, then this continuity property transfers to $F$, so that using Theorem \ref{thm:generalFejer-rates} we can moreover provide a bound that guarantees that $\phi(S_t(x))-\mu\leq\varepsilon$ holds along the region of metastability. We do not spell this out here any further.

Further, we can give the following result on rates of convergence under regularity assumptions in general Hadamard spaces.

\begin{theorem}\label{thm:bacak-reg}
Consider an Hadamard space~$X$ and a convex lsc function $\phi: X\to (-\infty,+\infty]$. Let $y\in\argmin \phi$, $x\in \mathrm{dom}\phi$ and let $b\in\mathbb{N}$ with $d(x,y)\leq b$. Assume that $F$ has a modulus of regularity $\tau$ w.r.t.\ $\overline{B}_b(y)$.

Then $S_t(x)$ satisfies
\[
\forall\varepsilon>0\ \forall t\geq \rho(\varepsilon)\left( \dist(S_t(x),\argmin \phi)<\varepsilon\right)
\]
and further $S_t(x)$ is Cauchy with
\[        
\forall\varepsilon>0\ \forall t,s\geq \rho(\varepsilon/2)\left(d(S_t(x),S_s(x))<\varepsilon\right),
\]
where $\rho$ is defined by $\rho(\varepsilon):=\left\lceil b^2/\tau(\varepsilon)\right\rceil+1$. In particular, $S_t(x)$ converges to a minimizer of $\phi$ with that rate.
\end{theorem}
\begin{proof}
The result immediately follows by instantiating Theorem \ref{thm:met-reg-rates} and simplifying the resulting bound. In particular, we set $G:=H:=(\cdot)^2$ as well as $h(\varepsilon):=\varepsilon^2$ and $g(\varepsilon):=\sqrt{\varepsilon}$. That $S_t(x)$ is $(G,H)$-Fej\'er monotone w.r.t.\ $F$ follows from the fact that $S_t(x)$ is a nonexpansive semigroup (recall also Lemma 2.37 in \cite{Mayer1998} and Lemma \ref{lem:firstSGFejer}). As we do not deal with errors, it suffices that $S_t(x)$ has the approximate $F$-point property which, with a respective bound $\varphi(\varepsilon):=\lceil b^2/\varepsilon\rceil$ follows using Lemma \ref{lem:firstSGApprox} as before. 
\end{proof}

The above in particular applies to functions $\phi$ with weak sharp minima (recall Example \ref{ex:functions}), and hence in particular yields a the strong convergence of the semigroup, together with an associated rate, when the function $\phi$ is uniformly convex. While derived here using our general methodology, such a rate of convergence for the uniformly convex case can in fact be easily read off from the convergence proof tailored to this special case in~\cite{Bacak2013}. Concretely, a more direct inspection of the proof given in \cite{Bacak2013} gives a rate of convergence $2b^2/\tau(\varepsilon)$, where $\tau$ is the modulus of uniform convexity, which is virtually the same as the rate $\left\lceil 4b^2/\tau(\varepsilon/2)\right\rceil+1$ derived from the above result.

\subsubsection{A class of nonlinear semigroups for a nonexpansive mapping}

We now move on to the semigroup generated by a nonexpansive map as in \cite{Stojkovic2012,BacakReich2014}. As before, let $X$ be a Hadamard space with metric $d$ and now fix a nonexpansive map $F: X\to X$. Define the associated function $G_{x,t}(y)$ and is fixed point selection $R_t(x)$ as before, and denote the semigroup associated with $R_t(x)$ by $T_t(x)$. Throughout, we assume that $\mathrm{Fix}F\neq\emptyset$.

The asymptotic behavior of these semigroups is studied in \cite{BacakReich2014,Stojkovic2012}. In particular, it is established in \cite{BacakReich2014} (see Theorem 1.6 therein) that $T_t(x)$ converges (in general weakly) to a fixed point of $F$, for any starting point $x\in X$. We now study this result from a quantitative perspective.

We begin with the boundedness of the semigroup.

\begin{lemma}\label{lem:secondSGboundedness}
Let $y\in\mathrm{Fix}F$, $x\in X$ and let $b\in\mathbb{N}$ with $d(x,y)\leq b$. Then $T_t(x)$ is bounded with $d(T_t(x),y)\leq b$.
\end{lemma}
\begin{proof}
Note that since $y\in\mathrm{Fix}F$, we also have $T_t(y):= y$. Further, $T_t(x)$ is a nonexpansive semigroup (see Theorem 3.10 in \cite{Stojkovic2012} for both). Combined, we get
\[
d(T_t(x),y)=d(T_t(x),T_t(y))\leq d(x,y)\leq b.\qedhere
\]
\end{proof}

To capture the corresponding fixed point problem in the setup of the previous sections, we set $X_0:=\overline{B}_{b}(y)$ and 
\[
\bar F(z)=\begin{cases} d(z,F(z))&\text{if }z\in X_0,\\+\infty&\text{otherwise},\end{cases}
\]
for a fixed solution $y\in\argmin f$ so that $\mathrm{zer}\bar F=\mathrm{Fix} F\cap X_0$ and 
\[
\mathrm{lev}_{\leq\varepsilon} \bar F=\{z\in X_0\mid d(z,F(z))\leq \varepsilon\}.
\]

We now first derive the following approximation property:

\begin{lemma}\label{lem:Hadamard-liminf}
Let $y\in\mathrm{Fix}F$, $x\in X$ and let $b\in\mathbb{N}$ with $d(x,y)\leq b$. Then $T_t(x)$ has appproximate $\bar F$-points with modulus $\varphi$, i.e.
\[
\forall \varepsilon>0\ \exists t\leq\varphi(\varepsilon)\left( d(T_t(x),F(T_t(x)))<\varepsilon\right),
\]
where we have $\varphi(\varepsilon):=\ceil*{(4b/\varepsilon)\cdot e^{4b/\varepsilon}}$.
\end{lemma}
\begin{proof}
As $F(y)=(y)$ entails $T_t(y)=(y)$, the claim is immediate for~$x=y$. In the following, we may thus assume~$b>0$ and hence $h:=4b/\varepsilon>0$. Due to the semigroup property and since $T_t$ is nonexpansive, we have
\begin{equation*}
    d(T_{t+h}(x),T_t(x))=d(T_t(T_h(x)),T_t(x))\leq d(T_h(x),x).
\end{equation*}
As in \cite{BacakReich2014} (see the fourth equation on page~197 therein), we can hence derive
\begin{equation*}
    d(T_t(x),F(T_t(x)))\leq\frac{d(T_h(x),x)}h+\frac{(e^h-1)\big(d(T_tx,FT_tx)-d(T_{t+h}(x),F(T_{t+h}(x)))\big)}h.
\end{equation*}
In view of $T_h(y)=y$ and since $T_h$ is nonexpansive, we have
    \begin{equation*}
        \frac{d(T_h(x),x)}h\leq\frac{d(T_h(x),T_h(y))+d(y,x)}h\leq\frac{2b}h=\frac\varepsilon2.
    \end{equation*}
So to complete the proof, it is enough to show that there is a $t\leq\varphi(\varepsilon)$ with
\begin{equation*}
    \frac{(e^h-1)\big(d(T_t(x),F(T_t(x)))-d(T_{t+h}(x),F(T_{t+h}(x)))\big)}h<\frac\varepsilon2.
\end{equation*}
Towards a contradiction, we assume that all $t\leq\varphi(\varepsilon)$ validate
\begin{equation*}
    d(T_t(x),F(T_t(x)))-d(T_{t+h}(x),F(T_{t+h}(x)))\geq\frac\varepsilon2\cdot\frac h{e^h-1}=\frac{2b}{e^h-1}.
\end{equation*}
Observe that we have
\begin{equation*}
    (N-1) h\leq\varphi(\varepsilon)\text{ for } N:=\left\lceil e^h\right\rceil.
\end{equation*}
Via a telescoping sum, we obtain
\begin{equation*}
    d(T_0(x),F(T_0(x)))-d(T_{N\cdot h}(x),F(T_{N\cdot h}(x)))\geq N\cdot\frac{2b}{e^h-1}>2b.
\end{equation*}
But we also have
\begin{equation*}
   d(T_0(x),F(T_0(x)))=d(x,F(x))\leq d(x,y)+d(F(y),F(x))\leq 2b,
\end{equation*}
which yields the desired contradiction.
\end{proof}

Our next goal is (a quantitative version of) Fej\'er monotonicity for the semigroup $T_t(x)$. In a preparatory lemma, we first provide a quantitative version of the fact that fixed points of $F$ are also fixed points of the semigroup.

\begin{lemma}\label{lem:fixQuant}
For any $x\in X$ and $t,\delta>0$, if $d(x,F(x))\leq\delta$, then $d(x,T_t(x))\leq\delta\cdot\frac{e^{2t}-1}2$.
\end{lemma}
\begin{proof}
Write $x_k=(R_\lambda)^{(k)}(x)$. Assuming $d(x,F(x))\leq\delta$, we shall establish
\begin{equation*}
d(x_n,x_{n+1})\leq \lambda(1+2\lambda)^n\cdot\delta
\end{equation*}
Lemma 3.4 in \cite{Stojkovic2012} (see also Lemma 2.2 in \cite{BacakReich2014}) yields that
\[
d(z,R_\lambda (z))\leq \lambda d(z,F(z))
\]
for all $z\in X$ and $\lambda>0$, so that the above estimate reduces to
\begin{equation*}
d(x_n,F(x_n))\leq(1+2\lambda)^n\cdot\delta.
\end{equation*}
For~$n=0$, this holds by assumption. Using that $F$ is nonexpansive, we inductively get
\begin{align*}
d(x_{n+1},F(x_{n+1}))&\leq d(x_{n+1},x_n)+d(x_n,F(x_n))+d(F(x_n),F(x_{n+1}))\\
&\leq 2\lambda(1+2\lambda)^n\cdot\delta+(1+2\lambda)^n\cdot\delta=(1+2\lambda)^{n+1}\cdot\delta.
\end{align*}
Let us now infer
\begin{equation*}
d(x,x_n)\leq\sum_{k=0}^{n-1}d(x_k,x_{k+1})\leq \lambda\cdot\delta\cdot\sum_{k=0}^{n-1}(1+2\lambda)^k=\delta\cdot\frac{(1+2\lambda)^n-1}2.
\end{equation*}
As $T_t(x)$ is defined as the limit of $(R_{t/n})^{(n)}(x)$ for~$n\to\infty$, we specialize to $\lambda=t/n$ and obtain
\begin{equation*}
d(x,T_t(x))=\lim_{n\to\infty}d\left(x,(R_{t/n})^{(n)}(x)\right)\leq\lim_{n\to\infty}\frac\delta2\cdot\left(\left(1+\frac{2t}n\right)^n-1\right)=\delta\cdot\frac{e^{2t}-1}2,
\end{equation*}
just like the lemma claims.
\end{proof}

We can now derive the following quantitative result on the Fej\'er monotonicity of $T_t(x)$.

\begin{lemma}\label{lem:Hadamard-Fejer}
Let $y\in\mathrm{Fix}F$, $x\in X$ and let $b\in\mathbb{N}$ with $d(x,y)\leq b$. Then $T_t(x)$ is uniformly $(G,H)$-Fej\'er monotone w.r.t.\ $\bar F$ and with $G=H=\mathrm{Id}$, i.e.\
\begin{equation*}
\forall \varepsilon>0\ \forall n,m\in\mathbb{N}\ \forall z\in\mathrm{lev}_{\leq\chi(\varepsilon,n,m)} \bar F\ \forall s\leq t\in [n,n+m]\left( d(T_t(x),z)\leq d(T_s(x),z)+\varepsilon\right),
\end{equation*}
where $\chi(\varepsilon,n,m):=2\varepsilon/(e^{2m}-1)$.
\end{lemma}
\begin{proof}
Let $\varepsilon,n,m$ be given and let $z\in\mathrm{lev}_{\leq\chi(\varepsilon,n,m)}$. Using Lemma \ref{lem:fixQuant} with $\delta=2\varepsilon/(e^{2m}-1)$, we get $d(T_{t-s}(z),z)\leq \varepsilon$. Using that $T_t(x)$ is a nonexpansive semigroup, we get
\begin{equation*}
d(T_t(x),z)\leq d(T_{t-s}(T_s(x)),T_{t-s}(z))+d(T_{t-s}(z),z)\leq d(T_s(x),z)+\varepsilon.\qedhere
\end{equation*}
\end{proof}

We can now give the following result on the quantitative asymptotic behavior of the semigroup in the context of proper Hadamard spaces, a quantitative version of (parts of) Theorem 1.6 in \cite{BacakReich2014}.

\begin{theorem}
Consider an Hadamard space~$X$ and a nonexpansive mapping $F:X\to X$. Let $y\in\mathrm{Fix}F$, $x\in X$ and let $b\in\mathbb{N}$ with $d(x,y)\leq b$. Define $X_0:=\overline B_b(y)$ and assume that $\gamma$ is a modulus of total boundedness for $X_0$. 

Then $T_t(x)$ is metastable in the sense that for all $\varepsilon>0$ and all nondecreasing $f:\mathbb{N}\to\mathbb{N}$, it holds that
\[
\exists n\leq \Delta(\varepsilon,f)\ \forall s,t\in [n,n+f(n)]\left( d(T_t(x),T_s(x))\leq\varepsilon\right),
\]
where moreover the bound $\Delta(\varepsilon,f)$ can be explicitly described by $\Delta(\varepsilon,f):=\varphi\left(\widehat\varepsilon_P\right)$ for $P:=\gamma\left(\varepsilon/6\right)+1$ with $\widehat\varepsilon_1:=\chi_f(\varepsilon/6,0)$ and
\begin{equation*}
\widehat\varepsilon_j:=\chi_f\left(\frac\varepsilon6,\varphi\left(\widehat\varepsilon_{j-1}\right)\right)
\end{equation*}
for $j\geq 1$, where
\begin{equation*}
\varphi(\varepsilon):=\left\lceil\frac{4r}\varepsilon\cdot e^{4r/\varepsilon}\right\rceil\text{ and }\chi_f(\varepsilon,n):=\frac{2\varepsilon}{e^{2(f(n+1)+1)}-1}.
\end{equation*}
\end{theorem}
\begin{proof}
The result immediately follows by instantiating Theorem \ref{thm:generalFejer-rates} over the given $X_0$ and simplifying the resulting bound (recall that $T_t(x)\in X_0$ for all $t\geq 0$ by Lemma \ref{lem:boundedness}). In particular, we set $G:=H:=\mathrm{Id}$ as well as $h:=g:=\mathrm{Id}$. That $T_t(x)$ is uniformly $(G,H)$-Fej\'er monotone w.r.t.\ $\bar F$ follows by Lemma \ref{lem:Hadamard-Fejer}, which in particular yields that $\chi(\varepsilon,n,m):=2\varepsilon/(e^{2m}-1)$ is an associated modulus. As we do not deal with errors, it suffices that $T_t(x)$ has the approximate $\bar F$-point property, which follows by Lemma \ref{lem:Hadamard-liminf} with a respective bound $\varphi(\varepsilon):=\ceil*{(4b/\varepsilon)\cdot e^{4b/\varepsilon}}$ that is immediately monotone under our assumptions. 
\end{proof}

We also derive a quantitative result under regularity and without total boundedness.

\begin{theorem}
Consider an Hadamard space~$X$ and a nonexpansive mapping $F:X\to X$. Let $y\in\mathrm{Fix}F$, $x\in X$ and let $b\in\mathbb{N}$ with $d(x,y)\leq b$. Assume that $\bar F$ has a modulus of regularity $\tau$ w.r.t.\ $\overline{B}_b(y)$.

Then $T_t(x)$ satisfies
\[
\forall\varepsilon>0\ \forall t\geq \rho(\varepsilon)\left( \dist(T_t(x),\mathrm{Fix} F)<\varepsilon\right)
\]
and further $T_t(x)$ is Cauchy with
\[        
\forall\varepsilon>0\ \forall t,s\geq \rho(\varepsilon/2)\left(d(T_t(x),T_s(x))<\varepsilon\right),
\]
where $\rho$ is defined by $\rho(\varepsilon):=\left\lceil (4b/\tau(\varepsilon))\cdot e^{4b/\tau(\varepsilon)}\right\rceil+1$. wIn particular, $T_t(x)$ converges to a fixed point of $F$ with that rate.
\end{theorem}
\begin{proof}
The result immediately follows by instantiating Theorem \ref{thm:met-reg-rates} and simplifying the resulting bound. In particular, we set $G:=H:=\mathrm{Id}$ as well as $h:=g:=\mathrm{Id}$. That $T_t(x)$ is $(G,H)$-Fej\'er monotone w.r.t.\ $\bar F$ follows from the fact that $T_t(x)$ is a nonexpansive semigroup (recall also Lemma \ref{lem:Hadamard-Fejer}). As we do not deal with errors, it suffices that $T_t(x)$ has the approximate $\bar F$-point property which, with a respective bound $\varphi(\varepsilon):=\ceil*{(4b/\varepsilon)\cdot e^{4b/\varepsilon}}$ follows using Lemma \ref{lem:Hadamard-liminf} as before. 
\end{proof}

\noindent{\textbf{Acknowledgments:}} We want to thank Ulrich Kohlenbach and Thomas Powell for helpful discussions and remarks on the topic of this paper. Also, we want to thank Radu Ioan Bo\c{t} for suggesting the works \cite{BotCsetnek2016,BotCsetnek2017} as interesting case studies for proof mining in the context of dynamical systems. Lastly, the first author wants to thank Lauren\c{t}iu Leu\c{s}tean for bringing the work \cite{ReichShafrir1990} to his attention in early 2023, which lead to an investigation of \cite{BacakReich2014} in the present work.\\

\noindent{\textbf{Funding:}} The work of Freund has been funded by the Deutsche Forschungsgemeinschaft (DFG, German Research Foundation) -- Project number 460597863.

\bibliographystyle{plain}
\bibliography{ref}

\end{document}